 \documentclass[11pt]{article}
\usepackage{amssymb, amsthm, amsmath, amscd}
\newtheorem{thm}{Theorem}[section]
\newtheorem{lem}[thm]{Lemma}
\newtheorem{prop}[thm]{Proposition}
\newtheorem{cor}[thm]{Corollary}
\newtheorem{NN}[thm]{}
\theoremstyle{definition}\newtheorem{df}[thm]{Definition}
\theoremstyle{definition}\newtheorem{rem}[thm]{Remark}
\theoremstyle{definition}

\renewcommand{\phi}{\varphi}

\newcommand{\N}{\mathbb{N}}
\newcommand{\Z}{\mathbb{Z}}
\newcommand{\Q}{\mathbb{Q}}

\newcommand{\C}{\mathbb{C}}
\newcommand{\T}{\mathbb{T}}

\DeclareMathOperator{\Aut}{Aut}

\newcommand{\Aff}{\operatorname{Aff}}

\newcommand{\id}{\operatorname{id}}

\newcommand{\morp}{contractive completely positive linear map}

\newcommand{\hm}{homomorphism}
\newcommand{\dt}{\delta}
\newcommand{\ep}{\epsilon}
\newcommand{\F}{{\cal F}}

\newcommand{\andeqn}{\,\,\,{\rm and}\,\,\,}
\newcommand{\rforal}{\,\,\,{\rm for\,\,\,all}\,\,\,}
\newcommand{\CA}{$C^*$-algebra}
\newcommand{\CAs}{$C^*$-algebras}
\newcommand{\SCA}{$C^*$-subalgebra}

\newcommand{\af}{{\alpha}}
\newcommand{\bt}{{\beta}}

\newcommand{\beq}{\begin{eqnarray}}
\newcommand{\eneq}{\end{eqnarray}}
\newcommand{\tforal}{\,\,\,\text{for\,\,\,all}\,\,\,}
\newcommand{\tand}{\,\,\,\text{and}\,\,\,}

\title{Kishimoto's Conjugacy Theorems in simple $C^*$-algebras
of tracial rank one}
\author{Huaxin Lin }
\date{}

\begin{document}

\maketitle

\begin{abstract}
Let $A$ be a unital separable simple amenable $C^*$-algebra with finite tracial rank which satisfies
the Universal Coefficient Theorem (UCT). Suppose $\alpha$ and $\beta$ are two automorphisms with the Rokhlin property
that { induce the same action on the $K$-theoretical data of $A$.}
We show that $\alpha$ and $\beta$ are strongly outer conjugate and uniformly approximately conjugate, that is,
there exists a sequence of unitaries $\{u_n\}\subset A$ and a sequence
of strongly asymptotically inner automorphisms $\sigma_n$ such that
$$
\alpha={\rm Ad}\, u_n\circ \sigma_n\circ \beta\circ \sigma_n^{-1}\,\,\, {\rm and}\,\,\, \lim_{n\to\infty}\|u_n-1\|=0,
$$
and that the converse holds. { We then give a $K$-theoretic description as to exactly when $\alpha$ and $\beta$ are outer conjugate, at least under a mild restriction.  Moreover,  we show that given any $K$-theoretical data, there exists an automorphism
$\alpha$ with the Rokhlin property which has the same $K$-theoretical data. }

\end{abstract}

\section{Introduction}
Let $A$ be a unital separable simple \CA\, and let $\Aut(A)$ be the automorphism group of $A.$ Kishimoto  (\cite{K1} and \cite{K2}) studied those
automorphisms with the Rokhlin property (see definition \ref{dRK} below).
Suppose $\af,\bt\in\Aut(A)$ have the
Rokhlin property and are {asymptotically unitarily equivalent}.
Kishimoto showed, under the assumption that $A$ is a unital simple $A\T$-algebra of real rank zero, $\af$ and $\bt$ are outer conjugate, i.e., there exists a unitary $u\in A$ and $\sigma\in \Aut(A)$
such that
\beq\label{int-1}
\af={\rm Ad}\, u\circ \sigma\circ \bt \circ \sigma^{-1}.
\eneq
Kishimoto also showed that automorphisms with the Rokhlin property
are abundant. Kishimoto's work actually revealed much more.  Matui (\cite{M}) found out, using a homotopy lemma, that the above-mentioned outer conjugacy result of Kishimoto holds for all unital separable amenable \CAs\, with tracial rank zero which satisfy the Universal
Coefficient Theorem (UCT). Furthermore, combining Kishimoto's argument with the condition of ${\cal Z}$-stability, {Sato recently proved that}, given any $\af\in \Aut(A),$ where
$A$ is a unital separable simple \CA\, with tracial rank zero,
there is $\bt\in \Aut(A)$ with the Rokhlin property such that
$\af$ and $\bt$ are asymptotically unitarily equivalent. This further shows that automorphisms with the Rokhlin property are abundant. We { would } also like to mention
that Kishimoto's result  holds for the case that $A$ is a purely infinite simple \CA\,
(see \cite{Nak}).

{ We are lead to consider the following questions that this paper attempts to answer:} 1) Could Kishimoto's result work for more general stably finite
simple \CA s? 2) Is it possible to remove the unitary $u$ in (\ref{int-1})? Or, is it possible
to obtain an approximate conjugacy result? 3) Is there a $KK$-theoretic description
of outer conjugacy (or approximate conjugacy)?

{To answer the first question}, we will consider the case that $A$ is a unital separable  amenable  simple \CA\, with  finite tracial rank.
We further assume that $A$ satisfies the UCT.  It has been shown in \cite{LnlocAH} that $A$ actually has tracial rank at most one. These \CA s include
all unital simple $A\T$-algebras, as well as algebras that may not have real rank zero.  By a classification result, these \CA s are precisely those unital simple $AH$-algebras with slow dimension growth.
Kishimoto's argument works very well in the case that $A$ has real rank zero.
However, there is a significant technical problem in
generalizing Kishimoto's method to  \CA s with infinite exponential length, which is the case for general unital simple AH-algebra with slow dimension growth.  Kishimoto's method
relies on something called the Basic Homotopy Lemma (\cite{BEEK}) which has a controlled length
for the homotopy. The control of the length is essential in the argument and {plays a very important  role}. The Basic Homotopy Lemma has been extended to a much more general
situation (see \cite{Lnmem} and \cite{Lnhomt1}) which includes the case when the \CA s are allowed to
have tracial rank one instead of zero. However, for that extension, one loses control of the length
of the homotopy. This cannot be improved since the exponential length is infinite in unital simple
AH-algebras whenever it has real rank other than zero.
Nevertheless it has been
demonstrated that it is possible to control the exponential length of some special unitaries, namely
those unitaries which are not only in the path-connected component of the identity but are also in the closure of
the commutator subgroup.  In this paper, we first present a very special homotopy lemma
for some special unitaries which are in the closure of the commutator subgroup. By proving an existence
type of result, we will show that the special homotopy lemma mentioned above would be sufficient
if one is willing to pay the toll for a narrow passage into Kishimoto's method. We will prove that
Kishimoto's conjugacy theorem holds for unital separable amenable simple \CA s with finite tracial rank satisfying the UCT.

For the second question, upon examining Kishimoto's  proof,  one realizes that, when $\af$ and $\bt$ are asymptotically unitarily equivalent and satisfy the Rokhlin property, $\af$ and $\bt$ are outer conjugate in a stronger sense: one can choose
$\sigma$ to be strongly asymptotically inner.  Furthermore, by examining his proof further,
one notices that Kishimoto's  argument
provided a way to remove $u,$ or at least make $u$ as close to the identity as possible, so that if certain obstacles disappeared,
a conclusion of being approximately conjugate is possible.
 We identify the obstacle as the quotient group $K_1(A)/H_1(K_0(A), K_1(A))$ (see \ref{dH1} below).
When $H_1(K_0(A), K_1(A))=K_1(A),$  that  $\af$ and $\bt$ are (strongly) asymptotically unitarily equivalent
and satisfy the Rokhlin property implies that
$\af$ and $\bt$ are not only (strongly) outer conjugate but are also uniformly approximately conjugate.  This is our answer to the second question (see \ref{CM2}).

Since Kishimoto was primarily concerned with approximately inner automorphisms at the time,
he did not provide a $K$-theoretic description of outer conjugacy.  It should be noted that
his outer conjugacy is not equivalent to asymptotic unitary equivalence even in the class of automorphisms with the Rokhlin property. However, strong outer conjugacy
is equivalent to asymptotic unitary equivalence which, by recent developments in Elliott's classification program, can be characterized by $K$-theoretical data (\ref{CM1}).  In the case that $K_1(A)=H_1(K_0(A), K_1(A)),$
this also gives a $K$-theoretical description of strong outer conjugacy and uniformly approximate conjugacy
(see \ref{CC1}). To have a better $K$-theoretical description
of outer conjugacy for automorphisms with the Rokhlin property,  we impose some mild restriction
(for example, we assume that $K_1(A)/{\rm Tor}(K_1(A))$ is  free). {We then introduce a $K$-theory related
group on invariant sets. We show, under the restriction, that $\af$ and $\bt$ are outer conjugate if and only if
their $K$-theoretical invariant ${\tilde{\mathfrak{K}}}(\af)$ and ${\tilde{\mathfrak{K}}}(\bt)$ are conjugate (see
\ref{LT}).}

Briefly, this paper is organized as follows. The next section is a list of definitions and notations that will be used
in this paper. In Section 3 we present a version of the  Basic Homotopy Lemma which controls the length of the path.
Section 4 uses results of Section 3 to present a proof of Kishimoto's conjugacy theorem
for unital separable amenable simple \CA s with tracial rank at most one which satisfy the UCT.
Section 5 serves as a preparation for the proof of Section 6. In Section 6, combining
Sato's refinement of Kishimoto's argument, we show that, given any automorphism $\af$ on a unital separable
amenable simple \CA\, $A$ with finite tracial rank satisfying the UCT, there exists an automorphism ${\tilde \af}$ on $A$
which has the Rokhlin property and which is strongly asymptotically unitarily equivalent to $\af.$  In Section 7,
we give the $K$-theoretic description of (strongly) outer conjugacy and uniformly approximate
conjugacy for automorphisms with the Rokhlin property on $A$ mentioned above.

{\bf Acknowledgements} This work is partially supported by
The Research Center for Operator Algebras in East China Normal University and
a grant of the National Science Foundation (NSF).

\section{Preliminaries}

\begin{df}
{\rm
Let $A$ be a unital \CA. Denote by $\Aut(A)$ the group of automorphisms
of $A.$
Let $u\in A$ be a unitary.
Define ${\rm Ad}\, u(a)=u^*au$ for all $a\in A.$ ${\rm Ad}\, u$ is  an automorphisms. These automorphisms are called inner
automorphisms.
}
\end{df}
\begin{df}\label{dRK}
{\rm
Let $\af\in \Aut(A).$ Following  Kishimoto (Definition 4.1 of \cite{K1}), we say that $\af$ has the Rokhlin property if for any integer $k\in \N,$ any $\ep>0$ and any finite subset ${\cal F}\subset A,$ there exists a family of projections
$\{e_{1,0}, e_{1,1},...,e_{1,k-1},e_{2,0}, e_{2,1},...,e_{2,k}\}$ in $A$ such that
\beq\label{drk1}
\sum_{j=0}^{k-1}e_{1,j}+\sum_{j=0}^ke_{2,j}&=&1,\\
\|\af(e_{i,j})-e_{i,j+1}\|&<&\ep\,\,\,j=0,1,...,k+i-3, \,i=1,2,\\
\|[e_{i,j}, x]\|&<&\ep\tforal x\in {\cal F}, \,\,\,j=0,1,...,k+i-2,\, i=1,2.
\eneq
As pointed out by Kishimoto, one has
\beq\label{drk-2}
\|\af(e_{1,k-1}+e_{2,k})-(e_{1,0}+e_{2,0})\|<(2k-1)\ep.
\eneq

The set of those elements in $\Aut(A)$ with the Rokhlin property
will be denoted by $\Aut_R(A).$
}

\end{df}

\begin{df}\label{dasym}
{\rm  We say $\af$ and $\bt$ are {\it asymptotically unitarily equivalent}, if there exists a continuous
path of unitaries $\{u(t): t\in [0,\infty)\}\subset A$ such that
\beq\label{dasym-1}
\af(x)=\lim_{t\to\infty} u(t)^*\bt(x)u(t)\tforal x\in A.
\eneq
We say that $\af$ and $\bt$ are {\it strongly asymptotically unitarily equivalent} if, in addition, $u(1)=1_A$ in (\ref{dasym-1}).

We say that $\af$ is (strongly) asymptotically inner if $\af$ is (strongly) asymptotically unitarily equivalent to ${\rm id}_A.$
}
\end{df}

\begin{df}\label{dconj}
{\rm
Let $\af$ and $\bt$ be two automorphisms on $A.$
Automorphisms $\af$ and $\bt$ are outer conjugate if there exists an automorphism
$\sigma\in \Aut(A)$ and a unitary $u\in A$ such that
$$
\af={\rm Ad}\, u\circ \sigma^{-1}\circ \bt\circ \sigma.
$$
Outer conjugacy is an equivalent relation.
Denote by $\Aut_R(A)/\sim_{cc}$ the outer conjugate classes of automorphisms
in $\Aut_R(A).$

We say $\af$ and $\bt$ are {\it strongly} outer conjugate if they are outer
conjugate with $\sigma$ being strongly asymptotically inner.  Denote by $\Aut_R(A)/\sim_{scc}$ the strongly outer conjugacy classes of automorphisms in $\Aut_R(A).$

}
\end{df}

\begin{df}\label{Duccc}
{\rm
Two automorphisms $\af, \bt\in \Aut(A)$ are said to be
outer conjugate and uniformly approximately conjugate if
there exists a sequence of a unitaries $u_n\in U(A)$ and a sequence
of automorphisms $\sigma_n\in \Aut(A)$ such that
\beq\label{Duccc-1}
\af={\rm Ad}\, u_n\circ \sigma_n^{-1}\circ \bt\circ \sigma_n\andeqn
\lim_{n\to\infty}\|u_n-1\|=0.
\eneq

Note, in this case,
\beq\label{Duccc-1+}
\lim_{n\to\infty}\|\af-\sigma_n^{-1}\circ \bt \circ \sigma_n\|=0.
\eneq

Two automorphisms $\af,\,\bt\in \Aut(A)$ are said be
{\it strongly} outer conjugate and  uniformly  approximately conjugate
if there exists a sequence of unitaries $u_n\in U(A)$ and
a sequence of strongly asymptotically inner automorphisms $\sigma_n\in \Aut(A)$ such
that
\beq\label{Duccc-2}
\af={\rm Ad}\, u_n\circ \sigma_n^{-1}\circ \bt \circ \sigma_n\andeqn
\lim_{n\to\infty}\|u_n-1\|=0.
\eneq

Denote by $\Aut_R(A)/\sim_{aucc}$ the outer conjugate and uniformly
approximately conjugate classes of automorphisms in $\Aut_R(A),$ and denote
by $\Aut_R(A)/\sim_{saucc}$ the strongly outer conjugate and uniformly approximately conjugate classes of automorphisms.

}

 \end{df}

\begin{df}\label{dcu}
{\rm
Let $A$ be a unital \CA. Denote by $T(A)$ the tracial state space of $A.$
Denote by $\rho_A: K_0(A)\to \Aff (T(A))$ the \hm\, defined by
$\rho_A([p])(\tau)=\tau(p)$ for all projections $p\in M_{\infty}(A)$ and for all
$\tau\in T(A).$

Denote by $U(A)$ the unitary group and $U_0(A)$ the normal subgroup
of $U(A)$ which consists of those unitaries in the path connected component containing the identity.
Denote by $DU(A)$ the commutator subgroup of $U(A)$ and by $CU(A)$ the closure of $DU(A).$
If $u\in U(A),$ denote by ${\bar u}$ the image of $u$ in  $U(A)/CU(A).$
Let $u\in U_0(A)$ and choose a piece-wise smooth and continuous path $\{u(t):t\in [0,1]\}$ of unitaries in $A$ with
$u(0)=u$ and $u(1)=1_A.$ Denote by
$$
\Delta(u)(\tau)={1\over{2\pi i}}\int_0^1\tau({du(t)\over{dt}}u(t)^*))dt\tforal \tau\in T(A).
$$
the de la Harpe and Skandalis determinant.  We use $\Delta: U_0(A)/DU(A)\to
\Aff T(A)/\rho_A(K_0(A))$ for the induced \hm.  We will use $\overline{\Delta}: U_0(A)/CU(A)\to \Aff(T(A))/\overline{\rho_A(K_0(A))}$ for the induced
\hm.
Note  that when $A$ has tracial rank at most one
the map from $U_0(A)/CU(A)\to U_0(M_n(A)/CU(M_n(A))$ is an isomorphism (see Cor. 3.5 of \cite{Lnhomt1}).  In this case, by a theorem of K. Thomsen (Cor. 3.2 of \cite{Kthm}),
$\Delta$ and $\overline{\Delta}$ are isomorphisms. Moreover,  in this case, Thomsen (Cor. 3.3 of \cite{Kthm}) also
provided the following splitting short exact sequence:
\beq\label{dcu-1}
0\to \Aff(T(A))/{\overline{\rho_A(K_0(A))}}\stackrel{\overline{\Delta}^{-1}}{\longrightarrow} U(A)/CU(A)\stackrel{\pi_1}{\rightleftharpoons}_{s_1} K_1(A)\to 0,
\eneq
where $\pi_1: U(A)/CU(A)\to U(A)/U_0(A)\cong K_1(A)$ is the quotient map and $s_1$ is a fixed splitting \hm.
Please note that $\pi_1$ and $s_1$ will be used later.

}
\end{df}

\begin{df}\label{Dcel}
{\rm 
Let $A$ be a unital \CA\, and let $u\in U_0(A).$ Denote by ${\rm cel}(u)$
the  infimum of all the length of the paths in $U_0(A)$ starting at $u$ and ending 
at $1_A.$
}
\end{df}

\begin{df}\label{dAq}
{\rm
Let $A$ be a unital \CA\, and let  $A^{\bf 1}$ be the unit ball of $A.$  Denote by $A^q$ the image
of $A_{s.a.}$ of the map given by $a\mapsto \hat{a},$ where $\hat{a}(\tau)=\tau(a)$ for all
$\tau\in T(A)$ and for all $a\in A_{s.a.}.$  By \cite{CP},
$A^q=\Aff(T(A))$ (see also Cor. 3.10 of \cite{BPT}). Denote by $A^{q,{\bf 1}}$ the image of $A^{\bf 1}$ in $A^q$ and
by $A_+^{q,{\bf 1}}$ the image of $A_+^{\bf 1}$ in $A^q,$ respectively.
}
\end{df}

\begin{df}\label{Dfunc}
{\rm
Define a continuous function $h_\lambda$ (for some $\lambda\in (0,1/2)$) on $[0,1]$ as follows
$$
h_\lambda=\begin{cases} 0, &  0\le t\le \lambda\\
                               t/\lambda-1, &  \lambda< t<2\lambda\\
                               1, & 2\lambda\le t\le 1.
                               \end{cases}
                               $$

}
\end{df}

\begin{df}\label{dfdimfunc}
{\rm
Let $A$ and $h_\lambda$ be as above. Let $a\in A_+\setminus\{0\}.$
Define, for each $\tau\in T(A),$
\beq\label{Ddim-1}
d_{\tau}(a)=\lim_{\lambda\to 0} \tau(h_\lambda(a))
\eneq
}
\end{df}

\begin{df}\label{dangle}
{\rm
Let $a\in A$ be an element such that $\|a^*a-1\|<1/8$ and $\|aa^*-1\|<1/8.$
Then $a(a^*a)^{-1/2}$ is a unitary in $A.$
 Set $\langle a \rangle =a(a^*a)^{-1/2}.$ Note that $\|\langle a\rangle-a\|<1.$  Moreover,
$\|\langle a \rangle -a\|$  is small  if   $\|a^*a-1\|$ is small (which does not depend on $A$ or $a$).
We will use this notation throughout the paper.
}
\end{df}

\begin{df}\label{dJS}
{\rm
Throughout this paper ${\cal Z}$ is the Jiang-Su algebra which is a unital projectionless simple
ASH-algebra with a unique tracial state (\cite{JS}).  ${\cal Z}$ is a strongly self-absorbing algebra. So  $\otimes_{n\in \Z}{\cal Z}\cong {\cal Z}.$
Denote by $\sigma_0: \bigotimes_{n\in \Z}{\cal Z}\to \bigotimes_{n\in \Z}{\cal Z}$ the shift, i.e.,
$\sigma_0(\cdots \otimes a_{-1}\otimes a_0\otimes a_1\otimes \cdots)=(\cdots\otimes a_{-2}\otimes a_{-1}\otimes a_0
\otimes \cdots)$ for all
$(\cdots a_{-1}\otimes a_0\otimes a_1\otimes \cdots)\in \bigotimes_{n\in \Z} {\cal Z}.$
By identifying $\bigotimes_{n\in \Z} {\cal Z}$ with ${\cal Z},$ we view
$\sigma_0$ is an automorphism on ${\cal Z}.$

We then write ${\cal Z}=\bigotimes_{n\in \N}{\cal Z}$ and denote
by $\sigma$ the $\otimes_{n\in \N}\sigma_0.$

It  is a strongly asymptotically inner automorphism.
}
\end{df}

\begin{df}\label{dH1}
{\rm
Let $A$ be a unital \CA.
Recall that
$$
H_1(K_0(A), K_1(A))=\{x\in K_1(A): \psi([1_A])=x \,\,\,{\rm for\,\,\, some\,\,\,}\psi\in {\rm Hom}(K_0(A), K_1(A))\}.
$$
If $K_0(A)=\Z\cdot [1_A]\oplus G$ for some abelian group $G,$ then
$H_1(K_0(A), K_1(A))=K_1(A).$  If $[1_A]$ does not have finite order and  $K_1(A)$ is divisible, then $H_1(K_0(A), K_1(A))=K_1(A).$
}
\end{df}

\begin{df}\label{dappdiv}
{\rm
Let $A$ be a unital \CA. We say that $A$ is  tracially approximately divisible,
if the following hold:
For any $\ep>0,$ any $a\in A_+\setminus \{0\},$ any finite subset ${\cal F}\subset A$ and any integer
$N\ge 1,$  there exists a projection $p\in A$ and a finite dimensional \SCA\, $C\cong \bigoplus_{i=1}^J M_{r_i}$ with $r_i\ge N$ and with $1_C=p$ such that
\beq\nonumber
\|[p, x]\|<\ep\tforal x\in {\cal F},\,\,\,(1-p)\lesssim a\, \,\, {\rm (in \,\,\, Cuntz\,\,\, relation)},\\
\|[pxp, c]\|<\ep\tforal x\in {\cal F} \andeqn \tforal c\in C\,\,\,{\rm with}\,\,\, \|c\|\le 1.
\eneq

It is proved in 5.4 of \cite{Lntr=1} that every  unital \CA\, with tracial rank at most one is tracially approximately
divisible.}
\end{df}

\begin{df}\label{Klbar}
{\rm
Let $A$ be a \CA. Using the notation introduced in \cite{DL}, we denote
$$
\underline{K}(A)=\bigoplus_{i=0,1}(K_i(A)\bigoplus \bigoplus_{n\ge 1} K_i(A, \Z/n\Z)).
$$
We will use $\boldsymbol{\bt}: \underline{K}(A)\to \underline{K}(A\otimes C(\T))$ for the injective \hm\, defined in 2.10 of \cite{Lnmem}.

Let $A$ and $B$ be two unital \CA s and let $h: A\to B$ be a unital \hm. Suppose that
$u\in B$ such that $h(a)u=uh(a)$ for all $a\in A.$ This induces a unital \hm\, ${\bar h}: A\otimes C(\T)\to B.$ As in 2.10 of \cite{Lnmem},  the map ${\rm Bott}(h, u): \underline{K}(A)\to \underline{K}(B)$ is defined by
$$
{\rm Bott}(h, u)=[{\bar h}]\circ {\boldsymbol{\bt}}.
$$
In particular, one has
\beq\label{Bottplus-2}
{\rm Bott}(h, 1_B)=0.
\eneq
If $\{u(t): t\in [0,1]\}\subset B$ is a continuous path of unitaries and $h(a)u_t=u_th(a)$
for all $a\in A$ and $t\in [0,1],$
then
\beq\label{Bottplus-1}
{\rm Bott}(h, u_0)={\rm Bott}(h, u_t)\rforal t\in (0,1].
\eneq
Suppose that $u_1, u_2,...,u_k\in B$ are  unitaries such that
$h(a)u_i=u_ih(a)$ for all $a\in A,$ $i=1,2,...,k.$
Define $H: A\to M_k(B)$ by
$H(a)={\rm diag}(h(a),h(a),...,h(a))$ for all $a\in A.$
Let $U={\rm diag}(u_1,u_2,...,u_k).$ Then
\beq\label{Bottplus-1+1}
{\rm Bott}(H, U)=\sum_{i=1}^k {\rm Bott}(h, u_i).
\eneq
Now let $V={\rm diag}(u_1u_2\cdots u_k, 1,1,...,1).$ Then, by a standard homotopy, using
(\ref{Bottplus-1}), (\ref{Bottplus-1+1}) and (\ref{Bottplus-2}),
\beq\label{Bottplus-1+2}
{\rm Bott}(h, u_1u_2\cdots u_k)=\sum_{i=1}^k {\rm Bott}(h, u_i).
\eneq

}
\end{df}

\begin{df}\label{KBottmap}
{\rm
Let $A$ and $B$ be two unital \CA s and let $h: A\to B$ be a unital \hm.
Let ${\cal P}\subset \underline{K}(A)$ be a finite subset.  There exists $\dt_{\cal P}>0$
and a finite subset ${\cal G}_{\cal P}\subset A$ (both depend on ${\cal P}$ but not on
$A$ or $h$) such that
$$
{\rm Bott}(h, u)|_{\cal P}
$$
is well defined, when $u\in B$ is a unitary such that
$\|h(b)u-uh(b)\|<\dt_{\cal P}$ for all $b\in {\cal G}_{\cal P}.$  In fact, ${\rm Bott}(h, u)|_{\cal P}=[\psi]|_{{\boldsymbol{\bt}}({\cal P})},$ where $\psi: B\otimes C(\T)\to A$ is the map induced by $h$ and $u$
(we refer the reader to Definition 2.10 of \cite{Lnmem} for details).
In what follows, whenever we write
${\rm Bott}(h, u)|_{\cal P},$ we assume that $\|h(b)u-uh(b)\|<\dt\le \dt_{\cal P}$ and for all $b\in {\cal G}\supset {\cal G}_{\cal P}.$

Suppose that $\{u_t: t\in [0,1]\}\subset B$ is a continuous path of unitearies such that
$\|h(a)u_t-u_th(a)\|<\dt_{\cal P}$ for all $a\in {\cal G}_{\cal P}$ and for all $t\in [0,1].$
Then
\beq\label{Bottplus-3}
{\rm Bott}(h, u_0)|_{\cal P}={\rm Bott}(h, u_t)|_{\cal P}\tforal t\in (0,1].
\eneq
Suppose that $u_1, u_2,...,u_k\in B$ are  unitaries such that
$\|h(a)u_i-u_ih(a)\|<{\dt_{\cal G}\over{2k}}$ for all $a\in A,$ $i=1,2,...,k.$
Then, as above, using \ref{Klbar} and (\ref{Bottplus-3}),
\beq\label{Bottplus-4}
{\rm Bott}(h, u_1u_2\cdots u_k)|_{\cal P}=\sum_{i=1}^k{\rm Bott}(h,u_i)|_{\cal P},
\eneq
a fact that has been used many times (see  \cite{ER} for example).

}
\end{df}

\section{Homotopy Lemmas with controlling of length}
This section contains some refinement of homotopy lemmas in \cite{LN2}.

Let $\nabla: (0,1)\to (0,1)$ be a non-increasing function with $\lim_{t\to 0}\nabla(t)=0.$
Define
\beq\label{Delta0}
\nabla_{00}(t)=\nabla({1\over{2^n}}),\,\,\,\,\,{\rm if}\,\,\, t\in [{1\over{2^{n+1}}}, {1\over{2^n}})\andeqn\\
\nabla_0(t)={\sqrt{2}\over{48}}\nabla_{00}({t\sqrt{2}\over{6}})t.
\eneq
Then $\nabla_{00}$ and $\nabla_{0}$ are non-decreasing and
$\lim_{t\to 0}\nabla_0(t)=0.$
This notation will be used in the following.

\begin{lem}\label{HL1}
Let $C=PM_r(C(X))P, $ where $X$ is a   compact metric space, $r\ge 1$ is an integer and  where $P\in M_r(C(X))$ is a projection, and let $\nabla: (0, 1)\to(0, 1)$ be a non-decreasing function with $\lim_{t\to 0} \nabla(t)=0$
and
$1/2>\eta>0.$

Let $\mathcal F\subseteq C$, $\mathcal G'\subseteq C\otimes \mathrm{C}(\mathbb T)$, $\mathcal H\subseteq C\otimes \mathrm{C}(\mathbb T)$ be finite subsets, and let $\ep>0$. Then there are   $\dt>0$ and a finite subset $\mathcal G\subseteq C$ such that for any C*-algebra $A$ which is tracially approximately divisible {\rm (}see \ref{dappdiv}{\rm )},  any homomorphism $\phi: C\to A$ and  any unitary $u\in A$ with
\beq\label{HL1-1-1}
\|[\phi(c), u]\|<\delta\quad\forall c\in\mathcal G \andeqn\\
\mu_{\tau\circ\phi}(O_a)> \nabla(a) \tforal \tau\in T(A)
\eneq
 ($\mu_{\tau\circ \phi}$ is the Borel measure induced by the positive linear functional $\tau\circ \phi$) for any open ball $O_a$
of $X$ with radius $a>\eta,$
there exist unitaries $w_1, w_2\in A$, a path of unitaries $\{w(t); t\in[0, 1]\}\subset A$ with $w(0)=1$ and $w(1)=w_1w_2w_1^*w_2^*=:w$, and a completely positive $\mathcal G'$-$\ep$-multiplicative linear maps $L_1, L_2: C\otimes\mathrm{C}(\mathbb T) \to A$  such that
\beq\label {hl1-1}
&&\|[w_i, \phi(a)]\|<\ep \tforal a\in {\cal F},\,\,\,i=1, 2,\\\label
{hl1-2}
&&\|[w(t),\, u]\|<\ep,\,\,\, \|[w(t), \phi(a)]\|<\ep,\,\,\,\tforal a\in\F,\,\,\, \tand t\in[0, 1],\\\label{hl-3}
&&\|L_1(a\otimes z)-(\phi(a)uw)\|<\ep,\quad \|L_1(a\otimes 1)-\phi(a)\|<\ep,\quad\tforal a\in\F,\\\label{hl-4}
 &&\|L_2(a\otimes z)-(\phi(a) w)\|<\ep,\quad \|L_2(a\otimes 1)-\phi(a)\|<\ep,\quad\tforal a\in\F,\\\label{hl-5}
&&|\tau\circ L_1(g)-\tau\circ L_2(g)|<\ep,\quad\tforal g\in \mathcal H,\ \tforal \tau\in \mathrm{T}(A),
\eneq
and
$$\mu_{\tau\circ L_i}(B_a)>\nabla_0(a),\quad i=1, 2, \tforal \tau\in T(A)$$ and for any open ball $B_a$ of $X\times\mathbb T$ with radius $a>3\sqrt{2}\eta$.

Moreover,
\beq\label{hl-6}
{\rm length}\{w(t)\}\le \pi.
\eneq
Furthermore, one may require that $w(t)\in CU(A)$ for all $t\in [0,1]$ but with
\beq\label{hl-7}
{\rm length}\{w(t)\}\le 2\pi.
\eneq
\end{lem}

\begin{proof}
The lemma  follows from  Lemma 3.7 of \cite{LN2} except the last part beginning with ``Moreover".
So the proof is exactly the same but we need to justify (\ref{hl-6}) (as well as (\ref{hl-7})) above.

Note that in the proof of Lemma 3.7 of \cite{LN2},
$$
\phi({\cal F}\cup {\cal F}_1\cup {\cal G}\cup {\cal T})\in_{\ep''} B'\cap A\andeqn u\in_{\ep''} B'\cap A.
$$
Moreover,
$w'(t)\in B$ for all $t\in [0,1]$ and $B\cong M_k$ for some integer $k\ge 1.$
Therefore the length of $\{w'(t)\}$ could be made no more than $\pi.$ Since
$w(t)=(1-p)+w'(t),$ as in the proof of Lemma 3.7 of \cite{LN2}, the length of $\{w(t)\}$ could  be
controlled by $\pi.$

Note that $w(1)=w\in CU(B).$  We may write $w=\exp(ih)$ with $h\in B,$ $\|h\| \le 2\pi$ and $\tau(h)=0$ for all $\tau\in T(B).$
One then can choose $w'(t)=\exp(i(1-t)h)$ for $t\in [0,1].$ Since $\tau(th)=t\tau(h)=0$ for all $t\in [0,1].$
$w'(t)\in CU(B)\subset CU(A)$ for all $t\in [0,1].$
\end{proof}

Let $B_a$ be an open ball with center $\xi\in X$ and radius $a\ge 1.$
One may note that, if $r\in (0,1)$, then
$\mu_{\tau\circ \phi}(B_a)\ge \mu_{\tau\circ \phi}(B_r),$ where
$B_r$ is the open ball with the same center and radius $r.$

\begin{thm}\label{HL2}
Let $C=\mathrm{C}(X)$ with $X$ a compact
metric space
and let $\nabla: (0, 1) \to (0, 1)$ be a non-decreasing map. For any $\ep>0$ and any finite subset $\mathcal F \subseteq C$, there exists $\delta>0$, $\eta>0$, $\gamma>0$, finite subsets $\mathcal G \subseteq C$, $\mathcal P\subseteq \underline{K}(C)$,  a finite
 subset ${\cal Q}=\{x_1, x_2,...,x_m\}\subset K_0(C)$ which generates a free
 subgroup and $x_i=[p_i]-[q_i],$ where $p_i, q_i\in M_n(C)$
 (for some integer $n\ge 1$) are projections,
 satisfying the following:

Suppose that $A$  is a unital simple C*-algebra with $TR(A)\leq 1$, $\phi: C\to A$ is a unital homomorphism and $u\in A$ is a unitary, and suppose that
$$\|[\phi(c), u]\|<\delta,\ \tforal c\in\mathcal G\quad \tand \quad \mathrm{Bott}(\phi, u)|_{\mathcal P}=0,$$
$$\mu_{\tau\circ \phi}(O_a)\geq\nabla(a)\ \tforal \tau\in T(A),$$
where $O_a$ is any open ball in $X$ with radius $\eta\leq a<1$ and $\mu_{\tau\circ \phi}$ is the Borel probability measure defined
by $\tau\circ \phi,$  and, for each $1\leq i\leq m$,
there is $v_i\in CU(M_n(A))$ such that
\beq\label{HL2-1}
\|\langle(\mathbf 1_n-\phi(p_i)+\phi(p_i)({\bf 1}_n\otimes {u}))
(\mathbf 1_n-\phi(q_i)+\phi(q_i)({\bf 1}_n\otimes {u}^*))\rangle-v_i\|<\gamma.
\eneq
Then there is a continuous path of unitaries $\{u(t): t\in[0, 1]\}$ in $A$ such that
$$u(0)=u, u(1)=1,\ \textrm{and}\ \|[\phi(c), u(t)]\|<\ep$$
for any $c\in\mathcal F$ and for any $t\in[0, 1]$.

Moreover,
\beq\label{HL2-2}
{\rm length}(\{u(t)\})\le 2\pi+\ep.
\eneq

One can also  require that
\beq\label{HL2-3}
{\rm dist}(u(t), CU(A))<\ep
\eneq
with ${\rm length}(\{u(t)\})\le 4\pi+\ep.$


\end{thm}

{\bf Remark}:  Here as indicated in \ref{KBottmap},  $\dt<\dt_{\cal P}$ and ${\cal G}\supset {\cal G}_{\cal P}.$
Moreover, we also assume that $\dt$ is sufficiently small and ${\cal G}$ is sufficiently large
so that $\|\phi(p_i)({\bf 1}_n\otimes u)-({\bf 1}_n\otimes u)\phi(p_i)\|$ is so small that (\ref{HL2-1}) makes sense. Condition (\ref{HL2-1}) may be understood that we insist that image of elements in ${\cal Q}$ under ``the Bott map"
are in $CU(M_n(A)).$  However, we only need to require that the torsion free part of $K_0(A)$ has that property. This is the reason ${\cal Q}$
is used in the statement.

\begin{proof}
Note that, by Lemma \ref{HL1} of this paper,  in the proof of Theorem 3.9 of \cite{LN2}, the length of $\{w'(t)\}$ is at most
$\pi.$ Thus the length of $\{w(t): t\in [1/2, 1]\}$ in the proof of 3.9 of \cite{LN2} is at most $\pi$ and
the length of $\{w(t): t\in [0, 1/4]\}$ in the proof of Theorem 3.9 of \cite{LN2} is at most $\pi.$
We also note that the length of $\{w(t): t\in [1/4, 1/2]\}$ is no more than
$$
2\arcsin (\ep_0/2).
$$
Note that, at the beginning of the proof of Theorem 3.9 of \cite{LN2},  we assume $\ep_0<\ep/2.$  This will imply
that $2\arcsin(\ep_0/2)<\ep$ for $0<\ep_0<\pi/4.$
But we certainly can choose $\ep_0$ so  that  $2\arcsin(\ep_0/2)<\ep$
at  the proof of  3.9 of \cite{LN2}. So the total length of $\{w(t): t \in [0,1]\}$ is at most $2\pi+\ep$ (for any given $\ep$).

Note that, in the proof of Theorem 3.9 of \cite{LN2}, as in the proof of Lemma \ref{HL1}, $\{w'(t)\}$ can be chosen to be in
$CU(A)$ with length at most $2\pi.$ Thus $w(t)\in CU(A)$ for all $t\in [1/2,1]$ and $w(t)\in CU(A)$ for $t\in [0,1/4].$ However,
the length of $\{w(t): t\in [1/2, 1]\}$ and the length of $\{w(t): t\in [0,1/4]\}$  will be  both bounded by $2\pi.$
It follows from this and the above estimates that (\ref{HL2-2}) and (\ref{HL2-3}) hold.

%
%

\end{proof}

\begin{df}\label{mea}
Let $X$ be a compact metric space and $P\in M_r(C(X))$ be a projection, where $r\geq 1$ is an integer. Put $C=PM_r(C(X))P$ and suppose $\tau\in T(C)$. It is known that there exists a probability measure $\mu_\tau$ on $X$ such that
$$\tau(f)=\int_X t_x(f(x)) d\mu_\tau(x),$$
where $t_x$ is the normalized trace on $P(x)M_rP(x)$ for all $x\in X$.
\end{df}

\begin{rem}
Regard $C(X)$ as the center of $C=PM_r(C(X))P$, and denote by $\iota: C(X)\to C$ the embedding. Then the measure $\mu_\tau$ is in fact induced by the trace $\tau\circ\iota$ on $C(X)$.
\end{rem}

\begin{cor}\label{CHL3}
Let $C=PM_r(C(X))P,$ where $X$ is a compact subset of a finite CW complex, $r\ge 1$ is an integer and $P\in M_r(C(X))$ is a projection.
Then Theorem \ref{HL2} holds for $C$ using the measure $\mu_{\tau}$ as in \ref{mea}.
\end{cor}

\begin{proof}
We will use \ref{HL2} and the  proof of Theorem 3.10 in \cite{LN2}.
\end{proof}

The following is an easy compactness fact.
\begin{prop}\label{positive}
Let $X$ be a compact metric space and $C=PM_r(C(X))P,$ where $r\ge 1$ is an integer and $P\in M_r(C(X))$ is a projection.
Let ${\overline \nabla}: C^{q,{\bf 1}}_+\setminus \{0\}\to (0,1)$ be an order preserving  map.
Then there is non-decreasing function $\nabla: (0,1)\to (0,1)$ such that,
for any $\eta>0,$ there is a finite subset ${\cal H}\subset C^{\bf 1}_+ \setminus \{0\}$ satisfying the following:
If $\phi: C\to A$ is a unital injective \hm, where $A$ is a unital simple
\CA\, with $T(A)\not=\emptyset,$  such that
$$
\tau\circ\phi(h)\ge {\overline\nabla}(\hat{h})\tforal \tau\in T(A)
$$
and for all $h\in {\cal H},$ then
$$
\mu_{\tau\circ \phi}(O_r)\ge \nabla(r).
$$
for all $\tau\in T(A)$ and for all $r\ge \eta.$

\end{prop}

The next is a restatement of \ref{CHL3} using \ref{positive}.

\begin{cor}\label{CHL2}
Let $X$ be a compact subset of a finite CW complex and
let $C=PM_r(C(X))P,$ where $r\ge 1$ is an integer and $P\in M_r(C(X))$ is a projection. Let ${\overline\nabla}: C^{q, {\bf 1}}_+\setminus \{0\}\to (0,1)$ be
an order preserving  map.
For any $\ep>0$ and any finite subset $\mathcal F \subseteq C$, there exists $\delta>0$,  $\gamma>0$,  finite subsets $\mathcal G \subseteq C$, $\mathcal P\subseteq \underline{K}(C)$,  a finite
 subset ${\cal Q}=\{x_1, x_2,...,x_m\}\subset K_0(C)$ which generates a free
 subgroup and $x_i=[p_i]-[q_i],$ where $p_i, q_i\in M_n(C)$
 (for some integer $n\ge 1$) are projections and
 there exists a finite subset ${\cal H}\subset C^{{\bf 1}}_+\setminus \{0\}$
satisfying the following:

Suppose that $A$  is a unital simple C*-algebra with $TR(A)\leq 1$, $\phi: C\to A$ is a unital homomorphism and $u\in A$ is a unitary, and suppose that
$$
\|[\phi(c), u]\|<\delta,\ \tforal c\in\mathcal G\quad \tand\quad \mathrm{Bott}(\phi, u)|_{\mathcal P}=0,
$$
$$
\tau\circ \phi(h)\ge {\overline\nabla}(\hat{h}) \tforal \tau\in T(A)
$$
and for all $h\in {\cal H},$
and, for each $1\leq i\leq m$,
there is $v_i\in CU(M_n(A))$ such that
\beq\label{CCHL2-1}
\|\langle(\mathbf 1_n-\phi(p_i)+\phi(p_i)({\bf 1}_n\otimes u))
(\mathbf 1_n-\phi(q_i)+\phi(q_i)({\bf 1}_n\otimes u^*))\rangle-v_i\|<\gamma,
\eneq
Then there is a continuous path of unitaries $\{u(t): t\in[0, 1]\}$ in $A$ such that
$$u(0)=u, u(1)=1,\ \textrm{and}\ \|[\phi(c), u(t)]\|<\ep$$
for any $c\in\mathcal F$ and for any $t\in[0, 1]$.

Moreover,
\beq\label{CCHL2-2}
{\rm length}(\{u(t)\})\le 2\pi+\ep.
\eneq

One can also  require that
\beq\label{CCHL2-3}
{\rm dist}(u(t), CU(A))<\ep
\eneq
with ${\rm cel}(\{u(t)\})\le 4\pi+\ep.$


\end{cor}

The following follows immediately from \ref{CHL2}.

\begin{thm}\label{MHLL}
Let $C$ be a unital $AH$-algebra
and let $\nabla: C^{q,{\bf 1}}_+\setminus \{0\} \to (0, 1)$ be an order preserving map. For any $\ep>0$ and any finite subset $\mathcal F \subseteq C$, there exists $\delta>0$,  $\gamma>0$,  finite subsets $\mathcal G \subseteq C$, $\mathcal P\subseteq \underline{K}(C)$,  a finite
 subset ${\cal Q}=\{x_1, x_2,...,x_m\}\subset K_0(C)$ which generates a free
 subgroup and $x_i=[p_i]-[q_i],$ where $p_i, q_i\in M_n(C)$
 (for some integer $n\ge 1$) are projections, and a finite subset
 ${\cal H}\subset C^{\bf 1}_+\subset \{0\}$
 satisfying the following:
Suppose that $A$  is a unital simple C*-algebra with $TR(A)\leq 1$, $\phi: C\to A$ is a unital homomorphism and $u\in A$ is a unitary, and suppose that
$$
\|[\phi(c), u]\|<\delta,\ \forall c\in {\mathcal G} \tand\,   \mathrm{Bott}(\phi, u)|_{\mathcal P}=0,
$$
\beq\label{MHLL-n}
\tau\circ \phi(h)\ge \nabla(\hat{h})\tforal  \tau\in T(A)
\eneq
and all $h\in {\cal H},$ and,
  for each $1\leq i\leq m$,
there is $v_i\in CU(M_n(A))$ such that
\beq\label{MHL2-1}
\|\langle(\mathbf 1_n-\phi(p_i)+\phi(p_i)({\bf 1}_n\otimes {u}))
(\mathbf 1_n-\phi(q_i)+\phi(q_i)({\bf 1}_n\otimes { u}^*))\rangle-v_i\|<\gamma.
\eneq
Then there is a continuous path of unitaries $\{u(t): t\in[0, 1]\}$ in $A$ such that
$$u(0)=u, u(1)=1 \tand\,  \|[\phi(c), u(t)]\|<\ep$$
for any $c\in\mathcal F$ and for any $t\in[0, 1]$.

Moreover,
\beq\label{MHL2-2}
{\rm length}(\{u(t)\})\le 2\pi+\ep.
\eneq

One can also  require that
\beq\label{MHL2-3}
{\rm dist}(u(t), CU(A))<\ep \rforal t\in [0,1]
\eneq
with ${\rm length}(\{u(t)\})\le 4\pi+\ep.$


\end{thm}

\begin{lem}\label{TR=meas}
Let $A$ be a unital separable simple amenable \CA\, with $TR(A)\le 1.$
There is an order preserving  map $\nabla: A^{q,{\bf 1}}_+\setminus \{0\}\to (0,1)$ with $\lim_{\|\hat{h}\|\to 0}\nabla(\hat{h})=0$
satisfying the following: For any unital \hm\, $\phi: A\to B,$ where
$B$ is a unital separable \CA\, with $T(B)\not=\emptyset,$ one has
$$
\tau\circ \phi(h)\ge \nabla(\hat{h})\tforal \tau\in T(A)
$$
and for all $h\in A_+\setminus \{0\}.$
\end{lem}

\begin{proof}
Let $h\in A_+^{\bf 1}\setminus \{0\}.$ Define
$$
\nabla(\hat{h})=\inf_{\tau\in T(A)} \tau(h).
$$
Then, since $A$ is simple, $\nabla(\hat{h})>0.$ Since
$\nabla(\hat{h})\le \|\hat{h}\|,$
$$
\lim_{\|\hat{h}\|\to 0}\nabla(\hat(h))=0.
$$
Now let $\phi: A\to B$ be any unital \hm, $\phi: A\to B.$
If $\tau\in T(B),$ the $\tau\circ \phi\in T(A).$ It follows that
$$
\tau\circ \phi(h)\ge \nabla(\hat{h})\tforal \tau\in T(B).
$$

\end{proof}

\begin{thm}\label{MHLT}
Let $C$ be a unital separable simple amenable \CA\, with $TR(C)\le 1$ which satisfies the UCT.
For any $\ep>0$ and any finite subset $\mathcal F \subseteq C$, there exists $\delta>0$,  $\gamma>0$,  finite subsets $\mathcal G \subseteq C$, $\mathcal P\subseteq \underline{K}(C)$,  a finite
 subset of projections $p_1, p_2,...,p_m\in C$
  satisfying the following:

Suppose that $A$  is a unital simple C*-algebra with $TR(A)\leq 1$, $\phi: C\to A$ is a unital homomorphism and $u\in A$ is a unitary, and suppose that
$$
\|[\phi(c), u]\|<\delta,\ \forall c\in\mathcal G\quad\mathrm{and}\quad \mathrm{Bott}(\phi, u)|_{\mathcal P}=0,
$$
and
 for each $1\leq i\leq m$,
there is $v_i\in CU(A)$ such that
\beq\label{MHL2-1+}
\|\langle (1-\phi(p_i))+\phi(p_i)u\rangle -v_i\|<\gamma,\,\,\, i=1,2,...,m.
\eneq
Then there is a continuous path of unitaries $\{u(t): t\in[0, 1]\}$ in $A$ such that
$$u(0)=u, u(1)=1,\ \textrm{and}\ \|[\phi(c), u(t)]\|<\ep$$
for any $c\in\mathcal F$ and for any $t\in[0, 1]$.

Moreover,
\beq\label{MHL2-2+}
{\rm length}(\{u(t)\})\le 2\pi+\ep.
\eneq

One can also  require that
\beq\label{MHL2-3+}
{\rm dist}(u(t), CU(A))<\ep\tforal t\in [0,1]
\eneq
with ${\rm cel}(\{u(t)\})\le 4\pi+\ep.$


\end{thm}
\begin{proof}
This follows from \ref{MHLL}.
In fact, it follows from the classification result in \cite{Lntr=1} that $C$ is isomorphic to
a unital simple AH-algebra with no dimension growth.   Note $C$ is simple.
So \ref{TR=meas} applies. In other words, $\nabla$ can be given and
(\ref{MHLL-n}) holds for this $\nabla.$  So
\ref{MHLL} applies.
\end{proof}

\begin{rem}\label{Rem1}

{\rm
Note that all paths $\{u(t)\}, \{w(t)\}$ in this section  can be made not only continuous but also piecewise
smooth. Furthermore, if ${\rm length}(\{u(t): t\in [0,1]\})\le C,$ then we can also assume that
$$
\|u(t)-u(t')\|\le C|t-t'|\tforal t, t'\in [0,1].
$$

}

\end{rem}

The \CA\, $C$ below is a unital AH-algebra  with no dimension growth which has stable rank one by the classification
theorem. Thus the next follows immediately from Theorem 3.13 in \cite{LN2}.

\begin{thm}\label{Ext1}
Let $C$ be a unital separable amenable simple \CA\, with $TR(C)\le 1$ which satisfies the UCT and let
 $G\subset K_0(C)$ be a finitely generated subgroup generated by projections $p_1,p_2,...,p_n\subset C.$
Let $A$ be a unital separable simple \CA\, with $TR(A)\le 1.$
Suppose $\phi: C\to A$ is a unital \hm. Then, for any $\ep>0,$ any $\gamma>0,$ any finite subset ${\cal F}\subset C,$ any finite subset ${\cal P}\subset \underline{K}(C),$ \, and any \hm\, $\Gamma: G\to U_0(A)/CU(A),$ there is a unitary $u\in U(A)$ such that
\beq\label{Ext1-1}
\|[\phi(x),\, u]\|<\ep\tforal x\in {\cal F},\,\, {\rm Bott}(\phi,\, u)|_{\cal P}=0\tand\\
{\rm dist}(\langle (1-\phi(p_i))+\phi(p_i)u\rangle,\, \Gamma([p_i]))<\gamma,
\,\,\, 1\le i\le n.
\eneq
\end{thm}

\section{Kishimoto's conjugacy theorem}

The following serves as an alternative for  Kishimoto's stability theorem (3.4 of \cite{K1}).
The proof uses the special homotopy lemmas in the previous section and
a construction due to  Kishimoto (see also \cite{HO}).

\begin{lem}\label{KishL1}
Let $A$ be a unital separable amenable \CA\, with $TR(A)\le 1$ which satisfies the UCT and let $\af\in \Aut(A)$ be an automorphism with the Rokhlin property.
For any $\ep>0$ and any finite subset ${\cal F}\subset A,$ there exists
$\dt>0$ and $\lambda>0,$  a finite subset ${\cal G}\subset A,$  a finite subset ${\cal P}\subset \underline{K}(A),$ and a finite subset
of projections $\{p_1,p_2,...,p_m\}\subset A$ satisfying the following:
If $u\in U(A)$ such that
\beq\label{KIshL1-1}
\|[a,\, u]\|<\dt \tforal a\in {\cal G},\,\,\, {\rm Bott}({\rm id}_A, \, u)|_{{\cal P}}=0\tand\\
{\rm dist}(\langle (1-p_i)+p_iu\rangle , CU(A))<\lambda,\,\,\, i=1,2,...,m,
\eneq
then there is a continuous path of unitaries  $\{V(t): t\in [0,1]\}\subset  A$ such that $V(1)=1,$
\beq\label{KIshL1-2}
\|u-V(0)\af(V(0)^*)\|<\ep\tand \|[b,\, V(t)]\|<\ep\tforal b\in {\cal F}
\eneq
and for all $t\in [0,1],$  and
\beq\label{KIshL1-2+}
\|V(t)-V(t')\|\le (4\pi+1)|t-t'|\tforal t,t'\in [0,1].
\eneq

Moreover, a version for  ${\cal F}=\emptyset$ holds as follows:
For any $\ep>0$ and $u\in CU(A),$ there exists  a continuous path of unitaries $\{U(t): t\in  [0,1]\}\subset  U(A)$ such that $U(1)=1,$
\beq\label{KIshL1-3}
\|u-U(0)\af(U(0)^*)\|<\ep\andeqn \|U(t)-U(t')\|\le (4\pi+1)|t-t'|
\eneq
for all $t,\, t'\in [0,1].$

\end{lem}

\begin{proof}
Let $\ep>0$ and let ${\cal F}\subset A$ be a finite subset.
Choose $N\ge 1$ such that
$\pi/(N-1)<\ep/(2^{10}\cdot 5).$
Let ${\cal F}_1=\cup_{j=-(N+1)}^{N+1} \af^{j}({\cal F}).$

 We will apply 3.10. Let $\dt_1>0$ (in place of $\dt$), $\lambda_1>0$ (in place of $\gamma$), ${\cal G}_1\subset A$ (in place of ${\cal G}$) be a finite subset and ${\cal P}_1\subset \underline{K}(A)$  be a finite subset  and let
$\{p_1', p_2',...,p_{k'}'\} \subset A$ (in place of $\{p_1, p_2,...,p_m\}$) be a finite subset of projections required by \ref{MHLT} for $\eta=\ep/2^{10}(N+1)^2$ (in place of $\ep$) and
${\cal F}_1$ (in place of ${\cal F}).$
Note that we may assume that $p_i'\in {\cal G}_1,$ $i=1,2,...,k'.$

Let ${\cal G}=\cup_{j=-N-1}^{N+1}\af^j({\cal G}_1),$  ${\cal P}=\cup_{j=-N-1}^{N+1}([\af^j]({\cal P}_1))$ and
let $\{p_1,p_2,...,p_m\}$ be  a finite subset of projections in $A$ which includes
$\{p_i', \af^j(p_i'): i=1, 2,...,k' \andeqn j=\pm 1, \pm 2,...,\pm N\}$ as well as $1_A.$
We may assume that ${\cal F}_1\subset {\cal G}.$
Let $\dt={\min\{\ep,\dt_1\}\over{2^{10}(N+1)^3}}$ and $\lambda=\lambda_1/2(N+1).$

Now suppose that $u\in U(A)$ which satisfies the assumption for the above
$\dt,$  $\lambda,$ ${\cal G},$ ${\cal P}$ and $\{p_1,p_2,...,p_m\}.$
Define
\beq\label{KI1-1}
u_0=1,\,\,\,u_1=u,\,\,\, u_k=u\af(u)\cdots \af^{k-1}(u),\,\,\,k=2,...,N+1.
\eneq
For $x\in {\cal G}_1,$ $x,\af^j(x)\in {\cal G},$ $-N\le j\le N.$
 We estimate, for each $x\in {\cal G}_1,$
\beq\label{KI1-2}
\|[u_k, \, x]\|<(N+1)\dt={\min\{\ep, \dt_1\}\over{2^{10}(N+1)^2}}
\eneq
for $k=0,1,...,N+1.$
Note that by the assumption,
\beq\label{KI1-2+1}
{\rm Bott}({\rm id}_A, \, u)|_{\af^{-k}({\cal P}_1)}=0,\,\,\,k=0,1,...,N.
\eneq
Therefore
\beq\label{KI1-2+2}
{\rm Bott}(\af^{-k},u)|_{{\cal P}_1}=0,\,\,\,k=0,1,...,N.
\eneq
It follows that
\beq\label{KI1-2+3}
{\rm Bott}({\rm id}_A, \af^k(u))|_{{\cal P}_1}&=&{\rm Bott}(\af^k\circ \af^{-k}, \af^k(u))|_{{\cal P}_1}\\
&=&[\af^k]\circ {\rm Bott}(\af^{-k}, u)|_{{\cal P}_1}=0,\,\,\, k=0,1,...,N.
\eneq
Hence, by (\ref{KI1-2}) and by (\ref{Bottplus-4})
\beq\label{KI1-2+}
{\rm Bott}({\rm id}_A, u_k)|_{{\cal P}_1}=\sum_{i=0}^{k-1}{\rm Bott}({\rm id}_A, \af^{i}(u))|_{{\cal P}_1}=0,\,\,\, k=2,3,...,N+1.
\eneq
Moreover, we note that
\beq\label{KI1-3}
\langle (1-p_i)+p_i\af^j(u)\rangle &=&\langle (1-\af^j(\af^{-j}(p_i)))+\af^j(\af^{-j}(p_i)u)\rangle\\
&=&\langle \af^j(1-\af^{-j}(p_i)+\af^{-j}(p_i)u)\rangle.
\eneq
Since $\af$ is an automorphism, $\af(CU(A))=CU(A).$ It follows that
\beq\label{KI1-4}
{\rm dist}(\langle (1-p_i)+p_i\af^j(u)\rangle, CU(A))<\lambda,\,\,\, i=1,2,...,m\andeqn j=0,1,...,N+1.
\eneq
Therefore,
\beq\label{KI1-5}
{\rm dist}(\langle (1-p_i)+p_iu_k\rangle, CU(A))<\lambda_1,\,\,\, i=1,2,...,m\andeqn k=0,1,2,...,N+1.
\eneq

By applying \ref{MHLT}, we obtain  piecewise smooth and continuous paths of unitaries $\{u_j(t): t\in [0,1]\}$ of $A$ such that $u_j(0)=u_j,$
$u_j(1)=1_A,$
\beq\label{KI1-6}
&&\|[u_j(t),\, x]\|<\eta \tforal x\in \cup_{j=-N-1}^{N+1} \af^k({\cal F})\andeqn \tforal t\in [0,1],\\\label{KI1-6+1}
&&\|u_j(t)-u_j(t')\|\le (4\pi+1)|t-t'|\tforal t,\,t'\in [0,1],
\eneq
 $j=1,2,...,N+1.$
Hence, for $x\in {\cal F}$ and $t\in [0,1],$
\beq\label{KI1-7}
\|\af^k(u_j(t))x-x\af^k(u_j(t))\|
&=&\|u_j(t)\af^{-k}(x)-\af^{-k}(x)u_j(t)\|<\eta,
\eneq
 $k=0, \pm 1, \pm 2,...,\pm N,$
and $j=1,2,...,N+1.$
Set $v(t)=u_N(1-t)$ and $w(t)=u_{N+1}(1-t)$ for $t\in [0,1].$
Thus,
\beq\label{KI1-8}
\|\af^k(v(t))x-x\af^k(v(t))\|<\eta\andeqn
\|\af^k(w(t))x-x\af^k(w(t))\|<\eta
\eneq
for all $x\in {\cal F},$ $t\in [0,1]$ and
$k=0,\pm 1,\pm 2,...,\pm N.$
Put
\beq\label{KI1-8+1}
v_j=v\left({j\over{N-1}}\right),\,\,\,j=0,1,...,N-1\andeqn w_j=w({j\over{N}}),\,\,\, j=0,1,...,N.
\eneq
Define $v_j(t)=v({j(1-t)\over{N-1}}),$ $j=0,1,...,N-1,$ and $w_j(t)=w({j(1-t)\over{N}}),\,\,\,j=0,1,...,N$ for $t\in [0,1].$
Note that
\beq\label{Kadd-1}
&&v_j(0)=v_j,\,\,\,v_j(1)=v(0)=u_N(1)=1_A,\\\label{Kadd-2}
&&w_j(0)=w_j,\,\,\,
w_j(1)=w(0)=u_{N+1}(1)=1_A
\eneq
In particular,
\beq\label{Kadd-3}
v_{N-1}(0)=v_{N-1}=v(1)=u_N(0)=u_N.
\eneq

There is an integer $N_1\ge 1$ and
$0=t_0<t_1<\cdots <t_{N_1}=1$ such that
\beq\label{KI1-8+2}
&&\|u_k(t_i)-u_k(t)\|<\ep/2^{13}(N+1)^2=\eta/2^3,\,\,\,k=1,2,...,N,\\\label{KI1-8+2+}
&&\|v_j(t_i)-v_j(t)\|<\ep/2^{13}(N+1)^2=\eta/2^3,\,\,\, j=0,1,...,N-1
\andeqn\\\label{KI1-8+2++}
&&\|w_j(t_i)-w_j(t)\|<\ep/2^{13}(N+1)^2=\eta/2^3,\,\,\, j=0,1,...,N,
\eneq
for all $t\in [t_{i-1}, t_i],$ $i=1,2,...,N_1.$
Put
$$
{\cal W}=\cup_{i=0}^{N_1}(\{u_k(t_i): 1\le k\le N\}\cup \{v_j(t_i): 0\le j\le N-1\}\cup\{w_j(t_i): 0\le j\le N\}).
$$

Let ${\cal H}={\cal F}\cup_{j=-N}^N\af^j({\cal W})$ be a finite subset.
Let $\{e_{1,0}, e_{1,1},...,e_{1,N-1},e_{2,0},..., e_{2,N}\}$ be a family of mutually orthogonal projections in the definition of the Rokhlin
property in \ref{dRK} such that
\beq\label{KI1-9}
\sum_{j=0}^{N-1} e_{1,j}+\sum_{j=0}^N e_{2,j}&=&1,\\\label{KI1-9+1}
\|e_{i,j}y-ye_{i,j}\|&<& \eta/2=\ep/2^{11}(N+1)^2 \tforal y\in {\cal H},\, i=0,1,\\\label{KI1-9+2}
\|\af(e_{1,j})-e_{1,j+1}\|&<&\eta/2=\ep/2^{11}(N+1)^2,\,\,\,\,\,\,\,j=0,1,...,N-2
\andeqn\\\label{KI1-9+3}
\|\af(e_{2,j})-e_{2,j+1}\|&<&\eta/2=\ep/2^{11}(N+1)^2,\,\,\,j=0,1,2,...,N-1.
\eneq

It follows that (see 4.2 of \cite{K1})
\beq\label{KI1-10-1}
\|\af(e_{1,N-1}+e_{2,N})-(e_{1,0}+e_{2,0})\|<(2N-1)(\eta/2)<\ep/2^{9}(N+1).
\eneq
Therefore
\beq\label{KI1-10-2}
\|e_{1,j}\af(e_{1,N-1})\|&<&\ep/2^{9}(N+1)\,\,\,j=1,2,...,N-1,\\\label{KI1-10-2+1}
\|e_{2,j}\af(e_{1, N-1})\|&<&\ep/2^{9}(N+1),\,\,\, j=1,2,...,N,\\\label{KI1-10-2+2}
\|e_{1,j}\af(e_{2,N})\|&<& \ep/2^{9}(N+1),\,\,\, j=1,2,...,N-1\andeqn\\\label{KI1-10-2+3}
\|e_{2,j}\af(e_{2,N})\|&<& \ep/2^{9}(N+1),\,\,\, j=1,2,...,N.
\eneq
Moreover,  by (\ref{KI1-8+2}), (\ref{KI1-8+2+}), (\ref{KI1-8+2++}) and (\ref{KI1-9+1}),  for all $t\in [0,1],$
\beq\label{KI1-10-3}
&&\|[e_{i,j},\, \af^l(u_k(t))]\|<\eta,\,\,\,\,\,\,\,\,k=1,2,...,N,\\\label{KI1-10-3+1}
&&\|[e_{i,j},\,\af^l( v_k(t))]\|<\eta,\,\,\,\,\,\,\,\,\, k=0,1,...,N-1,\\\label{KI1-10-3+2}
&&\|[e_{i,j},\, \af^l(w_{k}(t))]\|<\eta,\,\,\,\,\,\,\,\, k=0,1,...,N
\eneq
 for
$j=1,2,...,N+i-2,\,\,\,i=0,1,$ and $-N\le l\le N.$

Define
\beq\label{KI1-10}
U(t)=\sum_{j=0}^{N-1} u_{j+1}(t)\af^{j-N+1}(v_j(t)^*)e_{1,j}+\sum_{j=0}^N u_{j+1}(t)\af^{j-N}(w_j(t)^*)e_{2,j}.
\eneq
Note that, by (\ref{Kadd-1}) and (\ref{Kadd-2}),
\beq\label{Kadd2-1}
U(1)&=&\sum_{j=0}^{N-1}u_{j+1}(1)\af^{j-N+1}(v_j(1)^*)e_{1,j}+\sum_{j=0}^N u_{j+1}(1)\af^{j-N}(w_j(1)^*)e_{2,j}\\
&=& \sum_{j=0}^{N-1}e_{1,j} +\sum_{j=0}^Ne_{2,j}=1_A.
\eneq
Put
$$
\eta_1=4(N-1)N\eta+4N(N+1)\eta+4N(N+1)\eta+4N(N+1)\eta<4^2N(N+1)\eta.
$$
Since $e_{1,j}e_{1,i}=0,$ if $i\not=j,$ and $e_{1,i}e_{2,j}=e_{2,j}e_{1,i}=0$ for all $i,j,$
one checks that,  by (\ref{KI1-10-3}), (\ref{KI1-10-3+1}) and (\ref{KI1-10-3+2}),
\beq\nonumber
\hspace{-04in}U(t)^*U(t) &= &\sum_{i,j}e_{1,j}\af^{j-N+1}(v_j(t))u_{j+1}(t)^*u_{i+1}(t)\af^{i-N+1}(v_i(t)^*)e_{1,i}\\\nonumber
  &&\hspace{0.2in}+\sum_{i,j}e_{2,j}\af^{j-N}(w_j(t))u_{j+1}(t)^*u_{i+1}(t)
  \af^{i-N}(w_i(t)^*)e_{2,i}\\\nonumber
  &&\hspace{0.3in}+\sum_{i,j}e_{1,j}\af^{j-N+1}(v_j(t))u_{j+1}(t)^*u_{i+1}(t)
  \af^{i-N}(w_i(t)^*)e_{2,i}\\
  &&\hspace{0.4in}+\sum_{i,j}^N e_{2,j}\af^{j-N}(w_j(t))u_{j+1}(t)^*u_{i+1}(t)\af^{i-N+1}(v_i(t)^*)e_{1,i}\\
  &&\hspace{-0.2in}\approx_{\eta_1}
  \sum_{j=0}^{N-1}e_{1,j}\af^{j-N+1}(v_j(t))u_{j+1}(t)^*u_{j+1}(t)\af^{j-N+1}
  (v_j(t)^*)e_{1,j}\\
  &&\hspace{0.5in}+\sum_{j=0}^N e_{2,j}\af^{j-N}(w_j(t))u_{j+1}(t)^*u_{j+1}(t)\af^{j-N}(w_j^*(t))e_{2,j}\\
  &=&\sum_{j=0}^{N-1}e_{1,j}+\sum_{j=0}^Ne_{2,j}=1
  \eneq
  for all $t\in [0,1].$
Thus
\beq\label{KI1-12}
\|U(t)^*U(t)-1\|<4^2\eta=16(\ep/2^{10})=\ep/2^6\tforal t\in [0,1].
\eneq
Similarly,
\beq\label{KI1-13}
\|U(t)U(t)^*-1\|<\ep/2^6\tforal t\in [0,1].
\eneq
One also checks that, if $x\in {\cal F},$ by (\ref{KI1-9+1}), (\ref{KI1-7}), (\ref{KI1-8}) and (\ref{KI1-6}),
\beq\label{KI1-13+}
\hspace{-0.7in}U(t)x&=&\sum_{j=0}^{N-1} u_{j+1}(t)\af^{j-N+1}(v_j(t)^*)e_{1,j}x+\sum_{j=0}^N u_{j+1}(t)\af^{j-N}(w_j(t)^*)e_{2,j}x\\
&&\hspace{- 1in}\approx_{(2N+1)\eta/2}\sum_{j=0}^{N-1} u_{j+1}\af^{j-N+1}(v_j(t)^*)xe_{1,j}
+\sum_{j=0}^N u_{j+1}(t)\af^{j-N}(w_j(t)^*)xe_{2,j}\\
&\approx_{(2N+1)\eta}& \sum_{j=0}^{N-1} u_{j+1}(t)x\af^{j-N+1}(v_j(t)^*)e_{1,j}
+\sum_{j=0}^N u_{j+1}(t)x\af^{j-N}(w_j(t)^*)e_{2,j}\\
&&\hspace{-1 in}\approx_{(2N+1)\eta} \sum_{j=0}^{N-1} xu_{j+1}(t)\af^{j-N+1}(v_j(t)^*)e_{1,j}
+\sum_{j=0}^N xu_{j+1}(t)\af^{j-N}(w_j(t)^*)e_{2,j}
=xU(t).
\eneq
Therefore
\beq\label{KI1-14}
\|[U(t),\, x]\|<6N\eta\le 6\ep/2^{10}(N+1)<\ep/2^7\tforal x\in {\cal F}.
\eneq
We have, by (\ref{KI1-9+2}), (\ref{KI1-10-2}), (\ref{KI1-9+1}), (\ref{KI1-6+1}), (\ref{KI1-8+1}) and (\ref{KI1-9+1}),
\beq\label{KI1-15}
 &&\hspace{-1in}\sum_{i,j}u_{j+1}\af^{j-N+1}(v_j^*)e_{1,j}\af(e_{1,i})\af^{i-N+2}(v_i)\af(u_{i+1}^*)\\\nonumber
&&\hspace{-0.5in}\approx_{ N\ep/2^{9}(N+1)+ N(N-1)\eta}\,\,
u_1e_{1,0}\af(e_{1,N-1})\af(v_{N-1})
\af(u_N^*)+\\
&&\hspace{0.8in}\sum_{j=1}^{N-1}u_{j+1}\af^{j-N+1}(v_j^*)e_{1,j}\af^{j-N+1}(v_{j-1})
\af(u_j^*)\\
\hspace{-0.8in}({\rm by\, (\ref{Kadd-3})})\hspace{0.4in} &&=u_1e_{1,0}\af(e_{1,N-1})+
\sum_{j=1}^{N-1}u_{j+1}\af^{j-N+1}(v_j^*)e_{1,j}\af^{j-N+1}(v_{j-1})
\af(u_j^*)\\
&&\hspace{-0.5in}\approx_{4(N-1)\eta}\,\, u_1e_{1,0}\af(e_{1,N-1})+\sum_{j=1}^{N-1}e_{1,j}u_{j+1}\af^{j-N+1}(v_j^*v_{j-1})\af(u_j^*)e_{1,j}\\
&&\hspace{-0.5in}\approx_{(4\pi+1)/(N-1)}\,\, u_1e_{1,0}\af(e_{1,N-1})+\sum_{j=1}^{N-1}e_{1,j}u_{j+1}\af(u_j^*)e_{1,j}
\\
&&\hspace{-0.5in}\approx_{(N-1)\eta}\,\, ue_{1,0}\af(e_{1, N-1})+\sum_{j=1}^{N-1} ue_{1,j}.
\eneq
Thus
\beq\label{KI1-16}
\hspace{-0.4in}\|\sum_{i,j} u_{j+1}\af^{j-N+1}(v_j^*)e_{1,j}\af(e_{1,i})\af^{i-N+2}(v_i)\af(u_{i+1}^*)-(\sum_{j=1}^{N-1} ue_{1,j}+ue_{1,0}\af(e_{1, N-1}))\|
<{\ep\over{2^7}}.
\eneq
We estimate that, by (\ref{KI1-9+3}), (\ref{KI1-10-2+1}), (\ref{KI1-9+1}), (\ref{KI1-8+1}), (\ref{KI1-6+1}) and (\ref{KI1-9+1}),
\beq\label{KI1-17}
&&\hspace{-0.8in}\sum_{i,j}u_{j+1}\af^{j-N}(w_j^*)e_{2,j}\af(e_{2,i})\af^{i-N+1}(w_i)\af(u_{i+1}^*)\\\nonumber
&&\approx_{N\ep/2^{10}(N+1)+N(N+1)\eta}\,\,u_1e_{2,0}\af(e_{2,N})\af(w_N)\af(u_{N+1}^*)+\\
 &&\hspace{1.6in}\sum_{j=1}^{N}u_{j+1}\af^{j-N}(w_j^*)e_{2,j}\af^{j-N}(w_{j-1})\af(u_j^*)\\
&&\approx_{4N\eta}\,\,u_1e_{2,0}\af(e_{2,N})+ \sum_{j=1}^{N}e_{2,j}u_{j+1}\af^{j-N}(w_j^*)\af^{j-N}(w_{j-1})\af(u_j^*)e_{2,j} \\
&&\approx_{(4\pi+1)/(N-1)}\,\, ue_{2,0}\af(e_{2,N})+\sum_{j=1}^{N}e_{2,j}u_{j+1}\af(u_j^*)e_{2,j} \\
&&\approx_{N\eta}\,\,ue_{2,0}\af(e_{2,N})+ \sum_{j=1}^{N}ue_{2,j}.
\eneq
Thus
\beq\label{KI1-18}
\hspace{-0.2in}\|\sum_{i,j}u_{j+1}\af^{j-N}(w_j^*)e_{2,j}\af(e_{2,i})\af^{i-N+1}(w_i)\af(u_{i+1}^*)-(\sum_{j=1}^{N}ue_{2,j} +ue_{2,0}\af(e_{2,N}))\|<{\ep\over{2^7}}.
\eneq
Moreover, by (\ref{KI1-9+3}), (\ref{KI1-10-2+2}), (\ref{KI1-9+2}) and (\ref{KI1-10-2+1}),
\beq\label{KI1-19}
&&\hspace{-0.8in}\sum_{i,j}u_{j+1}\af^{j-N+1}(v_j^*)e_{1,j}\af(e_{2,i})\af^{i-N+1}(w_i)\af(u_{i+1}^*)\\
&&\approx_{N^2\eta} \sum_{j=0}^{N-1}u_{j+1}\af^{j-N+1}(v_j^*)e_{1,j}\af(e_{2, N})\af(w_N)\af(u_{N+1}^*)\\
&&\approx_{N\ep/2^{10}(N+1)}u_{1}\af^{-N+1}(v_0^*)e_{1,0}\af(e_{2, N})\\
&&=ue_{1,0}\af(e_{2, N})\andeqn\\
&&\hspace{-0.8in}\sum_{i,j}u_{j+1}\af^{j-N}(w_j^*)e_{2,j}\af(e_{1,i})\af^{i-N+2}(v_i^*)
\af(u_{i+1}^*)\\
&&\approx_{(N+1)N\eta} \sum_{j=0}^{N}u_{j+1}\af^{j-N}(w_j^*)e_{2,j}\af(e_{1,N-1})\af(v_{N-1}^*)\af(u_N)\\
&&\approx_{N\ep/2^{10}(N+1)}u_1\af^{-N}(w_0^*)e_{2,0}\af(e_{1, N_1})\\
&&=ue_{2,0}\af(e_{1, N-1}).
\eneq
In other words,
\beq\label{KI1-20}
&&\hspace{-0.3in}\|\sum_{i,j}u_{j+1}\af^{j-N+1}(v_j^*)e_{1,j}\af(e_{2,i})\af^{i-N+1}(w_i)\af(u_{i+1}^*)-ue_{1,0}\af(e_{2, N})\|<\ep/2^9\andeqn\\\label{KI1-20+1}
&&\hspace{-0.3in}\|\sum_{i,j}u_{j+1}\af^{j-N}(w_j^*)e_{2,j}\af(e_{1,i})\af^{i-N+2}(v_i^*)
\af(u_{i+1}^*)-ue_{2,0}\af(e_{1, N-1})\|<\ep/2^9.
\eneq
By (\ref{KI1-16}), (\ref{KI1-18}), (\ref{KI1-20}) and (\ref{KI1-20+1}),
\beq\nonumber
\hspace{-0.6in}U(0)\af(U(0)^*)&=&\sum_{i,j} u_{j+1}\af^{j-N+1}(v_j^*)e_{1,j}\af(e_{1,i})\af^{i-N+2}(v_i)\af(u_{i+1}^*)\\\nonumber
&&\hspace{0.2in}+\sum_{i,j}u_{j+1}\af^{j-N}(w_j^*)e_{2,j}\af(e_{2,i})\af^{i-N+1}(w_i)
\af(u_{i+1}^*)\\\nonumber
&&\hspace{0.3in}+\sum_{i,j}u_{j+1}\af^{j-N+1}(v_j^*)e_{1,j}\af(e_{2,i})\af^{i-N+1}(w_i)
\af(u_{i+1}^*)\\
&&\hspace{0.4in}+\sum_{i,j}u_{j+1}\af^{j-N}(w_j^*)e_{2,j}\af(e_{1,i})\af^{i-N+2}(v_i)
\af(u_{i+1}^*)\\\nonumber
&&\hspace{-0.6in}\approx_{\ep/64+\ep/64+\ep/64+\ep/64}
\sum_{j=1}^{N-1} ue_{1,j}+ue_{1,0}\af(e_{1, N-1})
+\sum_{j=1}^{N}ue_{2,j} +ue_{2,0}\af(e_{2,N})\\
&&\hspace{1.6in}+ue_{1,0}\af(e_{2, N})+ue_{2,0}\af(e_{1, N-1})\\\nonumber
&&\hspace{-0.6in}=\sum_{j=1}^{N-1} ue_{1,j}+ue_{1,0}(\af(e_{1, N-1})+\af(e_{2, N}))
+\sum_{j=1}^{N}ue_{2,j}\\
 &&\hspace{1.6in}+ue_{2,0}(\af(e_{1, N-1})+\af(e_{2,N}))\\
&&\hspace{-0.6in}\approx_{2\ep/2^{10}(N+1)}\sum_{j=0}^{N-1}ue_{1,j}+\sum_{j=0}^N ue_{2,j}=u.
\eneq
Therefore
\beq\label{KI1-22}
\|U(0)\af(U(0)^*)-u\|<\ep/16 +\ep/2^9(N+1).
\eneq
Let $V(t)=U(t)(U(t)^*U(t))^{-1/2}$ for $t\in [0, 1].$ Then $\{V(t): t\in [0,1]\}$ is a continuous path of unitaries in $A$. Since $U(1)=1,$ $V(1)=1.$ Moreover,
\beq
\|V(0)\af(V(0)^*)-u\|<\ep\andeqn \|[V(t),\, x]\|<\ep\tforal x\in {\cal F}
\eneq
and for all $t\in [0,1].$

For the last part (with ${\cal F}=\emptyset$), we note that, in the above,
$u_k\in CU(A),$ $k=0,1,...,N+1.$ Note that in this case, $CU(A)\subset U_0(A).$ Moreover, by \cite{Lnexp2},
${\rm cel}(u_k)\le 2\pi,$ $k=0,1,...,N+1.$ So $V(t)$ can be constructed (without worrying
about the set ${\cal F}$).
\end{proof}

\begin{thm}\label{MT1}
Let $A$ be a unital separable simple amenable \CA\, with $TR(A)\le 1$ which satisfies the UCT. Suppose that  $\af,\, \bt\in \Aut(A)$ have the Rokhlin
property.  Then
the following holds:

(1) If  $\af\circ \bt^{-1}$ is asymptotically inner, then there exists a unitary $u\in U(A)$ and  a strongly asymptotically inner automorphism $\sigma$ such that
$$
\af={\rm Ad}\, u\circ \sigma \circ \bt\circ \sigma^{-1}.
$$

(2) If $\af\circ \bt^{-1}$ is strongly asymptotically inner, then
 there exists a sequence of unitaries $u_n\in U(A)$ and a sequence $\{\sigma_n\}$ of strongly asymptotically inner automorphisms of $A$ such that
$$
\af= {\rm Ad}\, u_n\circ \sigma_n\circ \bt\circ \sigma_n^{-1}\andeqn
\lim_{n\to\infty}\|u_n-1\|=0.
$$

\end{thm}

\begin{proof}
By the assumption there exists a continuous path of unitaries $\{v(t): t\in [0, \infty)\}\subset U(A)$ such that
\beq\label{MT1-1}
\af\circ \bt^{-1}(a)=\lim_{t\to\infty}v_t^*av_t\tforal a\in A.
\eneq
By replacing $\bt$ by ${\rm Ad}\, v_0\circ \bt,$ without loss of generality, we may assume that $v_0=1.$ Note that ${\rm Ad} v_0\circ \bt$ also has the Rokhlin property (see Lemma 3.2 in \cite{BEK}).
 It is important to note  that if $\af\circ \bt^{-1}$ is strongly asymptotically inner, then
one may always assume that $v_0=1$ without replacing $\bt$ by ${\rm Ad}\, v_0\circ \bt.$ In what follows we will only prove part (2) of the theorem.

Let $\{{\cal F}_n\}$ be an increasing sequence of finite subsets
in the unit ball of $A$ such that $\cup_{n=1}^{\infty} {\cal F}_n$ is dense in the unit ball of $A.$ We assume that $1_A\in {\cal F}_1.$
Let $1>\ep>0.$
We choose $t_1>0$ such that
\beq\label{MT1-10}
\af\approx_{\ep/2^4} {\rm Ad}\, v_t\circ \bt\,\,\,{\rm on}\,\,\,  {\cal F}_1\,\,\rforal t>t_1.
\eneq
It follows from  \ref{Ext1} that there is a unitary $v_1'\in U(A)$ such that
\beq\label{MT1-11}
v_{t_1}v_1'\in CU(A)\andeqn \|[\af(x),\, v_1']\|<\ep/2^4\tforal  x\in  {\cal F}_1
\eneq
(by considering $\Gamma: G_{00}\to U(A)/CU(A)$ defined by $\Gamma([1_A])=\overline{v_{t_1}^*},$ where $G_{00}$ is the group generated by
$[1_A]$).
Let $V_1=v_{t_1}v_1' \in CU(A).$ Then
\beq\label{MT1-12}
\af\approx_{\ep/2^3} {\rm Ad}\, V_1\circ \bt\,\,\,{\rm on}\,\,\, {\cal F}_1.
\eneq
It follows from the last part of \ref{KishL1} that there are  unitaries $u_1, w_1\in  U_0(A)$ such that
\beq\label{MT1-13}
V_1=\bt(u_1^*)u_1w_1\andeqn \|w_1-1\|<\ep/4.
\eneq

Put
$$
\bt_1={\rm Ad}\, w_1\circ {\rm Ad}\, u_1\circ \bt\circ {\rm Ad}\, u_1^*
$$
and $v_t^{(1)}=V_1^*v_{t+t_1}.$
Then, by  (\ref{MT1-10}),
\beq\label{MT1-14}
\af=\lim_{t\to\infty}{\rm Ad}\, v_t^{(1)}\circ \bt_1,\,\,\, v_0^{(1)}=(v_1')^*\andeqn\\\label{MT1-14+1}
\af\approx_{\ep/2^4} {\rm Ad}\, v_t^{(1)}\circ \bt_1\,\,\, {\rm on}\,\,\,  {\cal F}_1\tforal t.
\eneq

Let ${\cal F}_1'={\cal F}_1\cup{\rm Ad}\, u_1({\cal F}_1).$
Let $\dt_1>0$ (in place of $\dt$), $\lambda_1>0$ (in place of $\lambda$), let ${\cal G}_1\subset A$ (in place of ${\cal G}$) be a finite
subset, let ${\cal P}_1\subset \underline{K}(A)$ (in place of ${\cal P}$) be a finite subset and let $\{p_{1,1}, p_{1,2},...,p_{1,p(1)}\}$ (in place
of $\{p_1, p_2,...,p_n\}$) be a finite subset of projections in  $A$ required by \ref{KishL1}
for $\ep/4^{3}$(in place of $\ep$) and ${\cal F}_1'$ (in place of ${\cal F}$).
We may assume that ${\cal F}_1'\cup\{p_{1,j}: 1\le j\le p(1)\}\subset {\cal G}_1$  and $\{[p_{1,j}]: 1\le j\le p(1)\}\subset {\cal P}_1.$ Without loss of generality, we may also assume
that $p_1=1_A.$
Let $G_0^{(1)}$ be the subgroup of $K_0(A)$ generated by $\{[p_{1,j}]: 1\le j\le p(1)\}.$ We may assume that $\dt_1$ is sufficiently small  so that
$\Gamma([p_{1,j}])=\overline{\langle (1-p_{1,j})+p_{1,j}W\rangle}$ defines
a \hm\, from $G_0^{(1)}$ into $U_0(A)/CU(A),$  and ${\rm Bott}({\rm id}_A,\, W)|_{{\cal P}_1}$ is well-defined, where $W$ is any unitary
in  $A$ such that $\|[x, \, W]\|<\dt_1$ for all $x\in {\cal G}_1.$ Moreover, we may further assume that
\beq\label{MT1-14+2}
{\rm Bott}({\rm id}_A,\, W_1W_2)|_{{\cal P}_1}={\rm Bott}({\rm id}_A,\, W_1)|_{{\cal P}_1}
+{\rm Bott}({\rm id}_A,\, W_2)|_{{\cal P}_1}\andeqn\\\label{MT1-14+3}
\langle (1-p_{1,j})+p_{1,j}W_1W_2\rangle=\langle (1-p_{1,j})+p_{1,j}W_1\rangle \langle (1-p_{1,j})+p_{1,j}W_2\rangle,
\eneq
provided that
\beq\label{MT1-14+4}
\|[y,\, W_k]\|<\dt_1\tforal y\in {\cal G}_1,\,\,\, k=1,2,.
\eneq
Put $\ep_1=\min\{\ep, \dt_1/2, \gamma_1/2\}.$
Since $\bt_1=\lim_{t\to\infty}{\rm Ad}\,( v_t^{(1)})^* \circ \af,$ we choose $t_2>0$ such that for any $t\ge t_2,$
\beq\label{MT1-15}
\bt_1\approx_{\ep_1/4^{4}} {\rm Ad}\, (v_t^{(1)})^*\circ \af\,\,\,{\rm on}\,\,\, \bt_1^{-1}({\cal G}_1)\cup {\cal F}_1'.
\eneq
It follows from \ref{Ext1} again that there is $v_2'\in U(A)$ such that
\beq\label{MT1-16}
&&v_{t_2}^{(1)}v_2'\in CU(A),\,\,\,\|[x,\, v_2']\|<\ep_1/4^4\tforal x\in \af( \bt_1^{-1}({\cal G}_1)\cup {\cal F}_1')\\\label{MT1-16+1}
&&\hspace{0.2in}\andeqn {\rm Bott}({\rm id}_A,\, v_2')|_{{\cal P}_1}=0.
\eneq
Let $V_2=v_{t_2}^{(1)}v_2'\in CU(A).$  Then
\beq\label{MT1-17}
\bt_1\approx_{\ep_1/4^3} {\rm Ad}\, V_2^*\circ \af\,\,\,{\rm on}\,\,\,\,\bt_1^{-1}({\cal G}_1)\cup {\cal F}_1'.
\eneq
By the last part of \ref{KishL1}, we choose $u_2, w_2\in U_0(A)$ such that
\beq\label{MT1-18}
V_2^*=\af(u_2^*)u_2w_2\andeqn \|w_2-1\|<\ep/4^2.
\eneq

Set $v_t^{(2)}=(v_{t+t_2}^{(1)})^*V_2^*$ and $\af_2={\rm Ad}\, w_2\circ {\rm Ad}\, u_2\circ \af\circ {\rm Ad}\, u_2^*.$
Then, by (\ref{MT1-17}),
\beq\label{MT1-n1}
\bt_1\approx_{\ep/4^3} \af_2\,\,\,{\rm on}\,\,\, {\cal F}_1'\cup \bt_1^{-1}({\cal G}_1).
\eneq
Moreover, by (\ref{MT1-15}),
\beq\label{MT1-19}
v_0^{(2)}=v_2',\,\,\,\bt_1=\lim_{t\to\infty} {\rm Ad}\, (v_t^{(2)})^*\circ \af_2\andeqn\\\label{MT1-19+0}
\bt_1\approx_{\ep_1/4^4} {\rm Ad}\, (v_t^{(2)})^*\circ \af_2\,\,\, {\rm on}\,\,\, \bt_1^{-1}({\cal G}_1)\cup {\cal F}_1'\tforal t\ge 0.
\eneq

Let ${\cal F}_2^{(0)}={\cal F}_1'\cup {\cal F}_2$ and
${\cal F}_2'={\cal F}_2^{(0)}\cup {\rm Ad}\, u_1({\cal F}_2^{(0)})\cup{\rm Ad}\, u_2({\cal F}_2^{(0)}).$

(So far we have not yet used the full strength of \ref{KishL1}. We will use it in the construction of $\bt_3,$ $w_3$ and $u_3$ which we have prepared.  However,
since the construction of $\af_4,$ $w_4$ and $u_4$ are intertwiningly
related to the construction of $\bt_3,$ $w_3$ and $v_3,$ we need to
prepare for the next usage of \ref{KishL1}.)

Let $\dt_2>0$ (in place of $\dt$), $\lambda_2>0$ (in place of $\lambda$), let ${\cal G}_2\subset A$ (in place of ${\cal G}$) be a finite
subset, let ${\cal P}_2\subset  \underline{K}(A)$ (in place of ${\cal P}$) be a finite subset and let $\{p_{2,1}, p_{2,2},...,p_{2,p(2)}\}$ (in place
of $\{p_1, p_2,...,p_n\}$) be a finite subset of projections in $A$ required by \ref{KishL1} for $\ep_1/4^{5}$(in place of $\ep$) and ${\cal F}_2'$\,(in place of ${\cal F}$).  We may assume that $
{\cal P}_1\cup\{[p_{2,j}]:1\le j\le p(2)\}\subset {\cal P}_2$ and ${\cal F}_2'\cup\{p_{2,j}: 1\le j\le p(2)\}\subset {\cal G}_2.$
We may assume that ${\cal G}_2\supset \af_2\circ \bt_1^{-1}({\cal G}_1)\cup {\cal F}_2'.$
Without loss of generality, we may also assume
that $p_{2,1}=1_A.$
Let $G_0^{(2)}$ be the subgroup of $K_0(A)$ generated by $\{[p_{2,j}]: 1\le j\le p(2)\}.$ We may assume that $\dt_2$ is sufficiently small  so that
$\Gamma([p_{2,j}])=\overline{\langle (1-p_{2,j})+p_{2,j}W\rangle}$ defines
a \hm\, from $G_0^{(2)}$ into $U_0(A)/CU(A),$  and ${\rm Bott}({\rm id}_A,\, W)|_{{\cal P}_2}$ is well-defined, where $W$ is any unitary
in  $A$ such that $\|[x, \, W]\|<\dt_2$ for all $x\in {\cal G}_2.$ Moreover, we may further assume that
\beq\label{MT1-19+1}
{\rm Bott}({\rm id}_A,\, W_1W_2)|_{{\cal P}_2}={\rm Bott}({\rm id}_A,\, W_1)|_{{\cal P}_2}
+{\rm Bott}({\rm id}_A,\, W_2)|_{{\cal P}_2}\andeqn\\\label{MT1-19+2}
\langle (1-p_{2,j})+p_{2,j}W_1W_2\rangle=\langle (1-p_{2,j})+p_{2,j}W_1\rangle \langle (1-p_{2,j})+p_{2,j}W_2\rangle,
\eneq
provided that
\beq\label{MT1-19+3}
\|[y,\, W_k]\|<\dt_2\tforal y\in {\cal G}_1,\,\,\, k=1,2,.
\eneq

Let $\ep_2=\min\{\ep_1,\dt_2/2,\lambda_2/2\}.$

 Since  $\af_2=\lim_{t\to\infty} {\rm Ad}\, v_t^{(2)}\circ \bt_1,$ we choose $t_3>0$ such that for $t\ge t_3,$
 \beq\label{MT1-20}
 \af_2\approx_{\ep_2/4^{6}} {\rm Ad}\, v_t^{(2)}\circ \bt_1\,\,\, {\rm on}\,\,\, \af_2^{-1}({\cal G}_2)\cup {\cal F}_2'.
 \eneq
 It follows from (\ref{MT1-n1}) and (\ref{MT1-19+0}), for $x\in \bt^{-1}({\cal G}_1),$
 \beq\label{MT1-21}
 \bt_1(x)\approx_{\ep_1/4^3} \af_2(x) \approx_{\ep_2/4^6} {\rm Ad}\, v_t^{(2)}(\bt_1(x))\tforal t\ge 0.
 \eneq
 In other words, for $t\ge 0,$
 \beq\label{MT1-22}
\| [y, v_t^{(2)}]\|<\dt_1/4^3\tforal y\in {\cal G}_1.
 \eneq
 Note that $v_0^{(2)}=v_2'.$  Therefore, by (\ref{MT1-16+1}) and (\ref{MT1-22}),
 \beq\label{MT1-22+1}
 {\rm Bott}({\rm id}_A,\, v_{t_3}^{(2)})|_{{\cal P}_1}={\rm Bott}({\rm id}_A,\, v_2')|_{{\cal P}_1}=0.
 \eneq
 By the choices of $\dt_1$ and ${\cal G}_1,$ let
 $\Gamma_1: G_0^{(1)}\to U_0(A)/CU(A)$ be the \hm\, defined by
 $\Gamma_1([p_{1,j}])=\overline{\langle (1-p_{1,j})+p_{1,j}v_{t_3}^{(2)}\rangle},$ $j=1,2,...,p(1).$
 Then, since $\{[p_{1,j}]: 1\le j\le p(1)\}\subset {\cal P}_1,$ by (\ref{MT1-22+1}),
 $\Gamma_1$ maps $G_0^{(1)}$ into $U_0(A)/CU(A).$
 It follows from \ref{Ext1} that there is a unitary $v_3'\in U(A)$ such that
 \beq\label{MT1-23}
&& \|[y, \, v_3']\|<\ep_2/4^6\tforal y\in \af_2^{-1}({\cal G}_2)\cup {\cal F}_2' \cup \bt_1(\af_2^{-1}({\cal G}_2)\cup {\cal F}_2' ),\\\label{MT1-23+1}
&& {\rm Bott}({\rm id}_A,\, v_3')|_{{\cal P}_2}=0\andeqn\\\label{MT1-23+2}
 &&{\rm dist}(\overline{\langle (1-p_{1,j})+p_{1,j}(v_3')^*\rangle},\, \Gamma_1([p_{1,j}]))<\lambda_1/2,\,\,\, j=1,2,...,p(1).
 \eneq
Define $V_3=v_3'v_{t_3}^{(2)}.$
Since ${\cal G}_1\subset \bt_1(\af_2^{-1}({\cal G}_2)),$ then
\beq\label{MT1-24}
\af_2\approx_{\ep_2/4^5} {\rm Ad}\, V_3\circ \bt_1\,\,\, {\rm on}\,\,\,\af_2^{-1}({\cal G}_2)\cup {\cal F}_2'\andeqn\\\label{MT1-24+1}
\|[y,\, V_3]\|<\dt_1/2\tforal y\in {\cal G}_1.
\eneq
Let $z_t^{(3)}=v_3'v_{t}^{(2)}.$  Then $z_{t_3}^{(3)}=V_3.$ By (\ref{MT1-22}) and (\ref{MT1-23}),
\beq\label{MT1-25}
\|[y,\, z_t^{(3)}]\|<\dt_1/4\rforal y\in {\cal G}_1\andeqn\,\,\, \rforal t\ge 0.
\eneq
Since $z_0^{(3)}=v_3'v_0^{(2)}=v_3'v_2',$ by (\ref{MT1-25}), (\ref{MT1-14+2}), (\ref{MT1-16+1}) and (\ref{MT1-23+1}),
\beq\label{MT1-26}
{\rm Bott}({\rm id}_A,\, V_3)|_{{\cal P}_1}&=&{\rm Bott}({\rm id}_A,\, z_t^{(3)})|_{{\cal P}_1}
={\rm Bott}({\rm id}_A,\, v_3'v_2')|_{{\cal P}_1}\\ \label{MT1-26+1}
&=&{\rm Bott}({\rm id}_A,\, v_3')|_{{\cal P}_1}
+{\rm Bott}({\rm id}_A,\, v_2')|_{{\cal P}_1}=0.
\eneq
Furthermore, by (\ref{MT1-14+3}),
\beq\label{MT1-27}
\overline{\langle (1-p_{1,j})+p_{1,j}V_3\rangle}&=&\overline{\langle (1-p_{1,j})+p_{1,j}v_3'v_{t_3}^{(2)}\rangle}\\
&=& \overline{\langle (1-p_{1,j})+p_{1,j}v_3'\rangle\langle (1-p_{1,j})+p_{1,j}v_{t_3}^{(2)}\rangle}\\
&=& \overline{\langle (1-p_{1,j})+p_{1,j}v_3'\rangle}\Gamma_1([p_{1,j}]).
\eneq
It follows from (\ref{MT1-23+2}) that
\beq\label{MT1-28}
{\rm dist}(\langle (1-p_{1,j})+p_{1,j}V_3\rangle, CU(A))<\lambda_1,\,\,\, j=1,2,...,p(1).
\eneq
It follows from (\ref{MT1-24+1}), (\ref{MT1-26+1}) and (\ref{MT1-28}),  and by the choice of ${\cal G}_1,$
${\cal P}_1,$ $\dt_1,$ $\lambda_1$ and $\{p_{1,j}: 1\le j\le p(1)\},$ by applying \ref{KishL1},   there are  unitaries $u_3, w_3\in U(A)$ and
a continuous path of unitaries $\{u_3(t): t\in [0,1]\}$ such that
$u_3(0)=u_3,$ $u_3(1)=1_A,$
\beq\label{MT1-29}
V_3=\bt_1(u_3^*)u_3w_3,\,\,\, \|w_3-1\|<\ep/4^3\andeqn \|x-{\rm Ad}\, u_3(t)(x)\|<\ep/4^3\tforal x\in {\cal F}_1'
\eneq
and $t\in [0,1].$
Set $v_t^{(3)}=v_{t+t_3}^{(2)}V_3^*$ and $\bt_3={\rm Ad}\, w_3\circ {\rm Ad}\, u_3 \circ \bt_1\circ {\rm Ad}\, u_3^*.$
Then, by (\ref{MT1-24}),
\beq\label{MT1-n2}
\af_2\approx_{\ep/4^5} \bt_3\,\,\, {\rm on}\,\,\, {\cal F}_2'.
\eneq
Moreover, by (\ref{MT1-20}),
\beq\label{MT1-30}
v_0^{(3)}=v_3',\,\,\,\af_2=\lim_{t\to\infty} {\rm Ad}\, (v_t^{(3)})\circ \bt_3\andeqn\\\label{MT1-30+1}
\af_2\approx_{\ep_2/4^6} {\rm Ad}\, (v_t^{(3)})\circ \bt_3\,\,\, {\rm on}\,\,\, \af_2^{-1}({\cal G}_2)\cup {\cal F}_2'\tforal t\ge 0.
\eneq

We now construct  $\af_4,$ $w_4$ and $u_4.$
To make sure that the process continues, we will also prepare the next step.
Let
$
{\cal F}_3^{(0)}={\cal F}_2'\cup {\cal F}_3.
$ and let
$${\cal F}_3'={\cal F}_3^{(0)}\cup {\rm Ad}\, u_1({\cal F}_3^{(0)})\cup {\rm Ad}\, u_2({\cal F}_3^{(0)})\cup {\rm Ad}\,u_3 ({\cal F}_3^{(0)})
\cup{\rm Ad}\, u_3u_1({\cal F}_3^{(0)}).
$$

Let $\dt_3>0$ (in place of $\dt$), $\lambda_3>0$ (in place of $\lambda$), let ${\cal G}_3\subset A$ (in place of ${\cal G}$) be a finite
subset, let ${\cal P}_3\subset \underline{K}(A)$ (in place of ${\cal P}$) be a finite subset and let $\{p_{3,1}, p_{3,2},...,p_{3,p(3)}\}$ (in place
of $\{p_1, p_2,...,p_n\}$) be a finite subset of projections in  $A$ required by \ref{KishL1} for $\ep/4^{7}$\,(in place of $\ep$) and ${\cal F}_3'$\,(in place of ${\cal F}$). We may assume that $
{\cal P}_2\cup\{[p_{3,j}]:1\le j\le p(3)\}\subset {\cal P}_3$ and ${\cal F}_3'\cup\{p_{3,j}: 1\le j\le p(3)\}\subset {\cal G}_3.$
We may assume that ${\cal G}_3\supset \bt_3\circ \af_2^{-1}({\cal G}_2)\cup {\cal F}_2'.$
Without loss of generality, we may also assume
that $p_{3,1}=1_A.$
Let $G_0^{(3)}$ be the subgroup of $K_0(A)$ generated by $\{[p_{3,j}]: 1\le j\le p(3)\}.$ We may assume that $\dt_3$ is sufficiently small  so that
$\Gamma([p_{3,j}])=\overline{\langle (1-p_{3,j})+p_{3,j}W\rangle}$ defines
a \hm\, from $G_0^{(3)}$ into $U_0(A)/CU(A),$  and ${\rm Bott}({\rm id}_A,\, W)|_{{\cal P}_3}$ is well-defined, where $W$ is any unitary
in  $A$ such that $\|[x, \, W]\|<\dt_3$ for all $x\in {\cal G}_3.$ Moreover, we may further assume that
\beq\label{MT1-31}
{\rm Bott}({\rm id}_A,\, W_1W_2)|_{{\cal P}_3}={\rm Bott}({\rm id}_A,\, W_1)|_{{\cal P}_3}
+{\rm Bott}({\rm id}_A,\, W_2)|_{{\cal P}_3}\andeqn\\\label{MT1-31+1}
\langle (1-p_{3,j})+p_{3,j}W_1W_2\rangle=\langle (1-p_{3,j})+p_{3,j}W_1\rangle \langle (1-p_{3,j})+p_{3,j}W_2\rangle,
\eneq
provided that
\beq\label{MT1-32}
\|[y,\, W_k]\|<\dt_3\tforal y\in {\cal G}_3,\,\,\, k=1,2,.
\eneq

Let $\ep_3=\min\{\ep_2,\dt_3/2,\lambda_3/2\}.$

 Since  $\bt_3=\lim_{t\to\infty} {\rm Ad}\, (v_t^{(3)})^*\circ \af_2,$ we choose $t_4>0$ such that for $t\ge t_4,$
 \beq\label{MT1-33}
 \bt_3\approx_{\ep_2/4^{8}} {\rm Ad}\, (v_t^{(3)})^*\circ \af_2\,\,\, {\rm on}\,\,\, \bt_3^{-1}({\cal G}_3)\cup {\cal F}_3'.
 \eneq
 It follows from (\ref{MT1-24}) and (\ref{MT1-30+1}), for $x\in \af_2^{-1}({\cal G}_2),$
 \beq\label{MT1-34}
 \af_2(x)\approx_{\ep_2/4^5} \bt_3(x) \approx_{\ep_2/4^6} {\rm Ad}\, (v_t^{(3)})^*(\af_2(x))\tforal t\ge 0
 \eneq
 In other words, for $t\ge 0,$
 \beq\label{MT1-35}
\| [y, v_t^{(3)}]\|<\dt_2/4^5\tforal y\in {\cal G}_2.
 \eneq
 Note that $v_0^{(3)}=(v_3')^*.$  Therefore, by (\ref{MT1-35}) and (\ref{MT1-23+2}),
 \beq\label{MT1-36}
 {\rm Bott}({\rm id}_A,\, v_{t_4}^{(3)})|_{{\cal P}_2}={\rm Bott}({\rm id}_A,\, (v_3')^*)|_{{\cal P}_2}=0.
 \eneq
 By the choices of $\dt_2$ and ${\cal G}_2,$ let
 $\Gamma_2: G_0^{(2)}\to U_0(A)/CU(A)$ be the \hm\, defined by
 $\Gamma_2([p_{2,j}])=\overline{\langle (1-p_{2,j})+p_{2,j}v_{t_4}^{(3)}\rangle},$ $j=1,2,...,p(2).$
 Then, since $\{[p_{2,j}]: 1\le j\le p(2)\}\subset {\cal P}_2,$ by (\ref{MT1-36}),
 $\Gamma_2$ maps $G_0^{(2)}$ into $U_0(A)/CU(A).$
 It follows from \ref{Ext1} that there is a unitary $v_4'\in U(A)$ such that
 \beq\label{MT1-37}
&& \|[y, \, v_4']\|<\ep_2/4^8\tforal y\in \af_2(\bt_3^{-1}({\cal G}_3))\cup {\cal F}_3'  ),\\\label{MT1-37+1}
&& {\rm Bott}({\rm id}_A,\, v_4')|_{{\cal P}_2}=0\andeqn\\\label{MT1-37+2}
 &&{\rm dist}(\overline{\langle (1-p_{2,j})+p_{1,j}(v_4')^*\rangle},\, \Gamma_2([p_{2,j}]))<\lambda_2/2,\,\,\, j=1,2,...,p(1).
 \eneq
Define $V_4=v_{t_4}^{(3)}v_4'.$
Since ${\cal G}_2\subset \af_2(\bt_3^{-1}({\cal G}_3)),$ then
\beq\label{MT1-38}
\bt_3\approx_{\ep_2/4^7} {\rm Ad}\, V_4^*\circ \af_2\,\,\, {\rm on}\,\,\,\bt_3^{-1}({\cal G}_3)\cup {\cal F}_3'\andeqn\\\label{MT1-38+1}
\|[y,\, V_4]\|<\dt_2/2\tforal y\in {\cal G}_2.
\eneq
Let $z_t^{(4)}=v_{t}^{(3)}v_4'.$ Then $z_{t_4}^{(4)}=V_4.$ By (\ref{MT1-37}) and (\ref{MT1-35}),
\beq\label{MT1-39}
\|[y,\, z_t^{(4)}]\|<\dt_2/4^2\rforal y\in {\cal G}_2\andeqn\,\,\, \rforal t\ge 0.
\eneq
Since $z_0^{(4)}=v_0^{(3)}v_4'=(v_3')^*v_4',$ by (\ref{MT1-39}),(\ref{MT1-19+1}), (\ref{MT1-23+1}) and (\ref{MT1-37+1}),
\beq\label{MT1-40}
{\rm Bott}({\rm id}_A,\, V_4)|_{{\cal P}_2}&=&{\rm Bott}({\rm id}_A,\, z_t^{(4)})|_{{\cal P}_2}={\rm Bott}({\rm id}_A,\, v_3'v_4')|_{{\cal P}_2}\\\label{MT1-40+1}
&=&{\rm Bott}({\rm id}_A,\, v_3')|_{{\cal P}_2}
+{\rm Bott}({\rm id}_A,\, v_4')|_{{\cal P}_2}=0.
\eneq
Furthermore, by (\ref{MT1-19+2}),
\beq\label{MT1-41}
\overline{\langle (1-p_{2,j})+p_{2,j}V_4\rangle}&=&\overline{\langle (1-p_{2,j})+p_{2,j}v_{t_4}^{(3)}v_4'\rangle}\\
&=& \overline{\langle (1-p_{2,j})+p_{2,j}v_{t_4}^{(3)}\rangle\langle (1-p_{2,j})+p_{2,j}v_4'\rangle}\\\label{MT1-41+}
&=& \Gamma_2([p_{2,j}])\overline{\langle (1-p_{2,j})+p_{2,j}v_4'\rangle}.
\eneq
It follows from (\ref{MT1-37+2}) that
\beq\label{MT1-42}
{\rm dist}(\langle (1-p_{2,j})+p_{2,j}V_4\rangle, CU(A))<\lambda_2,\,\,\, j=1,2,...,p(2).
\eneq
It follows from (\ref{MT1-39}), (\ref{MT1-40+1}) and (\ref{MT1-42}),  and by the choice of ${\cal G}_2,$
${\cal P}_2,$ $\dt_2,$ $\lambda_2$ and $\{p_{2,j}: 1\le j\le p(2)\},$ by applying \ref{KishL1},   there are  unitaries $u_4, w_4\in U(A)$  and
a continuous path of unitaries $\{u_4(t): t\in [0,1]\}$ in $A$ such that
$u_4(0)=u_4$ and $u_4(1)=1_A,$
\beq\label{MT1-43}
V_4^*=\af_2(u_4^*)u_4w_4,\,\,\, \|w_4-1\|<\ep/4^4\andeqn \|x-{\rm Ad}\, u_4(t)(x)\|<\ep/4^5\tforal x\in {\cal F}_2'
\eneq
and $t\in [0,1].$
Set $v_t^{(4)}=(v_{t+t_4}^{(3)})V_4^*$ and $\af_4={\rm Ad}\, w_4\circ u_4 \circ \af_2\circ {\rm Ad}\, u_4^*.$
Then, by (\ref{MT1-38}),
\beq\label{MT1-n3}
\bt_3\approx_{\ep/4^7} \af_4\,\,\,{\rm on}\,\,\,{\cal F}_3'\cup\bt_3^{-1}({\cal G}_3).
\eneq
Then, by (\ref{MT1-33}),
\beq\label{MT1-44}
v_0^{(4)}=v_4',\,\,\,\bt_3=\lim_{t\to\infty} {\rm Ad}\, (v_t^{(4)})^*\circ \af_4\andeqn\\\label{MT1-44+1}
\bt_3\approx_{\ep_2/4^8} {\rm Ad}\, (v_t^{(4)})^*\af_4\,\,\, {\rm on}\,\,\, \bt_3^{-1}({\cal G}_3)\cup {\cal F}_3'\tforal t\ge 0.
\eneq

We repeat this and obtain
\beq\label{MT1-45}
\bt_1, \af_2, \bt_3, \af_4,...,\\
u_1,u_2, u_3(t), u_4(t), u_5(t), u_6(t)...,\\
w_1,w_2,w_3, w_4...,\\
{\cal F}_1', {\cal F}_2', {\cal F}_3', {\cal F}_4',...
\eneq
which satisfy the following:

(i) $\|[x, \, u_i]\|<\ep/4^{i+1},\,\,\,i=1,2,\,\,\, \|[x,\, u_n(t)]\|<\ep/4^{n+1}\tforal x\in {\cal F}_n'$ and $t\in [0,1],$ $n=3,4,...,$

(ii) $\|w_n-1\|<\ep/4^n,$ $n=1,2,...,$

(iii) $\af_{2k}={\rm Ad}\,w_{2k}\circ {\rm Ad}\, u_{2k}\circ \af_{2k-2}\circ {\rm Ad}\, u_{2k}^*,$ $k=1, 2,3,...,$ (with $\af_0=\af$),

(iv) $\bt_{2k+1}={\rm Ad}\, w_{2k+1} \circ {\rm Ad}\, u_{2k+1}\circ \bt_{2k-1}\circ {\rm Ad}\, u_{2k+1}^*,$ $k=0,1,2,...,$ with $\bt_{-1}=\bt,$

(v) $\af_{2n}\approx_{\ep/4^{2n+1}}\bt_{2n+1}\,\,\,{\rm on}\,\,\,{\cal F}_{2n}',$ $n=1,2,...,$

(vi) $\bt_{2n-1}\approx_{\ep/4^{2n+1}}\af_{2n}\,\,\,{\rm on}\,\,\,{\cal F}_{2n-1}',$ $n=1,2,....$

(vii) ${\cal F}_1'={\cal F}_1\cup {\rm Ad}\, u_1\circ ({\cal F}_1),$
  ${\cal F}_2'={\cal F}_2^{(0)}\cup {\rm Ad}\, u_1({\cal F}_2^{(0)})\cup
  {\rm Ad}\, u_2({\cal F}_2^{(0)})$
  $$
  {\cal F}_{2n-1}'={\cal F}_{2n-1}^{(0)}\cup \cup_{j=1}^{2n-1}{\rm Ad}\, u_j({\cal F}_{2n-1}^{(0)})\cup\cup_{j=2}^{n} {\rm Ad}\, u_{2j-1}u_{2j-3}\cdots u_1({\cal F}_{2n-1}^{(0)})
  $$
  $$
  {\cal F}_{2n}'={\cal F}_{2n}^{(0)}\cup \cup_{j=1}^{2n}{\rm Ad} u_j({\cal F}_{2n}^{(0)})\cup \cup_{j=2}^n{\rm Ad}\, u_{2j}u_{2(j-1)}\cdots u_2({\cal F}_{2n}^{(0)}),
  $$
  with ${\cal F}_n^{(0)}={\cal F}_{n-1}'\cup{\cal F}_n,$ $n=1, 2,....$

Define $\gamma_{k,0},\gamma_{k,1}\in \Aut(A)$ by
\beq\label{MT1-46}
\gamma_{k,0}={\rm Ad}\, u_{2k} u_{2(k-1)}\cdots u_2\andeqn \gamma_{k,1}={\rm Ad}\, u_{2k+1}u_{2k-1}\cdots u_1.
\eneq
Fix $k_0\ge 1,$ if $k\ge k_0,$
\beq\label{MT1-46+1}
\gamma_{k,0}({\cal F}_{k_0})\subset {\cal F}_{k+1}'.
\eneq
Since $\|[x,\, u_n]\|<\ep/4^n$ for $x\in {\cal F}_n'$ and $0<\ep<1,$
we have if $m>k,$
\beq\label{MT1-47}
\|\gamma_{m,0}(x)-\gamma_{k,0}(x)\|\le \sum_{j=k}^m \ep/4^j<\ep/4^{k-1}
\eneq
for all $x\in {\cal F}_{k_0}.$  Since $\|\gamma_{n,0}\|$ is bounded and
$\cup_{k=1}^{\infty}{\cal F}_k$ is dense in the unit ball, it follows that
$\{\gamma_{k,0}(x)\}$ is Cauchy for each $x\in A.$
This gives a linear map $\gamma_0$ on $A$ such that
 $$
 \gamma_0(x)=\lim_{n\to\infty}\gamma_{n,0}(x)\tforal x\in A.
 $$
Since each $\gamma_{n,0}$ is an automorphism, $\gamma_0$ is a unital  injective \hm. Exactly the same reason shows that  $\gamma_{k,0}^{-1}$ converges point-wisely to another
unital injective \hm\, $\gamma_0'.$ It is easy to check that
$\gamma_0'=\gamma_0^{-1}.$ So $\gamma_0$ is an automorphism.

Define
$\{W(t): t\in [2, \infty)\}$ as follows:
$$
W(2)=u_2,\,\,\,W(t)=u_{2k}(k+1-t)u_{2(k-1)}u_{2(k-2)}\cdots u_2\tforal t\in (k, k+1],\,\,\,k=2,3,....
$$
Note $\{W(t): t\in [1, \infty)\}$ is a continuous path of unitaries in $A$ with $W(2)=u_2\in U_0(A).$ Since $u_2\in U_0(A),$ we may assume that $W(t)$ defined on $[1, \infty),$ with $W(1)=1_A.$
Then
\beq\label{MT1-48-1}
\gamma_0(x)=\lim_{t\to\infty} {\rm Ad}\, W(t) (x)\tforal x\in A.
\eneq
Therefore $\gamma_0$ is strongly asymptotically inner.
Similarly, $\gamma_1(x)=\lim_{n\to\infty}\gamma_{n,1}(x)$ for all $x\in A$ defines another strongly asymptotically inner automorphism on $A.$

Define $W_{2,0}=w_2$  and
\beq\label{MT1-48}
W_{2k,0}=w_{2k}u_{2k}W_{2k-2,0}u_{2k}^*,\,\,\,k=2,3,....
\eneq
We estimate that
\beq\label{MT1-49}
\|W_{2k,0}-1\| &\le & \|w_{2k}-1\|+\|u_{2k}W_{2k-2}u_{2k}^*-1\|\\
 &<& \ep/4^{2k}+\|W_{2k-2}-1\|.
\eneq
It follows from the induction that
\beq\label{MT1-50}
\|W_{2k,0}-1\|<\sum_{j=1}^k \ep/4^{2j},\,\,\,k=1,2,....
\eneq
It follows that $\lim_{k\to\infty} W_{2k,0}=W_0$ is a unitary in $A$ with
\beq\label{MT1-51}
\|W_0-1\|<\ep/2.
\eneq
Define $W_{1,1}=w_1$ and
\beq\label{MT1-52}
W_{2k+1,1}=w_{2k+1}u_{2k+1}W_{2k-1,1}u_{2k+1}^*,\,\,\,k=1,2,....
\eneq
As above, $W_1=\lim_{k\to\infty} W_{2k+1,1}$ is a unitary in $A$ with
\beq\label{MT1-53}
\|W_1-1\|<\ep/2.
\eneq

From (iii) above,
\beq\label{MT1-54}
\af_{2k}&=&{\rm Ad}\, w_{2k}\circ {\rm Ad}\, u_{2k}\circ \af_{2k-2}\circ {\rm Ad}\, u_{2k}^*\\\label{MT1-54+1}
&=&{\rm Ad} W_{2k,0} \circ \gamma_{k,0}\circ \af\circ \gamma_{k,0}^{-1}.
\eneq
Similarly, by (iv),
\beq\label{MT1-55}
\bt_{2k+1}&=& {\rm Ad}\, w_{2k+1}\circ  {\rm Ad}\,  u_{2k+1}\circ \bt_{2k-1}\circ {\rm Ad}\, u_{2k-1}^*\\\label{MT1-55+1}
&=& {\rm Ad} W_{2k+1,1}\circ \gamma_{k,1}\circ \bt\circ \gamma_{k,1}^{-1}.
\eneq
It follows from (v) and (vi) that the sequence
$\bt_1, \af_2,\bt_3,...$ converges point-wise. By (\ref{MT1-54}) and
(\ref{MT1-55+1}),
the limit of the sequence is
\beq\label{MT1-56}
{\rm Ad}\, W_0\circ \gamma_0\circ \af\circ \gamma_0^{-1}=
{\rm Ad}\, W_1\circ \gamma_1\circ \bt\circ \gamma_1^{-1}.
\eneq
Therefore
\beq\label{MT1-57}
\af&=&\gamma_0^{-1}\circ {\rm Ad}\, W_0^*\circ {\rm Ad}\, W_1\circ \gamma_1\circ \bt\circ \gamma_1^{-1}\circ \gamma_0\\
&=& {\rm Ad}(\gamma_0^{-1}(W_0^*W_1))\circ \gamma_0^{-1}\circ \gamma_1\circ \bt\circ \gamma_1^{-1}\circ \gamma_0.
\eneq
Note that, since $\gamma_0$ and $\gamma_1$ are strongly asymptotically inner,
so is $\gamma_1^{-1}\circ \gamma_0.$
Finally, let $w=\gamma_0^{-1}(W_0^*W_1),$
$\sigma=\gamma_0^{-1}\circ \gamma_1.$ Then
$\sigma^{-1}=\gamma_1^{-1}\circ \gamma_0$ and
\beq\label{MT1-58}
\|w-1\|=\|\gamma_0^{-1}(W_0^*W_1)-1\|=\|W_0^*W_1-1\|<\ep/2+\ep/2=\ep.
\eneq
 Moreover,
\beq\nonumber
\af={\rm Ad}\, w\circ \sigma\circ \bt\circ  \sigma^{-1}.
\eneq
The theorem follows.
\end{proof}

\section{Projections in $M_n\otimes {\cal Z}$ }

The following is a refinement of Theorem 5.4 of \cite{Lntr=1}.

\begin{lem}\label{appdiv}
Let $A$ be a unital simple \CA\, with $TR(A)\le 1.$
Let $\ep>0,$ $\eta>0,$ ${\cal F}\subset A$ be a finite subset,
$N\ge 1$ be an integer.
There is a projection $p\in A$ and a finite dimensional
\SCA\, $B\cong M_N$
with $1_B=p$ and a finite subset ${\cal F}_1\subset A$ such that
\beq\label{appdiv-1}
\|[p,\,x]\|<\ep\tforal x\in {\cal F},\,\,\,yb=by\tforal y\in {\cal F}_1\andeqn b\in B,\\
pxp\in_{\ep} {\cal F}_1\tforal x\in {\cal F}\andeqn
\tau(1-p)<\eta\tforal \tau\in T(A).
\eneq

\end{lem}

\begin{proof}
Let $N_1$ be an integer such that $N/N_1<\eta/4.$  Without loss of generality, to simplify the notation, we may
assume that ${\cal F}$ is in the unit ball of $A.$
Since $TR(A)\le 1,$ as (in the proof of) Theorem 5.4 of \cite{Lntr=1},
there exists a projection $p_1\in A$ and a finite dimensional \SCA\, $C\cong \bigoplus_{j=1}^JM_{r(j)}$
with $r(j)\ge N\cdot N_1$ and with $1_{C}=p_1$ such that
\beq\label{appdiv-2}
\|[p_1,\, x]\|<\ep/16 (N+1)^2\rforal x\in {\cal F}, \tau(1-p)<\eta/4\rforal \tau\in T(A),\\\label{appdiv-2+1}
\|[p_1xp_1,\,c]\|<\ep/16(N+1)^2 \rforal x\in {\cal F}\andeqn c\in C\,\,\,{\rm with}\,\,\, \|c\|\le 1.
\eneq
Write $r(j)=R_jN+r_j,$ where $R_j\ge N_1$ and $0\le r_j<N$ are integers, $j=1,2,...,J.$
Thus $C$ has a projection $p$ such that
$pCp\cong M_N\otimes C_1,$ where $C_1$ is a finite dimensional \SCA\, and
\beq\label{appdiv-3}
t(p)>1-\eta/4\rforal t\in T(C).
\eneq
Since $p\in C,$ by (\ref{appdiv-2}) and (\ref{appdiv-2+1}),
\beq\label{appdiv-3+1}
\|[p,\, x]\|<\ep/8(N+1)^2 \tforal x\in {\cal F}\andeqn\\\label{appdiv-3+2}
\|[pxp,\,c]\|<\ep/8(N+1)^2\tforal x\in {\cal F}\andeqn \tforal c\in pCp.
\eneq
Let $\{e'_{i,j}: 1\le i,j\le N\}$ be a matrix unit for $M_N$ and let $e_{i,j}=e'_{i,j}\otimes 1_{C_1}.$
Define, for each $x\in {\cal F},$
\beq\label{appdiv-5}
\Phi(pxp)=\sum_{i=1}^Ne_{i,i}pxpe_{i,i}.
\eneq
By (\ref{appdiv-3+2}), we estimate that
\beq\label{appdiv-5+1}
\|\Phi(pxp)-pxp\|<N(N-1)\ep/8(N+1)^2<\ep/8
\eneq
We also have
\beq\label{appdiv-7}
\Phi(pxp)b=b\Phi(pxp)\rforal x\in {\cal F}\andeqn b\in M_N\otimes \C \cdot 1_{C_1}\cong M_N.
\eneq
Now let ${\cal F}_1=\{\Phi(pxp): x\in {\cal F}\}.$ Then
\beq\label{appdiv-8}
pxp\in_{\ep} {\cal F}_1.
\eneq
Finally, we note that (\ref{appdiv-3}) and (\ref{appdiv-2+1}) imply
that
$$
\tau(p)>1-\eta\tforal \tau\in T(A).
$$

\end{proof}

Note that the lemma holds for unital simple \CA s which are tracially approximately divisible.

\begin{lem}\label{Extpro}
Let $a\in {\cal Z}$ be a non-zero element with $0\le a\le 1.$
For any $\ep>0$ and $1/2>\eta>0,$   there is an integer $N\ge 1$ satisfying the following:
If $n\ge N,$ there is a projection $E\in M_n$ and a projection $p\in A=EM_nE\otimes {\cal Z}$ such that
\beq\label{Extp-1}
{{\rm rank} E\over{n}}>1-\ep,\,\,\,\|ph_{\eta}(E\otimes a)-h_{\eta}(E\otimes a)\|<\ep,\\
p\in \overline{(1\otimes a)A(1\otimes a)}\andeqn
(t\otimes \tau)(p)>d_\tau(a)-\ep,
\eneq
where $\tau$ is the unique tracial state on ${\cal Z}$ and
$t$ is the normalized trace on $M_n.$
\end{lem}

\begin{proof}

Let $1>\ep>0$  and let  $1/2>\eta>0$ be any positive number.
Define $B=Q\otimes {\cal Z},$ where $Q$ is the UHF-algebra with
$(K_0(Q), K_0(Q)_+, [1_Q])=(\Q, \Q_+, 1).$  Then $B\cong Q.$
In particular, $\overline{(1_Q\otimes a)B(1_Q\otimes a)}$ has real rank zero.
Therefore, there is a projection
$q\in \overline{(1_Q\otimes a)B(1_Q\otimes a)}$ such that
\beq\label{Extp-1+1}
\|qh_\eta(1\otimes a)-h_\eta(1\otimes a)\|<\ep/2\andeqn\\
(t_0\otimes \tau)(q)>d_{t_0\otimes \tau}(1\otimes a)-\ep/16=d_\tau(a)-\ep/16,
\eneq
where $t_0$ is the unique tracial state on $Q$ (see \ref{Dfunc} and \ref{dfdimfunc}
for the definition of $h_\eta,$  $d_{t_0\otimes \tau}$ and $d_{\tau}$).
Write $Q=\lim_{n\to\infty}(M_{n!},\imath_n).$ There exists $N_0\ge 1$ such that there is a projection $p'\in M_{N_1}\otimes {\cal Z}$ such that
\beq\label{Extp-2}
\|p'-q\|<\ep/4 \andeqn 1/N_0<\ep/16,
\eneq
where $N_1=N_0!$ and where we identify $M_{N_1}$ as a unital \SCA\, of $Q.$
Since $q\in \overline{(1_Q\otimes a)B(1_Q\otimes a)},$ there is a continuous
function $f\in C_0((0,1])$ with $0\le f\le 1$ such that
\beq\label{Extp-3}
\|qf(1_Q\otimes a)-q\|<\ep/8.
\eneq
It follows that
\beq\label{Extp-4}
\|p'f(1_Q\otimes a)-p'\|<\ep/4.
\eneq
View $p'\in M_{N_1}\otimes {\cal Z},$ we also have
\beq\label{Extp-5}
\|p'f(1_{M_{N_1}}\otimes a)-p'\|<\ep/4\andeqn \|p'h_\eta(1_{M_{N_1}}\otimes a)-h_\eta(1_{M_{N_1}}\otimes a)\|<\ep.
\eneq
Put $b=f(1_{M_{N_1}}\otimes a).$ Then
\beq\label{Extp-6}
\|bp'b-p'\|<\ep/2.
\eneq
Then (by Lemma 2.5.3 in \cite{Lnbk}) there exists a projection
$p_0\in \overline{b(M_{N_1}\otimes {\cal Z})b}$ such that
\beq\label{Extp-7}
\|p_0-p'\|<\ep.
\eneq
Since $\ep<1,$ by (\ref{Extp-7} and (\ref{Extp-2}),
\beq\label{Extp-8}
(t_1\otimes \tau)(p_0)=(t_1\otimes \tau)(p')=(t_0\otimes \tau)(p')=(t_0\otimes \tau)(q)>d_\tau(a)-\ep/16,
\eneq
where $t_1$ is the tracial state on $M_{N_1}.$  Note that $b\in \overline{(1_{M_{N_1}}\otimes a)B(1_{M_{N_1}}\otimes a)}.$

Let $N=(N_0+1)!.$
If $n\ge N,$ we may write
$n=dN_0+r,$ where $d\ge N_0!$ and $0\le r<N_0.$
There is a projection $E\in M_n$ with ${\rm rank}E=dN_0.$
Then $EM_nE\cong M_d\otimes M_{N_0}.$  Moreover
$$
{{\rm rank E}\over{n}}>1-{r\over{n}}>1-\ep/16>1-\ep.
$$
Put $A=M_n\otimes {\cal Z}.$ The proof above shows that
there exists a projection $p\in \overline{(E\otimes a)A(E\otimes a)}$ such that
\beq\label{Extp-9}
(t_E\otimes \tau)(p)>d_{\tau}(a)-\ep/16,
\eneq
where $t_E$ is the tracial state on $EM_nE.$
Clearly $p\in \overline{(1_{M_n}\otimes a)A(1_{M_n}\otimes a)}.$
We compute that
\beq\label{Extp-10}
(t\otimes \tau)(p)>(d_{\tau}(a)-\ep/16)(1-r/n)\ge d_{\tau}(a)-\ep,
\eneq
where $t$ is the tracial state on $M_n.$ Note, by (\ref{Extp-5}),
\beq\label{Extp-11}
\|ph_\eta(E\otimes a)-h_\eta(E\otimes a)\|<\ep.
\eneq

\end{proof}

The following is not used in its full strength. When $A$ has sufficiently many projections,
for example, $A$ has real rank zero,
a version of the following is proved in \cite{Sato1}(Lemma 3.5). An early version of this can be found in
Lemma 2.8 of \cite{OP} and may be traced back to the proof of Theorem 4.5 of \cite{Ki0}.

\begin{lem}\label{LComp}
Let $A$ be a unital simple \CA\, with $TR(A)\le 1$ and let
$e,\, f\in A$ be two projections such that
$2\tau(e)<\tau(f)$ for all $\tau\in T(A).$ Then, for any $1>\ep>0,$ there exists
a projection $p\le f$ such that
\beq\label{LC-0}
\|pe\|<\ep,\,\,\, \tau(p)>\tau(e)\tand \tau(p)>\tau(f)-\tau(e)-\ep\tforal \tau\in T(A).
\eneq
In particular, there is
a partial isometry $w\in A$ such that
\beq\label{LC-1}
w^*w=e,\,\,\, ww^*\le f\andeqn \|w^2\|<\ep.
\eneq
Moreover,  one may requires that
\beq\label{LC-1+}
\|ph_{\ep^2/32}(fef)\|<\ep.
\eneq

\end{lem}

\begin{proof}
Since $TR(A)\le 1$ and $A$ has strict comparison (see 4.7 of \cite{Lntr=1}),
we have two mutually orthogonal and mutually equivalent projections $f_1, f_2\in A$ such that
$f_1+f_2\le f$ and $f_1$ is equivalent to $e.$
Since $A$ has property (SP) (see  3.2 and  of \cite{Lntr=1}),
there are mutually orthogonal and mutually equivalent non-zero projections
$f_3, f_4, f_5, f_6\in (f-f_1-f_2)A(f-f_1-f_2).$
Put $f_0=f-\sum_{i=1}^6f_i.$

Let $N\ge 1$ be an integer such that $1/N<\min\{\max_{\tau\in T(A)}\{\tau(f_3)\},\ep\}/8$ for all $\tau\in T(A).$
There are partial isometries $v_1, v_2,  v_3, u_1, u_2,...,u_5\in A$ such that
\beq\label{LC-3-1}
&&v_1^*v_1=e, v_1v_1^*=f_1,\,\, v_2^*v_2=e,
 v_2v_2^*=f_2, v_3^*v_3=f_1, v_3v_3^*=f_2,\\
&&u_j^*u_j=f_j\andeqn u_ju_j^*=f_{j+1},\,\,\, j=3,4,5.
\eneq

Choose $\eta>0$  such that
\beq\label{LC-3n+1}
\|h_{\ep^2/32}(a)-h_{\ep^2/32}(b)\|<\ep/16
\eneq
provided that $0\le a,\, b \le 1$ and $\|a-b\|<\eta,$ where $ a,\, b\in A.$

If $ef=0,$ the lemma follows easily.  So we assume that $fef\not=0.$
Let $\dt=\min\{\ep^2/128, \|fef\|/4, 1/4\}.$
Since $A$ is simple, there are
$x_1,x_2,...,x_n\in A$ such that
\beq\label{LComp-n1}
\sum_{j=1}^n x_i^*h_\dt(fef)x_i=1.
\eneq
Let $M=\max\{\|x_i\|: 1\le i\le n\}$ and
let
$$
{\cal F}=\{f,e,fe, ef,  fef, f_1, ...,f_6,v_1, v_2, v_3, u_1,...,u_5\}\cup\{x_i, x_i^*:1\le i\le n \},
$$
Fix  $0<\eta'\le \min\{\ep/2, \dt/16, \eta/2\}.$
Since $TR(A)\le 1,$  using ${\cal F},$ one obtains   a projection $P\in A$ and a \SCA\,
$C\cong \bigoplus_{j=1}^J M_{r(j)}(C((I_j)))$ with $1_C=P$ and with
$r(j)\ge N,$
where $I_j$ is either a point or $I_j=[0,1],$ such that
\beq\label{LC-3}
\tau(1-P)<\min\{\tau(f_3),\ep\}/16\tforal \tau\in T(A),\\\label{LC-3+0}
\|xP-Px\|<\eta'/96(64Mn)\tforal x\in {\cal F},\\\label{LC-3+}
 \|f_i-(f_i'+f_i'')\|<\eta'/96(16Mn),\,\,\,\|e-(e'+e'')\|<\eta'/96(16Mn),\\
e',f_i'\in C, \,\,\, e'',f_i'', f''\in (1-P)A(1-P), \,\,\,i=0,1, ...,6,
\eneq
where $e', f_i', f_i'', e''$ are non-zero projections, and
$f_i'f_j'=0,$ $f_i''f_j''=0,$ if $i\not=j.$  We may further assume
that
\beq\label{LC-4}
\hspace{-0.3in}t(e')=t(f_i'),\,\,\, t(f_1')=t(f_2'), t(f'_3)=t(f_4')=t(f_5')=t(f_6') \tforal t\in T(C),\,\,\, i=1,2.
\eneq
Note that (\ref{LC-3+})  and (\ref{LC-4}) can be hold since $f_i, v_i, u_i\in {\cal F}.$ 

Put $f'=\sum_{i=0}^6f_i' $ and 
put $a=f'e'f'.$
Note that
\beq\label{LC-4+}
\|\sum_{i=1}^ny_i^*h_\dt(a)y_i-P\|<\dt/16
\eneq
for some $y_i\in C,$ $i=1,2,...,n,$ if we choose $\eta'$ sufficiently small
(depending only additionally on $\dt$).

                  We see
                            that $h_{\dt}(a)$ is non-zero at each point of $I_j$ (in each summand of $C$).
                            Let $R(j)$ be the least rank of $h_{\dt}(a)$ and $R(j)'$ be the largest  rank of $h_{\dt/2}(a)$
in $f'M_{r(j)}(I_j)f',$ $j=1,2,...,J.$   Note the rank of $e'$ is at least  $R(j)'$ in $j$-th summand.
So $f_2'$ has rank at most $R(j)'$ in $j$-th summand.

It follows from Lemma C of \cite{BDR} that there is a projection $q\in f'Cf'$ such that
\beq\label{LC-5}
qh_\dt(a)=h_\dt(a)
\eneq
and the rank of $q$ at each summand $f'M_{r(j)}(I_j)f'$ is at most
$R(j)'+1.$
Put $p_1=f'-q\in f'Cf'.$ Then
\beq\label{LC-6}
p_1h_\dt(a)=0.
\eneq
Moreover
\beq\label{LC-6+}
\tau(p_1) &>&\tau(f')-\tau(e')-1/N>\tau(f)-\tau(e')-\ep/16-\tau(1-P)\\\label{LC-6++}
&>& \tau(f)-\tau(e)-\ep/4\tforal \tau\in T(A).
\eneq

We also have
\beq\label{LC-7}
\hspace{-0.4in}\|h_\dt(a)f'e'-f'e'\|^2\\
&=&\|h_\dt(a)f'e'f'h_\dt(a)-h_\dt(a)f'e'f'-f'e'f'h_\dt(a)+f'e'f'\|\\
&\le& 2\|h_\dt(f'e'f')f'e'f'-f'e'f'\|\le 2\dt.
\eneq
It follows that
\beq\label{LC-8}
\|p_1e'\|&=&\|qf'e'-f'e'\|\\
&\le & \|qf'e'-qh_\dt(a)f'e'\|+\|qh_\dt(a)f'e'-h_\dt(a)f'e'\|\\
&&+\|h_\dt(a)f'e'-f'e'\|\\\label{LC-8+}
&<& \sqrt{2\dt}+\sqrt{2\dt}<\ep/64.
\eneq
Therefore there is a projection
$e_1'\in (1-p_1)C(1-p_1)$ such that
\beq\label{LC-8n}
\|e_1'-e'\|<\ep/16.
\eneq
Now we count rank of $p_1.$ Since $f'=\sum_{i=0}^6f_i'$ and
the rank of $q$ in $j$-th summand is not more than that of $f_1',$
$f_2'$ has the same rank as that of $f_1'$ which has the rank as that of $e'$ at each
summand of $C,$
$p_1=f'-q$ has rank great than that of $e_1'$ at each summand of $C.$
 So there is a partial isometry
 $w_1\in C$ such that
\beq\label{LC-9}
w_1^*w_1=e_1'\andeqn w_1w_1^*\le p_1.
\eneq
Since $e_1'\in (1-p_1)A(1-p_1),$ it follows  that
\beq\label{LC-9+}
\|w_1w_1\|=\|w_1e_1'p_1w_1\|=0
\eneq
Note that
\beq\label{LC-11}
\tau(1-P)<\tau(f_5')\le \tau(p_1-w_1w_1^*)\tforal \tau\in T(A).
\eneq
There is a partial isometry $w_2\in A$ such that
\beq\label{LC-12}
w_2^*w_2=e''\andeqn w_2w_2^*\le p_1-w_1w_1^*.
\eneq
Since $p_1-w_1w_1^*\in PAP,$ we have  $w_2^2=0.$ Moreover
\beq\label{LC-12+1}
w_1w_2=w_1e_1'(p_1-w_1w_1^*)w_2=0\andeqn w_2w_1=w_2e''p_1w_1=0.
\eneq
Put $w'=w_1+w_2.$ Then, by (\ref{LC-9+}) and (\ref{LC-12+1})
\beq\label{LC-12+2}
(w')^*w'=e_1'+e'',\,\,\,w'(w')^*\le p_1\andeqn (w')^2=0,
\eneq
Since
\beq\label{LC-12+}
\|fp_1-p_1\|<\ep/16,
\eneq
there is a projection $p\in A$ such that $p\le f$ and
\beq\label{LC-13}
\|p_1-p\|<\ep/8.
\eneq
It follows from (\ref{LC-12+2}), (\ref{LC-3+}) and (\ref{LC-13}) that there exists a partial isometry $w\in A$ such that
\beq\label{LC-13+}
w^*w=e,\,\,ww^*\le f\andeqn \|w^2\|<\ep.
\eneq
From  (\ref{LC-9}) and (\ref{LC-11}),
we have
\beq\label{LC-14}
\tau(e)=\tau(e')+\tau(e'')=\tau(e_1')+\tau(e'')<\tau(w_1w_1^*)+\tau(p_1-w_1w_1^*)
=\tau(p)
\eneq
for all $\tau\in T(A).$
We also have, (using  (\ref{LC-6++} for (\ref{CL-15+})),
\beq\label{LC-15}
\|pe\|<\ep/8+\| p_1(e'+e'')\|=\ep/8 +\|p_1e'\|<\ep/8+\ep/16<\ep\andeqn\\\label{CL-15+}
\tau(p)=\tau(p_1)>\tau(f)-\tau(e)-\ep/4\tforal \tau\in T(A).
\eneq
Finally, by the choice of $\eta,$ one computes
that
$$
ph_{\ep^2/32}(fef)\approx_{\ep/8} p (h_{\ep^2/32}(a)+h_{\ep^2/32}(f''e''f''))\approx_{\ep/8}p_1h_{\dt}(a)=0.
$$
This completes the proof.
\end{proof}

The following probably holds in much more general setting. Recall that a sequence
$\{a_n\}$ is said to be central in $A,$ if $\lim_{n\to\infty}\|a_nb-ba_n\|=0$
for all $b\in A.$

\begin{prop}\label{PSI}
Let $A$ be a unital separable amenable  simple \CA\, with $TR(A)\le 1.$
Let $\{e_n\}$ and $\{p_n\}$ be two central sequences  of projections for $A$ satisfying
\beq\label{PSI-1}
\lim_{n\to\infty} \sup_{\tau\in T(A)}\tau(e_n)=0\andeqn
\liminf_{n\to\infty} \min_{\tau\in T(A)} \tau(p_n)>0.
\eneq
Then there exists a central sequence of partial isometries $\{w_n\}$ in $A$ such that
\beq\label{PSI-2}
\lim_{n\to\infty}\|w_n^*w_n-e_n\|=0,\,\,\,\lim_{n\to\infty}\|p_nw_n-w_n\|=0\andeqn \lim_{n\to\infty} \|w_n^2\|=0
\eneq
\end{prop}

\begin{proof}
We first note, by Corollary 8.4 of \cite{LnlocAH}, since $TR(A)\le 1,$ $A$ is ${\cal Z}$-stable. Let $\{{\cal F}_n\}$ be an increasing sequence of finite subsets of the unit ball of $A$
such that $\cup_{n=1}^{\infty} {\cal F}_n$ is dense in the unit ball of $A.$
Since $\{p_ne_np_n\}$ is a central sequence for $A,$ for each $k,$  there exists $n(k)$ such that
\beq\label{PSI-3}
\|h_{1/32k^2}(p_ne_np_n)x-xh_{1/32k^2}(p_ne_np_n)\|<1/k\tforal x\in {\cal F}_k
\eneq
and for all $n\ge n(k),$
$k=1,2,....$ We may assume that $n(k+1)>n(k),$ $k=1,2,....$
Define $a_1=1/32,$ $a_j=1/32k^2,$ if $n(k)\le j< n(k+1),$ $k=1,2,....$
Then $\{h_{a_n}(p_ne_np_n)\}$ is also a central sequence for $A.$

For each $n,$ define $y_n=p_n-h_{a_n}(p_ne_np_n).$ Then $0\le y_n\le 1,$ $n=1,2,...,$ and $\{y_n\}$ is a central sequence for $A.$
By applying \ref{LComp}, one has a projection $q_n\le  p_n$ such
that
\beq\label{PSI-4}
\|q_nh_{a_n}(p_ne_np_n)\|<1/k \andeqn \tau(q_n)\ge \tau(p_n)-\tau(e_n)-1/k
\tforal n\ge n(k),
\eneq
$k=1,2,....$
We also have
\beq\label{PSI-4+1}
\lim_{n\to\infty}\|h_{a_n}(p_ne_np_n)p_ne_n-p_ne_n\|^2=0.
\eneq
It follows that
\beq\label{PSI-4+2}
\lim_{n\to\infty}\|y_ne_n\|&=&
\lim_{n\to\infty}\|(p_n-h_{a_n}(p_ne_np_n)e_n\|\\\label{PSI-4+3}
&=&\lim_{n\to\infty}\|p_ne_n-h_{a_n}(p_ne_np_n)p_ne_n\|=0.
\eneq
For $n\ge n(k),$
\beq\label{PSI-5}
\|q_ny_n-q_n\|=\|q_nh_{a_n}(p_ne_np_n)\|<1/k, \,\,\,k=1,2,....
\eneq
Therefore,
for any integer $m\ge 1,$ if $n\ge n(k),$
\beq\label{PSI-6}
\|y_n^mq_ny_n^m-q_n\|<2m/k,\,\,\,k=1,2,....
\eneq
Thus, if $n\ge n(k),$
\beq\label{PSI-7}
\tau(y_n^m)\ge \tau(y_n^{2m})\ge \tau(y_n^mq_ny_n^m)>\tau(q_n)-2m/k,
\eneq
$k=1,2,....$
It follows that
\beq\label{PSI-8}
\inf_{m\in \N}\lim_{n\to\infty} \min_{\tau\in T(A)}\tau(y_n^m) &\ge & \inf_{m\in \N}\lim_{n\to\infty}\min_{\tau\in T(A)}\tau(q_n)\\
&=&\inf_{m\in \N}\lim_{n\to\infty}\min_{\tau\in T(A)}\tau(p_n)>0
\eneq
By Theorem 4.2 of \cite{MS}, since $A$ is ${\cal Z}$-stable, there is a central sequence
$\{s_n\}$ for $A$ such that
\beq\label{PSI-9}
\lim_{n\to\infty}\|s_n^*s_n-e_n\|=0\andeqn
\lim_{n\to\infty}\|y_ns_n-s_n\|=0.
\eneq
It is standard that there exists a central sequence of partial isometries
$\{w_n\}$ of $A$ such that
\beq\label{PSI-10}
w_n^*w_n=e_n\andeqn w_nw_n^*\le p_n\andeqn \|w_n-s_n\|=0.
\eneq
Therefore $\{w_n\}$ is a central sequence of $A.$
By  (\ref{PSI-9}) and (\ref{PSI-4+3}),
\beq\label{PSI-11}
\lim_{n\to\infty}\|w_n^2\| &=&\lim_{n\to\infty}\|s_n^2\|\\
&=&\lim_{n\to\infty}\|s_ne_ny_ns_n\|\\
&=&\lim_{n\to\infty}\|e_ny_n\|=0.
\eneq
\end{proof}

\section{Existence of automorphisms with the Rokhlin  property}

The following fact will be used without further notices.
\begin{lem}\label{afsmall}
Let $A$ be a unital simple \CA\, with $T(A)\not=\emptyset.$
Suppose that $\af\in \Aut(A).$ Then, for any positive element $a\in A,$
\beq\label{afsmall-1}
\inf\{\tau\circ \af(a): \tau\in T(A)\}=\inf\{\tau(a): \tau\in T(A)\}\tand\\
\sup\{\tau\circ \af(a): \tau\in T(A)\}=\sup\{\tau(a): \tau\in T(A)\}.
\eneq
In particular, if  $1>\eta>0$ and $e\in A$ is a projection such that $\tau(e)<\eta$
for all $\tau\in T(A),$ then
\beq\label{afsmall-2}
\tau(\af(e))<\eta\tforal \tau\in T(A).
\eneq

\end{lem}

\begin{proof}
This follows from the fact
that $\tau\circ \af$ is a a tracial state of $A$ for all $\tau\in T(A).$
\end{proof}

\begin{lem}\label{Text1}
Let $A$ be a unital separable amenable simple \CA\, with $TR(A)\le 1.$
Let $\af\in \Aut(A)$ and $\sigma\in \Aut({\cal Z})$ as defined in \ref{dJS}.
Then $\af\otimes \sigma$ has tracial Rokhlin property {\rm (see \cite{Ph1})} Moreover,
one has the following:
Let $k\ge 2$ be an integer.
Let $\ep>0$ and $1/k>\eta>0,$ let ${\cal F}\subset A\otimes {\cal Z}$ be a finite subset.
There are mutually orthogonal projections $p_0, p_1,...,p_{k-1}\in A\otimes {\cal Z}$ such that
\beq\label{Text1-1}
\|[p_j,\, x]\|<\ep\tforal x\in {\cal F},\,\,\,
(\af\otimes \sigma)^j(p_0)=p_j,\\
\tau(p_j)>1/k-\eta\tforal \tau\in T(A),\,\,\, j=0,1,2,...,k-1,\\
\tau(1-\sum_{j=0}^{k-1}(\af\otimes \sigma)^j(p_0))<\eta\tand \tau((\af\otimes \sigma)(p_{k-1}))>1/k-\eta
\tforal \tau\in T(A).
\eneq
\end{lem}

\begin{proof}
Note that, by Corollary 8.4 of \cite{LnlocAH}, $A\otimes {\cal Z}\cong A.$
We may prove the part of the statement in the  theorem after ``Moreover" for  ${\cal F}\subset A.$

Write $Z=\bigotimes_{n\in \N}{\cal Z}$ as an inductive limit
of $\bigotimes_{1\le j\le n}{\cal Z}.$
To simplify notation,   if necessary, without loss of generality, we may assume that
${\cal F}$ is in the unit ball of $A\bigotimes \bigotimes_{1\le j\le n}{\cal Z}.$
Put $\bt_0=\af\otimes \otimes_{1\le j\le n}\sigma_0.$  Note that
$\af\otimes \sigma(x)=\bt_0(x)$ for all $x\in {\cal F}.$
We will continue to use $\sigma$ for $\sigma|_{\bigotimes_{m\ge n+1}{\cal Z}}$
and identify $\bigotimes_{m\ge n+1}{\cal Z}$ with ${\cal Z}$ whenever it is convenient.

Let $1/4k>\eta>0.$  It follows from Proposition 4.4 of \cite{Sato2} that
there are mutually orthogonal elements $f_0,f_2,...,f_{k-1}\in {\cal Z}$
with $0\le f_j\le 1$ ($0\le j\le k-1$) such that
\beq\label{2ML-2}
\sigma(f_j)=f_{j+1},\,\,\,j=0,1,...,k-1\tand
\tau_Z(1-\sum_{j=0}^{k-1} f_j)<\eta/16,
\eneq
where $\tau_Z$ is the unique tracial state on ${\cal Z}.$

Put $f_k=\sigma(f_{k-1}).$  Note that $\sigma$ is asymptotically inner.
Therefore $\tau(f_k)=\tau(f_0)$ for all $\tau\in T(A).$

Fix $\ep_1>0$ with  $\ep_1<\min\{\eta, \ep\}.$
There is $N$  given by \ref{Extpro} for $f_k$ and
$\ep_1/64k$ (in stead of $\ep$).

Define
$$
{\cal F}_1=\{\bt_0^j(x): x\in {\cal F}\andeqn -k\le j\le k\}.
$$
Put $A_1=A\bigotimes \bigotimes_{m\ge n+1}{\cal Z}.$  Note that $A_1\cong A\otimes {\cal Z}\cong A.$
It follows from \ref{appdiv} that
there is a projection $P_1\in A_1,$  a finite dimensional
\SCA\, $B_1\cong M_{N}$ of $A_1$ with $1_{B_1}=P_1$ and a finite subset ${\cal F}_2$ in the unit ball of $A_1$
such that
\beq\label{2ML-3}
\tau(1-P_1)<\ep_1/64\tforal \tau\in T(A),\\\label{2ML-3+1}
\|[P_1,\,x]\|<\ep_1/64,\,\,\, P_1xP_1\in_{\eta_1/64} {\cal F}_2\tforal x\in {\cal F}_1\andeqn\\\label{2ML-3+2}
yb=by\tforal x\in {\cal F}_2\andeqn b\in B_1.
\eneq

Put $C_1=B_1\otimes {\cal Z}.$
By applying \ref{Extpro}, we obtain projections
$p_k\in \overline{(P_1\otimes f_k)C_1(P_1\otimes f_k)}$ such that
\beq\label{2ML-4}
(t\otimes \tau_Z)(p_{k})>d_{\tau_Z}(f_{k})-\ep_1/64k=d_{\tau_Z}(f_0)-\ep_1/64k
\eneq
for the tracial state $t$ on $B_1.$
Let $D_j$ be hereditary \SCA\, generated by
$1_A\otimes f_j.$

Put $\bt=\bt_0\otimes \sigma$ and put
$p_{k-j}=\bt^{-j}(p_k).$ Then $p_j\in D_j$ and $p_0, p_1,...,p_{k-1}$  are mutually orthogonal and
\beq\label{2ML-5}
\bt(p_j)=p_{j+1},\,\,\,j=0,1,...,k-1.
\eneq

For each $\tau\in T(A),$ define
$t(b)={1\over{\tau(P_1)}}\tau(b)$ for $b\in B_1.$ Then
$t$ is the unique tracial state on $B_1.$
We also estimate that
\beq\label{2ML-6}
({1\over{\tau(P_1)}})(\tau\otimes \tau_Z)(p_0)> (t\otimes\tau_Z)(p_0)>d_{\tau_Z}(f_0)-\eta/64k.
\eneq
It follows that
\beq\label{2ML-7}
(\tau\otimes \tau_Z)(p_0)>(1-\eta/64)(d_{\tau_Z}(f_0)-\eta/64k)
\eneq
for all $\tau\in T(A).$
Since $\bt$ is an automorphism,  for any $\tau\in T(A),$
$(\tau\otimes \tau_Z)\circ \bt$ is also tracial state in $A\otimes {\cal Z}.$ Hence
\beq\label{2ML-7+}
(\tau\otimes\tau_Z)(\bt^j(p_0))&>&(1-\eta/64)(d_{\tau_Z}(f_0)-\eta/64k)\\
&>& (1-\eta/64)(d_{\tau_Z}(f_j)-\eta/64k)
\eneq
for all $\tau\in T(A).$
Therefore
\beq\label{2ML-8}
(\tau\otimes \tau_Z)(\sum_{j=0}^{k-1}\bt^j(p_0))&>&
(1-\eta/64)(\sum_{j=0}^{k-1} \tau_Z(f_j)-\eta/64)\\
&>&(1-\eta/64)(1-\eta/16)\\
&>&1-\eta/4\tforal \tau\in T(A).
\eneq

Finally, we note that since $p_k\in C_1,$ for any
$y\in {\cal F}_2,$
\beq\label{2ML-9}
p_ky=yp_k\tforal y\in {\cal F}_2.
\eneq
Let $x\in {\cal F}.$  There is $y\in {\cal F}_2$ such that
\beq\label{2ML-10}
\|y-P_1\bt_0^j(x)P_1\|<\eta_1/64.
\eneq
It follows that, for $x\in {\cal F},$
\beq\label{2ML-11}
p_{k-j}x &=&\bt^{-j}(p_k)\bt^{-j}(P_1)\bt^{-j}(\bt^j(x))=\bt^{-j}(p_kP_1\bt^j(x))\\
&\approx_{\eta_1/64}& \bt^{-j}(p_kP_1\bt^j(x)P_1)\approx_{\eta_1/64} \af^{-j}(p_ky)\\
&=& \bt^{-j}(yp_k)\approx_{\eta_1/64} \bt^{-j}(P_1\af^j(x)P_1p_k)\\
&\approx_{\eta_1/64} &\bt^{-j}(\bt^j(x)p_k)=x\bt^{-j}(p_k)=xp_{k-j},
\eneq
$j=1,2,...,k.$

\end{proof}

The proof of the following is based on, again, an argument of Kishimoto
which was also further clarified in  Lemma 4.3 of \cite{Sato1}.
We state here for the case that $A$ is not assumed to have real rank zero.
We should notice that if $\af$ has the Rokhlin property then
$\bt\circ \af$ may not have the Rokhlin property in general even in the case that $\bt$ is an
approximately inner automorphism.

\begin{lem}\label{3ML}
Let $A$ be a unital separable amenable simple \CA\,
with $TR(A)\le 1,$ let
$\af\in \Aut(A)$ be an automorphism
and let $k\ge 1$ be an integer.
Suppose $\af$ has the following property:
there exists a central sequence of $\{p_n\}$ of projections
$\{p_n\}$ of $A$ such that
\beq\label{3ML-3}
\lim_{n\to\infty}\|\af^i(p_n)\af^j(p_n)\|=0\,\,\,{\rm if}\,\,\,i\not=j,\,\,\,0\le i,j\le k-1,\andeqn\\
\lim_{n\to\infty} \sup_{\tau\in T(A)}\tau(1-\sum_{j=0}^{k-1}\af^j(p_n))=0.
\eneq
Then,  for any $\ep>0$ and any finite subset ${\cal F}\subset A,$  there exist projections
$$e_1^{(0)},e_2^{(0)},...,e_{k-1}^{(0)},e_1^{(1)},e_2^{(1)},...,e_k^{(1)}\subset A$$
with $\sum_{j=1}^{k-1}e_j^{(0)}+\sum_{j=1}^k e_j^{(1)}=1_A,$ a
continuous path of unitaries $\{U(t): t\in [0,1]\}\subset A$  such that
\beq\label{3ML-1}
U(0)=1_A,\,\,\, \|[x,\, U(t)]\|<\ep\tforal x\in {\cal F}\andeqn
t\in [0,1],\\
\|{\rm Ad}\, U(1)\circ \af(e_j^{(i)})-e_{j+1}^{(i)}\|<\ep,\,\,\,
j=0,1,2...,k+i-2,\,\,\,i=0,1\\
\|[x,\,e_j^{(i)}]\|<\ep\tforal x\in {\cal F},\,\,\,j=0,1,...,k+i-2,\,\,\,i=0,1.
\eneq

\end{lem}

\begin{proof}

 Let $l^{\infty}(A)=\prod_n^{\infty}A$ be the $C^*$-algebra product of infinite copies of $A$ and $c_0(A)=\bigoplus_{n=1}^{\infty} A$ be the $C^*$-algebra
 direct sum of infinite copies of $A.$ Denote by
 $A_d=\{(a,a,...,a,...): a\in A\},$  a \SCA\, of $l^{\infty}(A).$
 Let $q(A)=l^{\infty}(A)/c_0(A)$ and let $\pi: l^{\infty}(A)\to c_0(A)$ be the quotient map.
 It is standard that there are, for each $n,$  mutually orthogonal projections
$\{p_{j,n}: 0\le j\le k-1\}$ such that
\beq\label{3ML-4-1}
\pi(\{p_{j,n}\})=\pi(\{\af^j(p_n)\}),\,\,\,j=0,1,2,...,k-1.
\eneq
Let
$q_n=1-\sum_{j=0}^{k-1}p_{j,n},$ $n=1,2,....$
 Then $\pi(\{p_{j,n}\})$ and $\pi(\{q_n\})$ are in
 $q(A)\cap \pi(A_d)'.$
 Note that, by \ref{afsmall},
 \beq\label{3ML-4}
\liminf_{n\to\infty} \inf_{\tau\in T(A)}\tau(p_{j,n})=\liminf_{n\to\infty}\inf_{\tau\in T(A)} \tau(p_n)\ge 1/k\andeqn
 \lim_{n\to\infty}\sup_{\tau\in T(A)}\tau(q_n)=0.
 \eneq
 It follows from \ref{PSI} that there exists a central sequence of partial isometries
 $\{w_n\}\in l^{\infty}(A)$ such that
 \beq\label{3ML-5}
 \lim_{n\to\infty}\|w_n^*w_n-q_n\|=0,\,\,\,
 \lim_{n\to\infty}\|\af(p_{k-1,n})w_n-w_n\|=0\andeqn \lim_{n\to\infty}\|w_n^2\|=0.
 \eneq
 Put $p_{k,n}=\af(p_{k-1,n}),$ $n=1,2,....$ Let $w=\pi(\{w_n\}).$ Then $w$ is a partial isometry in $q(A)$ such that
 \beq\label{3ML-6}
 w^*w=\pi(\{q_n\}),\,\,\, ww^*\le \pi(\{p_{k,n}\})\andeqn w^2=0.
 \eneq
 Let $V=w+w^*+1-w^*w-ww^*\in q(A)$
 and $u_n'=w_n+w_n^*+(1-q_n-w_nw_n^*),$ $n=1,2,....$
 Note that $\pi(\{u_n'\})=V.$ Then $V$ is a unitary in $q(A).$
 Moreover, $V\in q(A)\cap \pi(A_d)'.$
 Put $E=\pi(\{q_n\}+\{p_{0,n}\}).$ Then $E$ is a projection in $q(A)\cap \pi(A_d).$
 It is important to note, since $w^2=0,$ $ww^*\pi(\{q_n\})=\pi(\{q_n\})ww^*=0.$
 Since $p_{0,n}, p_{1,n},...,p_{k-1,n}$ mutually orthogonal and
 \beq\label{3ML-7}
 \pi(\{q_n\})+\sum_{j=0}^{k-1}\pi(\{p_{j,n}\})=1_{q(A)},
 \eneq
 \beq\label{3ML-8}
 ww^*\le \sum_{j=0}^{k-1}\pi(\{p_{j,n}\}).
 \eneq
 Since
 \beq\label{3ML-9}
 \pi(\{\af(p_{j,n}\})\pi(\{\af(p_{k-1,n}\})=0\tforal j=0,1,...,k-2,
 \eneq
 \beq\label{3ML-10}
 ww^*\pi(\{p_{j,n}\})=\pi(\{p_{j,n}\})ww^*=0\tforal j=1,2,...,k-1.
 \eneq
 Combining (\ref{3ML-8}) and (\ref{3ML-9}),
 we have
 \beq\label{3ML-11}
 ww^*\le p_{0,n}.
 \eneq
 It follows that
 $V(1-E)=1-E.$  By (\ref{3ML-6}), we also have
 \beq\label{3ML-18}
\pi(\{q_n\})\le w^*\pi(\af(\{p_{k-1,n}\}))w.
 \eneq
 Since $V$ is selfadjoint, $sp(V)\subset \{-1, 1\}.$ It follows that there is a projection $F\in q(A)$ such that $V=F-(1-F)=2F-1.$  It follows that
 $F\in q(A)\cap \pi(A_d)'.$ Therefore there exists
 a central sequence of projections $\{f_n\}$ in $A$ such
 that $\pi(\{f_n\})=F.$ For each $n,$ define
 $U_n(t)=f_n+e^{\pi i t}(1-f_n)$ for $t\in [0,1].$ Then
 $\{U_n(t): t\in [0,1]\}$ is a continuous path of unitaries in
 $A$ such that $U_n(0)=1_A$ and $U_n(1)=f_n-(1-f_n).$
 Moreover $\pi(\{U_n(1)\})=V.$
 Define a \morp\, $\Lambda_n:A\to A$ by
 $\Lambda_n(a)=\af^{-1} (u_n'a(u_n')^*)$ for all $a\in A.$ (Note that
 $u_n'$ is not exactly a unitary.)
Note that $\pi\circ \{\Lambda_n\}$ is an automorphism on $\pi(A_d).$
 Define ${\bar e}_j^{(1)}=\pi(\Lambda_n^{k-j}(q_n)\}),$ $j=0,1,...,k.$ Note that ${\bar e}_j^{(1)}\in q(A)$ is a projection and ${\bar e}_k^{(1)}=\pi(\{q_n\}).$
  We compute that
 $$
 {\bar e}_{k-1}^{(1)}=\pi(\{\af^{-1}(u_n'q_n(u_n')^*)\})\le\pi(\{\af^{-1}\circ \af(p_{k-1,n})\})=\pi(\{p_{k-1,n}\}).
 $$
 Since $V(1-E)=1_E,$ we conclude that
 \beq\label{3ML-19}
 {\bar e}_j^{(1)}\le \pi(\{\af^{j-k})(p_{k-1,n})\})=\pi(\{p_{j,n}\}),
 \eneq
 $j=2,3,...k-1.$
 In particular, $\{ {\bar e}_j^{(1)}: 1\le j\le k\}$ is a set of mutually
 orthogonal projections. Therefore, we obtain, for each $j,$  a sequence of
 projections $e_{j,n}^{(1)}\in A$ such that $e_{j,n}^{(1)}\le p_{j,n}$
 for all $n,$ $\pi(\{e_{j,n}^{(1)}\})={\bar e}_j^{(1)},$ and $j=0,1,2,...,k-1$ and $e_{k,n}^{(0)}=q_n,$
 $n=1,2,....$ In particular, $\{e_{j,n}^{(1)}: 0\le j\le k\}$ is a set of mutually orthogonal projections for each $n.$
 Moreover,
 \beq\label{3ML-20}
\pi(\{{\rm Ad}\, u_n'\circ \af(e_{j,n}^{(1)})\})=\pi(\{e_{j+1,n}^{(1)}\}),\,\,\,
j=0,1,...,k-1.
 \eneq
 Define
 $$
 e_{j,n}^{(0)}=p_{j,n}-e_{j,n}^{(1)},\,\,\, j=0,1,...,k-1.
 $$
By (\ref{3ML-20}) and (\ref{3ML-4-1}),
\beq\label{3ML-11+}
{\rm Ad}\,V(\pi(\{\af(e_{j,n}^{(0)})\})=\pi(\{e_{j+1}^{(0)}\}),\,\,\,
j=0,1,2,...,k-2.
\eneq
Note that
\beq\label{3ML-12}
\sum_{j=1}^{k-1}e_{j,n}^{(0)}+\sum_{j=1}^k e_{j,n}^{(1)}=1.
\eneq

Now, fix an $\ep>0$ and a finite subset ${\cal F},$ the lemma follows
by taking $e_{j,n}^{(0)},$ $ e_{j,n}^{(1)}$  and
$U_n(t)$ for sufficiently large $n.$

\end{proof}

\begin{NN}\label{Dsigma}
{\rm Let $B_0=\C,$ $B_k=\bigotimes_{j=1}^{k} {\cal Z},$ the tensor product of $k$ copies of ${\cal Z},$
$k=1,2,....$
Denote  $C_k=\bigotimes_{j\ge k+1} {\cal Z},$  $D_1=\C,$  $D_n=\bigotimes_{j=1}^{n-1} B_j,$ $n=2, 3,....$

In what follows, we write ${\cal Z}=\bigotimes_{n\in \N} {\cal Z}=D_k\otimes C_{k(k-1)/2}.$
 Define
$\sigma^c_k=\sigma|_{C_k},$ $k=2,3,...$
}
\end{NN}

\begin{thm}\label{Mext}
Let $A$ be a unital separable amenable simple \CA\,  with $TR(A)\le 1$ and let $\af\in \Aut(A).$ Then
there exists an automorphism ${\tilde \af}\in \Aut(A)$ which has the Rokhlin
property such that
$\af$ and ${\tilde \af}$ are strongly asymptotically unitarily equivalent.
\end{thm}

\begin{proof}
Since $A\cong A\otimes {\cal Z}$ and ${\cal Z}\otimes {\cal Z}$ is asymptotically unitarily equivalent to ${\cal Z},$ $\af$ and $\af\otimes \sigma$ are strongly asymptotically unitarily equivalent.


For each $n$ and $k\le n,$  let $\{{\cal F}_{n,k}\}$ be an increasing sequence of
finite subsets of the unit ball of
$$A\bigotimes D_n\bigotimes  B_k\bigotimes 1_{C_{(n/2)(n-1)+k}}.$$
We may assume, without loss of generality,  that ${\cal F}_{n,k+1}\supset {\cal F}_{m,k}$ for all $k$ and $n\ge m,$  and
$\cup_{k=1}^{\infty}{\cal F}_{k,k}$ is dense in the unit ball of $A\otimes {{\cal Z}}.$

Put $\af_{1,1}=\af\otimes \sigma$ and $\af_{1,1}'=\af_{1,1}|_{A\otimes B_1}.$
Suppose that we have constructed two sets of projections
$\{e_{j,n,k}^{(0)}\},$ $j=0,1,...,k-1,$ and $\{e_{j,n,k}^{(1)}\},$ $j=0,1,2,...,k,$ in $A\bigotimes D_n\bigotimes B_k\bigotimes  1_{C_{(n/2)(n-1)+k}},$
a  continuous path of unitaries $\{w_{n,k}(t): t\in [0,1]\}$ with
$w_{n,k}(0)=1$  in $A\bigotimes D_n\bigotimes B_k\bigotimes  1_{C_{(n/2)(n-1)+k}},$ an
automorphism $\af_{n,k}'$ on
$A\bigotimes D_n\bigotimes B_k\bigotimes  1_{C_{(n/2)(n-1)+k}}$  and a
finite subset\\
 ${\cal G}_{n,k}\subset A\bigotimes D_n\bigotimes B_k\bigotimes 1_{C_{(n/2)(n-1)+k}}$ such that
\beq\label{Mext-2}
\hspace{-1in}&&\hspace{-0.4in}\sum_{j=1}^{k-1}e_{j,n,k}^{(0)}+\sum_{j=1}^k e_{j,n,k}^{(1)}=1;\\\label{Mext-2+1}
&&\hspace{-0.6in}\|[x,\,e_{j,n,k}^{(i)}]\|<1/(n+k)\tforal x\in {\cal G}_{n,k-1},\,\,\,
j=0,2,...,k-i+1\andeqn i=0,1,\\\label{Mext-2+2}
&&\hspace{-0.6in}\|{\rm Ad}\, w_{n,k}(1)\circ \af_{n,k-1}'(e_{j,n,k}^{(i)})-e_{j+1,n,k}^{(i)}\|<1/(n+k),\,\,\,j=0,1,...,
k-2+i,\,\,\, i=0,1,\\\label{Mext-2+3}
&&\hspace{-0.6in}\|[w_{n,k}(t),\, x]\|<1/2^{n+k},\,\,\, x\in {\cal G}_{n,k-1}\andeqn t\in [0,1],\\\label{Mext-2+4}
&&\hspace{-0.6in}{\cal G}_{n,k}\supset \cup_{j=1}^k{\cal F}_{n,j}\cup\af_{n,k}({\cal G}_{n,k-1})\cup
\{\af_{n,k}(e_{j,n,k}^{(i)}): i,j\},
\eneq
where $\af_{n,k}'= ({\rm Ad}\, w_{n,k}\circ \af_{n,k-1}')\otimes \sigma_0$
(defined on $A\bigotimes D_n\bigotimes B_k\bigotimes  1_{C_{(n/2)(n-1)+k+1}}$),
where we identify $B_k$ with $B_{k-1}\bigotimes {\cal Z},$
$\af_{n,n}'=\af_{n+1, 0}',$
$k=1,2,...,n_1\le n,$ $n=1,2,...,m.$  Note also, we will write $A\otimes D_{n+1}\otimes C_{(n/2)(n+1)}$ instead of $A\otimes D_n\otimes B_n\otimes C_{(n/2)(n+1)}$ and
${\cal G}_{n,n}={\cal G}_{n+1,0},$ $n=1,2,....$
It is important to note that
$\af_{m,n_1}'$ satisfies the assumption of \ref{3ML} for $k=n_1,$ by applying \ref{3ML} on $\af_{m,n_1}',$ we obtain (\ref{Mext-2}),
(\ref{Mext-2+1}),(\ref{Mext-2+2}), (\ref{Mext-2+3}) and (\ref{Mext-2+4})
for $k=n_1+1.$  By induction, we obtain (\ref{Mext-2}),
(\ref{Mext-2+1}),(\ref{Mext-2+2}), (\ref{Mext-2+3}) and (\ref{Mext-2+4})
for all $k=n_1+1,$ if $n_1<n.$ If $k=n_1=n,$
we also set $\af_{n+1,0}'=\af_{n,n}'$ and, as above, obtain $\{w_{n+1,1}(t): t\in [0,1]\}$  with $w_{n+1,1}(0)=1$ and
$\af_{n+1,1}'={\rm Ad}\, w_{n+1,1}(1)\circ \af_{n+1,0}'\otimes \sigma_0.$

Define $\af_{n,k}=\af_{n,k}'\otimes \sigma|_{C_{(n/2)(n-1)+k}}.$

Consider the sequence  of automorphisms
$\{\beta_m\}$ as follows: $\af_{1,1},$
$\af_{2,1},$ $\af_{2,2},$ ...,
$\af_{n-1, n-1}=\af_{n,0},$ $\af_{n,1},$ $\af_{n,2},...,$
$\af_{n,n}=\af_{n+1,0},$ $\af_{n+1,1},....$

We will show that $\beta_m(x)$ is Cauchy for all $x\in A\otimes {\cal Z}.$
In fact, if $x\in {\cal G}_{n,k-1}$ and $m\ge n>k,$ by (\ref{Mext-2+3}),
\beq\label{Mext-3}
&&\|\af_{m,k}(x)-\af_{m,k+1}(x)\|<1/2^{m+k}\andeqn {\rm if}\,\,\, m=n=k,\\\label{Mext-3+1}
&&\|\af_{n+1,0}(x)-\af_{n+1,1}(x)\|<1/2^{n+1}.
\eneq
It follows that, with $m\ge k', n$ and $n\ge k,$
\beq\label{Mext-3+2}
\|\af_{m,k'}(x)-\af_{n,k}(x)\|<{1\over{2^{n-1}}}\le {1\over{2^{{(n+k)/2-1}}}}\tforal x\in {\cal G}_{n,k-1}.
\eneq
Since $\cup_{n=1}^{\infty}\cup_{k=1}^{n} {\cal G}_{n,k}$ is dense in the unit ball of $A,$
we conclude that $\{\bt_m(x)\}$ is Cauchy for all $x\in A.$

Define ${\tilde \af}(x)=\lim_{m\to\infty}\bt_m(x).$ It is an endomorphism. Similar to
the above estimates as in (\ref{Mext-3}) and (\ref{Mext-3+1}), if $x\in {\cal G}_{n,k}$ and $m\ge n>k,$ then
\beq\label{Mext-4}
\|\af_{m,k}^{-1}(x)-\af_{m,k}^{-1}\circ {\rm Ad}w_{m,k+1}^*(1)(x)\|<1/2^{m+k}\andeqn {\rm if}\,\,\,n=k,\\
\|\af_{n+1,0}^{-1}(x)-\af_{n+1,0}^{-1}\circ {\rm Ad}\, w_{n+1,1}^*(1)(x)\|<1/2^{n+1}.
\eneq
Note that $\af_{n,k+1}^{-1}=\af_{n,k}^{-1}\circ {\rm Ad}\, w_{n,k+1}^*$ (if $n>k$) and
$\af_{n+1,1}^{-1}=\af_{n+1,0}^{-1}\circ {\rm Ad}\, w_{n+1,1}^*.$
It follows that, if $m\ge n, k$ and $n\ge k',$
\beq\label{Mext-5}
\|\af_{m,k}^{-1}(x)-\af_{n,k'}^{-1}(x)\|<2^{n-1}.
\eneq
Therefore $\{\bt_m^{-1}(x)\}$ is Cauchy for all $x\in A.$
We obtain an endomorphism ${\tilde \af}^{-1}$ such that
${\tilde \af}^{-1}(x)=\lim_{m\to\infty} \bt_m
^{-1}(x)$ for all
$x\in \cup_{k=1}^{\infty}{\cal G}_{k,k}.$ It follows that
${\tilde \af}^{-1}\circ {\tilde \af}={\rm id}_A={\tilde \af}\circ {\tilde \af}^{-1}.$ Thus ${\tilde \af}$ is an automorphism.

Fix $k\ge 1,$ $\ep>0$ and a finite subset ${\cal F},$ by (\ref{Mext-2+1}),
there exists an integer $n_1\ge 1$ such that
\beq\label{Mext-6}
\|[e_{j,n,k}^{(i)},\, x]\|<\ep\,\,\, j=0,1,...,k+i-1,\,\,\, i=0,1,
\eneq
for all $n\ge n_1.$
We may also assume that $1/n_1<\ep/4$ and $n_1\ge k.$
 Then, if $n\ge n_1,$
\beq\label{Mext-7}
e_{j+1,n,k}^{(i)}\approx_{1/(n+k)} {\rm Ad}\, w_{n, k}(1)\circ \af_{n,k-1}(e_{j,n,k}^{(i)})=\af_{n,k}(e_{j,n,k}^{(i)})\\
\approx_{2^{-(n+k+1)}}  {\rm Ad}\, w_{n,k+1}(1)\circ \af_{n,k}(e_{j,n,k}^{(i)})=\af_{n,k+1}(e_{j,n,k}^{(i)}),
\eneq
$j=0,1,...,k+i-2,\,\,\, i=0,1.$
If $n=k+l,$
\beq\label{Mext-8}
&&e_{j+1,n,k}^{(i)}
\approx_{1/(n+k)+\sum_{m=n+k+1}^{n+k+l}2^{-m}}\af_{n,k+l}(e_{j,n,k}^{(i)})\\
&&\approx_{1/2^{n+1}} {\rm Ad}\, w_{n+1,1}(1) \circ\af_{n+1,0}(e_{j,n,k}^{(i)})\\
&=&\af_{n+1,1}(e_{j,n,k}^{(i)}),
\eneq
or,
\beq\label{Mext-8+0}
&&e_{j+1,n,k}^{(i)}\approx_{1/(n+k)+2^{-(n+k)}+2^{-(n+k)/2\,-1}}\af_{n+1,1}(e_{j,n,k}^{(i)}),
\eneq
$j=0,1,...,k+i-2,\,\,\,i=0,1.$
It follows that (using also (\ref{Mext-3+2})), if $m\ge 2(n+1),$
\beq\label{Mext-8+}
\|e_{j+1,n,k}^{(i)}-\bt_m(e_{j,n,k}^{(i)})\|<1/(n+k)+1/2^{(n+k)/2} +1/2^{(n+k)/2-1}
\eneq
for $j=0,1,...,k+i-2,$ $i=0,1.$
Therefore
\beq\label{Mext-9}
\|{\tilde \af}(e_{j,n,k}^{(i)})-e_{j+1,n,k}^{(i)}\|<4/(n+k).
\eneq
Thus ${\tilde \af}$ has the Rokhlin property.
To show ${\tilde \af}$ is asymptotically unitarily equivalent to $\af\otimes \sigma,$ define $w(t)$ as follows:
\beq\label{Mext-10}
w(t)=w_{n,k}(nt-n(n-1)-(k-1))w_{n,k-1}(1)\cdots w_{1,1}\\
{\rm for}\,\,\, n-1+{k-1
\over{n}}\le t<n-1+{k\over{n}},
\eneq
where we identify $w_{n+1,0}(t)=w_{n,n}(t)$ for $t\in [0,1].$
Note that $\{w(t): t\in [0,\infty)\}$ is a continuous path of unitaries in $A\otimes {\cal Z}.$ It follows from (\ref{Mext-2+3}) that
\beq\label{Mext-11}
{\tilde \af}(x)=\lim_{t\to\infty}{\rm Ad}\, w(t)\circ (\af\otimes  \sigma)(x)\tforal x\in A\otimes {\cal Z}.
\eneq
Since $w_{n,k}(0)=1,$ we conclude that $w(0)\in U_0(A).$ Therefore
${\tilde \af}$ and $\af\otimes \sigma$ are strongly asymptotically
unitarily equivalent.

\end{proof}

\section{Classification of automorphisms
with Rokhlin property}

\begin{df}\label{DKK}
Let $A$ be a unital separable amenable \CA\, which satisfies the UCT.
At the present, we also assume that $A$ has stable rank one and simple.
Let $KK_e^{-1}(A,A)^{++}$ be the subset of those elements $\kappa$
in $KK(A,A)$ which is strictly positive, i.e.,
$\kappa(K_0(A)_+\setminus \{0\})\subset K_0(A)_+\setminus \{0\},$  preserving the identity,
$\kappa([1_A])=[1_A],$ invertible, i.e., there is $\kappa^{-1}$ such that
$\kappa\times \kappa^{-1}=\kappa^{-1}\times \kappa=[{\rm id}_A].$
Let $\gamma: T(A)\to T(A)$ be an affine homeomorphism and
$\lambda: U(A)/CU(A)\to U(A)/CU(A)$ be a continuous isomorphism.
Let $\rho_A: K_0(A)\to \Aff T(A)$ be the order preserving \hm\, defined by
$\rho_A([p])(\tau)=\tau(p)$ for all projections $p\in M_{\infty}(A).$ We say $\kappa$ and
$\gamma$ are compatible if
$\gamma(\tau)(p)=\tau(\rho_A(\kappa([p])))$ for all projections $p\in M_{\infty}(A)$ and
$\tau\in T(A).$  Let  $\gamma^*: \Aff T(A)\to \Aff T(A)$ be the continuous affine isomorphism
induced by $\gamma,$ i.e.,
$$
\gamma^*(f)(\tau)=f(\gamma(\tau))\tforal f\in \Aff T(A)
$$ and for all $\tau\in T(A).$ If $\gamma$ is compatible with $\kappa,$ denote by
$$
\overline{\gamma^*}:\Aff T(A)/\overline{\rho_A(K_0(A))}\to \Aff T(A)/\overline{\rho_A(K_0(A))}
$$
the isomorphism induced by $\gamma^*.$
We say $\lambda$ is compatible with $\kappa,$ if $\lambda$ maps
$U(A)/CU(A)$ into $U(A)/CU(A)$ and
$\pi_1\circ \lambda({\bar u})=\kappa([u])$ for all unitaries $u\in U(A),$ where $\pi_1: U(A)/CU(A)\to K_1(A)$ is the quotient map.
We say the triple
$(\kappa, \gamma,\lambda)$ are compatible, if $\gamma$ and $\lambda$ are compatible with $\kappa$ and
$$
{\bar \Delta}_A\circ \lambda\circ {\bar \Delta}_A^{-1}=\overline{\gamma^*},
$$
where ${\bar \Delta}_A: U_0(A)/CU(A)\to \Aff T(A)/\overline{\rho_A(K_0(A))}$ is the de la Harp and Skandalis
determinant.

Denote by $KKUT_e^{-1}(A,A)^{++}$ the set of all
compatible triples $(\kappa, \gamma, \lambda).$

For each $\phi\in \Aut(A),$ there is a triple
$([\phi], \phi_T, \phi^{\ddag}),$ where $[\phi]\in KK_e^{-1}(A,A)^{++}$ is induced $KK$-element,
$\phi_T: T(A)\to T(A)$ is the affine isomorphism induced by $\phi$ and defined by
$\phi_T(\tau)(a)=\tau(\phi(a))$ for all selfadjoint elements $a\in A$ and for all $\tau\in T(A),$ and where
$\phi^{\ddag}: U(A)/CU(A)\to U(A)/CU(A)$ is the isomorphism
induced by $\phi$ and defined by $\phi^{\ddag}({\bar u})=\overline{\phi(u)}$ for all unitraies $u\in U(A).$
Therefore there is a map  $\mathfrak{K}: \Aut(A)\to KKUT_e^{-1}(A,A)^{++}$ defined by
$\mathfrak{K}(\phi)=([\phi], \phi_T,\phi^{\ddag}).$

\end{df}

\begin{df}\label{Drotation}
Let $\phi, \psi: A\to A$ be two unital \hm s. Let
$$
M_{\phi, \psi}=\{f\in C([0,1], A): f(0)=\phi(a),\,\,\, f(1)=\psi(a)\,\,\,{\rm for\,\,\, some}\,\,\, a\in A\}.
$$
There is a \hm\,
$$
R_{\phi, \psi}: K_1(M_{\phi, \psi})\to \Aff T(A)
$$
defined by
$$
R_{\phi, \psi}([u])(\tau)={1\over{2\pi i}}\int_0^1 \tau({d u(t)\over{dt}}u(t)^* )dt.
$$
When $[\phi]=[\psi]$ in $KK(A,A),$ there is a splitting short exact
sequence
$$
0\to \underline{K}(SA)\to^{[\iota]} \underline{K}(M_{\phi,\psi})\rightleftharpoons_{\Theta} \underline{K}(A)\to 0,
$$
where $\iota: SA=C_0((0,1), A)\to M_{\phi, \psi}$ is the embedding. In particular,
$$
0\to K_0(A)\to^{\iota_{*1}} \underline{K}(M_{\phi,\psi})\rightleftharpoons_{\theta} K_1(A)\to 0,
$$
where $\theta=\Theta|_{K_1(A)}.$ 
We will also use  $R_{\phi, \psi}$ for the induced map $R_{\phi, \psi}: K_1(A)\to
\Aff (T(A)).$
See, for example, Definition 3.2 of \cite{Lninv} for details.

Let ${\cal R}_0\subset Hom(K_1(A), \Aff T(A))$  be the subset of those $\eta$ for which there
is a \hm\, $h: K_1(A)\to K_0(A)$ such that $\eta=\rho_A\circ h.$

We now assume that $\mathfrak{K}(\phi)=\mathfrak{K}(\psi).$  Then, by Lemma 9.2 of \cite{Lninv},
$$
R_{\phi, \psi}\in Hom(K_1(A), \overline{\rho_A(K_0(A))}).
$$
Denote by  ${\overline{R_{\phi, \psi}}}$ the element in $Hom(K_1(A), \Aff T(A))/{\cal R}_0.$
If $\overline{R_{\phi, \psi}}=0,$ then there is exists $\Theta\in Hom_{\Lambda}(\underline{K}(A), \underline{K}(M_{\phi, \psi}))$ such that $[\pi_0]\circ \Theta=[{\rm id}_A]$ and
$$
R_{\phi, \psi}\circ \Theta|_{K_1(A)}=0,
$$
where $\pi_0: M_{\phi, \psi}\to A$ is the point-evaluation at $t=0$ (see 3.2 of \cite{Lninv}). One also has
\beq\label{iota}
(\pi_0)_{*1}\circ \iota=0.
\eneq
\end{df}

\begin{thm}\label{CM1}
Let $A$ be a unital separable simple amenable \CA\, with $TR(A)\le 1$ which satisfies the UCT.
Let $\af$ and $\bt$ be two automorphisms with the Rokhlin property.
Then the following are equivalent:

{\rm (1)} $\af$ and $\bt$ are strongly outer conjugate;

{\rm (2)} $\af$ and $\bt$ are asymptotically unitarily equivalent and

{\rm (3)} $\mathfrak{K}(\af)=\mathfrak{K}(\bt)$ and
      $\overline{R_{\af, \bt}}=0.$
\end{thm}

\begin{proof}
We have proved ``(2) $\Rightarrow$ (1)" in \ref{MT1}.  That of `` (2) $\Leftrightarrow$ (3)"  follows from
Theorem 7.2 of \cite{Lninv}.

Now assume that $\af$ and $\bt$ are strongly outer conjugate. Then there is a unitary $u\in U(A)$ and
a strongly  asymptotically inner automorphism $\sigma$ such that
$\af={\rm Ad}\, u\circ \sigma^{-1}\circ \bt\circ \sigma.$  It follows that
$\sigma\circ {\rm Ad}\, u^*\circ \af=\bt \circ \sigma.$
 Since $\sigma$ is (strongly) asymptotically inner, there is a continuous path of unitaries
$\{u(t): t\in [1, \infty)\}\subset A$ such that
$$
\sigma(x)=\lim_{t\to\infty} u^*(t)xu(t)\tforal x\in A.
$$
It follows that
$$
\bt\circ \sigma(x)=\lim_{t\to\infty}\bt(u(t))^*\bt(x)\bt(u(t))\tforal x\in A.
$$
It follows that $\bt\circ \sigma$ is (strongly) asymptotically unitarily equivalent to $\bt.$ Similarly
$\sigma\circ {\rm Ad}\, u^*\circ \af$ is asymptotically unitarily equivalent to $\af$ (note
that $u^*$ may not be in $U_0(A)$). Therefore
$\af$ and $\bt$ are asymptotically unitarily equivalent. This proves that `` (1) $\Rightarrow$ (2)".

\end{proof}

The following also follows from \ref{MT1}.

\begin{thm}\label{CM2}
Let $A$ be a unital separable simple amenable \CA\, with $TR(A)\le 1$ which
satisfies the UCT and let $\af$ and $\bt$ be two automorphisms on $A$ with the Rokhlin property. Then
$\af$ and $\bt$ are strongly outer conjugate and uniformly approxiamtely conjugate if and only if
$\af$ and $\bt$ are strongly asymptotically unitarily equivalent.
\end{thm}

\begin{NN}\label{mult}
{\rm
Let $\langle \af\rangle , \langle \bt\rangle \in \Aut_R(A)/\sim_{scc}$  be two elements represented by
automorphisms $\af$ and $\bt$ with the Rokhlin property. One can define a multiplication
by $\langle \af\rangle\circ \langle \bt \rangle =\langle {\widetilde{\af\circ \bt}} \rangle,$ where
${\widetilde{ \af\circ \bt}}$ is an automorphism with the Rokhlin property which is strongly asymptotically
unitarily equivalent to $\af\circ \bt$ given by \ref{Mext}.
This makes $\Aut_R(A)/\sim_{scc}$ a group with the identity
represented by an asymptotically inner automorphism which has the Rokhlin property.

Similarly, $\Aut_R(A)/\sim_{saucc}$ is also a group with the identity
represented by a strongly asymptotically inner automorphism which has the Roklin property.

If $(\kappa_1, \gamma_1, \lambda_1),\, (\kappa_2,\gamma_2, \lambda_2)\in KKUT_e^{-1}(A,A)^{++},$ define
$$
(\kappa_1, \gamma_1, \lambda_1)\times  (\kappa_2,\gamma_2, \lambda_2)
=(\kappa_2\times \kappa_1, \gamma_2\circ \gamma_1, \lambda_1\circ \lambda_2).
$$
This makes  $KKUT_e^{-1}(A,A)^{++}$  a group. Let $\af,\, \bt\in \Aut(A).$
Then
$$
\mathfrak{K}(\af\circ \bt)=([\af\circ \bt], (\af\circ \bt)_T, (\af\circ \bt)^{\ddag})=([\bt]\times [\af], \bt_T\circ \af_T, \af^{\ddag}\circ \bt^{\ddag})=\mathfrak{K}(\af)\times \mathfrak{K}(\bt).
$$
Thus, by \ref{CM1},  $\mathfrak{K}$ gives a group \hm\, from
$\Aut_R(A)/\hspace{-0.05in}\sim_{scc}$ into $KKUT_e^{-1}(A,A)^{++}.$

}
\end{NN}

\begin{thm}\label{CMT2}
Let $A$ be a unital separable simple amenable \CA\, with $TR(A)\le 1$ which satisfies the UCT.
Then one has the following short exact sequence of groups :
\beq\label{CMT2-1}
1\to {\rm Hom}(K_1(A), \overline{\rho_A(K_0(A))})/{\cal R}_0\to \Aut_R(A)/\sim_{scc}
 {\stackrel{\mathfrak{K}}{\rightarrow}} KKUT_e^{-1}(A,A)^{++}\to 1.
\eneq

\end{thm}

\begin{proof}
Theorem \ref{Mext} shows that every automorphism $\af\in \Aut(A)$ is strongly asymptotically unitarily equivalent
to an automorphism in $\Aut_R(A).$ Thus the theorem follows from this, Theorem \ref{CM1} and
Corollary 9.10 of \cite{Lninv}.
\end{proof}

\begin{df}\label{DG0}
Let $G_0$ be the subset
of all those automorphisms which are                                                outer conjugate to some  asymptotically inner automorphisms which have the Rokhlin property.
Let ${\bar G}_0$ be the image of $G_0$ in $\Aut_R(A)/\sim_{scc}.$
Suppose that $\af,\bt$ are two automorphisms with the Rokhlin property whose image in
$\Aut_R(A)/\sim_{scc}$ are in ${\bar G_0}.$
Let ${\rm Ad}\, u\circ \sigma_1^{-1}\circ \af\circ \sigma_1=\dt_1$  and ${\rm Ad}\, v\circ \sigma_2^{-1}\circ \af\circ \sigma_2=\dt_2,$
 where $\dt_1, \dt_2$ are  asymptotically inner
automorphisms with the Rokhlin property, $u, v\in U(A),$ $\sigma_1$ and $\sigma_2$ are  asymptotically inner automorphisms.
Consider $\af\circ \bt.$
Then
\beq\label{mult-1}
&&{\rm Ad}\, v\circ \sigma_2^{-1}\circ (\af\circ \bt)\circ \sigma_2\\
&=&{\rm Ad}\,  v\circ \sigma_2^{-1}\circ \af\circ \sigma_2\circ {\rm Ad}\, v^*\circ {\rm Ad}\, v\circ
\sigma_2^{-1}\circ \bt\circ \sigma_2\\
&=&{\rm Ad}\,  v\circ \sigma_2^{-1}\circ \af\circ \sigma_2\circ {\rm Ad}\, v^*\circ \dt_2
\eneq
which is strongly asymptotically unitarily equivalent to ${\rm Ad}\,  v\circ \sigma_2^{-1}\circ \af\circ \sigma_2\circ {\rm Ad}\, w$ for some $w\in U(A).$ Therefore
\beq\label{mult-2}
\langle \af\circ \bt\rangle=\langle {\rm Ad}\,  v\circ \sigma_2^{-1}\circ \af\circ \sigma_2\circ {\rm Ad}\, w\rangle.
\eneq
However,
\beq\label{mult-2+}
{\rm Ad}\,  v\circ \sigma_2^{-1}\circ \af\circ \sigma_2\circ {\rm Ad}\, w={\rm Ad}\,( vw)\circ {\rm Ad}\, w^*\circ \sigma_2^{-1}\circ \af\circ
\sigma_2\circ {\rm Ad}\, w
\eneq
 is outer
conjugate to $\dt_1.$ It follows that
$\langle \af\circ \bt \rangle\in \overline{ G_0}.$  This shows that $\overline{G_0}$ is a subgroup.
Clearly it is  a normal subgroup. It is also clear that
$(\Aut_R(A)/\sim_{scc})/\overline{G_0}=\Aut_R(A)/\sim_{cc}.$ This also implies that $\Aut_R(A)/\sim_{cc}$ is a group
where the identity is the class of asymptotically inner automorphisms with the Rokhlin property.

\end{df}

\begin{thm}\label{CMT3}
Let $A$ be a unital separable simple amenable \CA\, with $TR(A)\le 1$ which satisfies the UCT.
Then there is a short exact sequence of groups:
\beq\label{CMT3-1}
1\to K_1(A)/H_1(K_0(A), K_1(A))\to \Aut_R(A)/\hspace{-0.05in}\sim_{saucc}\,\,\to \Aut_R(A)/\hspace{-0.05in}\sim_{scc}\,\to 1.
\eneq

\end{thm}

\begin{proof}
Denote $\pi: \Aut_R(A)/\hspace{-0.05in}\sim_{saucc}\to \Aut_R(A)/\hspace{-0.05in}\sim_{scc}$
the quotient map.
Let $\af$ be a strongly asymptotically inner automorphism on $A$ with the Rokhlin property.
There exists a continuous path of unitaries $\{u(t): t\in [1, \infty)\}\subset A$ with
$u(1)=1_A$ such that
\beq\label{CMT3-1+}
\af(a)=\lim_{t\to\infty}{\rm Ad}\, u(t)(a)=\lim_{t\to\infty}u(t)^*au(t)\tforal a\in A.
\eneq
Let $\bt$ be any asymptotically inner automorphism with the Rokhlin property.
There exists a continuous path of unitaries $\{v(t): t\in [1, \infty)\}\subset A$ such that
\beq\label{CMT3-2}
\bt(a)=\lim_{t\to\infty}{\rm Ad}\, v(t)(a)\tforal a\in A.
\eneq
Let $Q: K_1(A)\to K_1(A)/H_1(K_0(A), K_1(A))$ be the quotient map.
Define ${\mathfrak{ U}}(\bt)= Q([v(1)^*])$ in $K_1(A)/H_1(K_0(A), K_1(A)).$
Define $w(t)=v(t)^*u(t)\tforal t\in [1, \infty).$  Note that $w(1)=v(1)^*.$ Then
\beq\label{CMT3-3}
\af\circ \bt^{-1}(a)=\lim_{t\to\infty} w(t)^*aw(t)\tforal a\in A.
\eneq
By the proof of \ref{MT1}, there exists $v\in U(A)$ with
$[v]=[v(1)^*]$ in $U(A)/U_0(A)$  and $\sigma\in \Aut(A)$  which
is  strongly asymptotically inner such that
$\af={\rm Ad}\, v\circ \sigma^{-1}\circ \bt \circ \sigma.$
Suppose that there exists another continuous path of unitaries $\{z(t): t\in [1,\infty)\}\subset A$ such that
\beq\label{CMT3-4}
\bt(a)=\lim_{t\to\infty}{\rm Ad}\, z(t)(a)\tforal a\in A.
\eneq
 Then ${\rm Ad}\, z(1)^*\circ \bt$ is strongly asymptotically unitarily equivalent
to ${\rm id}_A.$ By (\ref{CMT3-2}), ${\rm Ad}\, v(1)^* \circ \bt$ is also strongly asymptotically
unitarily equivalent to ${\rm id}_A.$ It follows that ${\rm Ad}\, v(1)z(1)^*\circ \bt$ is strongly asymptotically
unitarily equivalent to $\bt.$ It follows from Proposition 12.3 of \cite{Lnamj} that
$[v(1)z(1)^*]\in H_1(K_0(A), K_1(A)).$ In other words, $Q([v(1)^*])=Q([z(1)^*]).$
Thus ${\mathfrak{U}}(\bt)$ is well defined \hm\, from ${\rm ker}\pi$ into $K_1(A)/H_1(K_0(A), K_1(A)).$

If ${\mathfrak{U}}(\bt)=Q([1_A])=0,$  then, by the proof of Lemma 10.4 and 10.5 of \cite{Lninv},
${\rm Ad}\, v(1)^*\circ \bt$ is strongly asymptotically unitarily equivalent to $\bt.$
But, from the above, ${\rm Ad}\, v(1)^*\circ \bt$ is strongly asymptotically inner. It follows
that $\bt$ is strongly asymptotically inner.  Therefore $\bt$ and $\af$ are strongly outer conjugate
and uniformly approximately conjugate.

So far, we have shown that ${\mathfrak{U}}$ is injective.
Choose any unitary $w\in U(A),$  by \ref{Mext}, there exists an automorphism $\gamma$ with
the Rokhlin property which is strongly asymptotically unitarily equivalent to
${\rm Ad}\, w\circ \af.$ Therefore $\gamma$ is asymptotically inner.
Then ${\mathfrak{U}}(\gamma)=Q([w^*]).$ This implies that ${\mathfrak{U}}$ is an isomorphism.

\end{proof}

\begin{cor}\label{CC1}
Let $A$ be a unital separable simple amenable
\CA\, with $TR(A)\le 1$ which satisfies the UCT. Suppose
that $K_1(A)=H_1(K_0(A), K_1(A)).$ Then two automorphisms
$\af,\, \bt\in \Aut_R(A)$ are strongly outer conjugate and
uniformly approximately conjugate if and only if
$$
\mathfrak{K}(\af)=\mathfrak{K}(\bt)\andeqn
\overline{R}_{\af, \bt}=0.
$$
Moreover,
\beq\label{CC1-1}
\Aut_R(A)/\hspace{-0.05in}\sim_{scc}=\Aut_R(A)/\hspace{-0.05in}\sim_{saucc}
\eneq
and there is a short exact sequence:
\beq\label{CC1-2}
1\to {\rm Hom}(K_1(A),\overline{\rho_A(K_0(A))})/{\cal R}_0
\to \Aut_R(A)/\hspace{-0.05in}\sim_{saucc} \to KKUT_e^{-1}(A,A)^{++}\to 1.
\eneq

\end{cor}

\begin{df}\label{Dtorsion}
{\rm
Let $A$ be a unital \CA\, with $T(A)\not=\emptyset.$
Put
$$
T\rho_A(K_0(A))=\{x\in \overline{\rho_A(K_0(A))}: kx\in \rho_A(K_0(A))
\,\,\,{\rm for\,\,\,some}\,k\in \N\setminus\{0\}\}
$$
and put
$$
TDU(A)=\{ x\in CU(A): kx\in DU(A)
\,\,\,{\rm for\,\,\,some}\,k\in \N\setminus\{0\}\}.
$$

}
\end{df}

\begin{lem}\label{cudiv}
Let $A$ be a unital infinite dimensional separable simple \CA\,
with $TR(A)\le 1.$
Then $\overline{\rho_A(K_0(A))}$ is divisible,
$T\rho_A(K_0(A))$ is divisible and the short exact sequence
\beq\label{cudiv-1}
0\to T\rho_A(K_0(A))\to \overline{\rho_A(K_0(A))}\to \overline{\rho_A(K_0(A))}/T\rho_A(K_0(A)\to 0
\eneq
splitting.
Moreover $\Aff(T(A))/T\rho_A(K_0(A))$ is torsion free.
Furthermore,
$TDU(A)$ is divisible and
\beq\label{cudiv-2}
0\to TDU(A)\to CU(A)\to CU(A)/TDU(A)\to 0
\eneq
is splitting.

\end{lem}

\begin{proof}
The fact that $\overline{\rho_A(K_0(A))}$ is divisible
follows from 3.6 of \cite{Lnexp1} (It also follows from \ref{appdiv}).
Now, if $x\in T\rho_A(K_0(A))$ and $n\ge 1$ be an integer.
Then there is $y\in \overline{\rho_A(K_0(A))}$ such that
$ny=x.$  Since $x\in T\rho_A(K_0(A)),$ there is an integer $k\ge 1$ such that $kx\in \rho_A(K_0(A)).$ Hence $kny\in \rho_A(K_0(A)).$
Therefore $y\in T\rho_A(K_0(A)).$ It follows that the short exact (\ref{cudiv-1}) splitting.
To see $\Aff(T(A))/T\rho_A(K_0(A))$ is torsion free,
it suffices to show that if $x\in \Aff(T(A))$ and $nx\in T\rho_A(K_0(A))$
for some integer $n\not=0,$ then $x\in  \overline{\rho_A(K_0(A))}.$
This, again, follows from 3.6 of \cite{Lnexp1}.
The assertion for $TDU(A)$ and (\ref{cudiv-2}) also follows (using the de la Harpe and Skandalis determant--see also
\cite{Np}).

\end{proof}

\begin{df}\label{Dkkcu}
{\rm
Let $A$ be a unital separable simple amenable \CA\, with $TR(A)\le 1.$
Consider the short exact sequence:
\beq\label{Dkkcu-2}
0\to CU(A)/TDU(A)\to U(A)/TDU(A){\stackrel{\pi_{uc}}{\longrightarrow}} U(A)/CU(A)\to 0.
\eneq
Note that
$$
CU(A)/TDU(A)\cong \overline{\rho_A(K_0(A))}/T\rho_A(K_0(A))\subset {\rm Aff}(T(A))/T\rho_A(K_0(A)).
$$
is divisible. Therefore the above short exact sequence
splits. Denote by $s_{uc}: U(A)/CU(A)\to U(A)/TDU(A)$ a fixed splitting map for (\ref{Dkkcu-2}). Consider the short exact sequence:
\beq\label{Dkkcu-3}
0\to U_0(A)/CU(A)\to U(A)/CU(A){\stackrel{\pi_1}{\longrightarrow}} K_1(A)\to 0.
\eneq
Note also that
\beq\label{Dkkcu-4}
U_0(A)/CU(A)\cong {\rm Aff}(T(A))/\overline{\rho_A(K_0(A))}
\eneq
is a divisible group. Therefore the short exact sequence
in (\ref{Dkkcu-3}) is also splitting. Denote by
$s_1: K_1(A)\to U(A)/CU(A)$ a fixed splitting map.

Let $\phi: A\to A$ be a \hm.  Denote by 
\beq
&&{\tilde \phi_T}: \Aff(T(A))/T\rho_A(K_0(A))\to \Aff(T(A))/T\rho_A(K_0(A))\andeqn\\
&&\phi^{\dag}: U(A)/TDU(A)\to U(A)/TDU(A)
\eneq
the induced \hm s.
Denote by $\pi_{\bf k}: U(A)/TDU(A)\to K_1(A)$
the composition  $\pi_{\bf k}=\pi_1\circ \pi_{uc}$ and denote
by $s_{\bf k}: K_1(A)\to U(A)/TDU(A)$ the composition
$s_{uc}\circ s_1.$
Note that
\beq\label{Dkkcu-5n-1}
\pi_{\bf k}\circ s_{\bf k}={\rm id}_{K_1(A)}\andeqn\\\label{Dkkcu-5+}
\pi_{\bf k}\circ \phi^{\dag}\circ s_{\bf k}=\phi_{*1}.
\eneq
Define
$$
{\widetilde{ \phi^{\dag}}}: K_1(A)/{\rm Tor}(K_1(A))\to U_0(A)/TDU(A)\cong
\Aff(T(A))/T\rho_A(K_0(A))
$$
to be the \hm\, induced by
\beq\label{Dkkcu-5}
{\widetilde{ \phi^{\dag}}}=({\rm id}-s_{\bf k}\circ \pi_{\bf k})\circ \phi^{\dag}\circ s_{\bf k}
\eneq
Note, since $\Aff(T(A))/T\rho_A(K_0(A))$ is torsion free (see \ref{cudiv}), the above
map vanishes on ${\rm Tor}(K_1(A)).$ Therefor the map in (\ref{Dkkcu-5})
defines a \hm\, from $K_1(A)/{\rm Tor}(K_1(A))$ into \linebreak $\Aff(T(A))/T\rho_A(K_0(A))$ or to $U_0(A)/TDU(A).$

Let $\kappa\in KK_e^{-1}(A,A)^{++}$ and $\gamma: T(A)\to T(A)$
be an affine homeomorphism which is compatible with $\kappa.$
Denote by $KKFT_e^{-1}(A,A)^{++}$ the set of triples
$(\kappa, \gamma, \zeta),$ where
$$\zeta: K_1(A)/{\rm Tor}(K_1(A))\to U_0(A)/TDU(A)\cong \Aff(T(A))/T\rho_A(K_0(A))$$ is a \hm.
Let $\gamma^*: \Aff(T(A))\to \Aff(T(A))$ be induced by $\gamma.$
Then $\gamma^*(\rho_A(K_0(A)))=\rho_A(K_0(A))$ and
$\gamma^*({\overline{\rho_A(K_0(A))}})={\overline{\rho_A(K_0(A))}}.$
Therefore $\gamma$ induces an isomorphism
$$
{\widetilde{\gamma^*}}: \Aff(T(A))/T\rho_A(K_0(A))\to \Aff(T(A))/T\rho_A(K_0(A)).
$$
Denote by ${\bar \kappa}: K_1(A)/{\rm Tor}(K_1(A))\to
K_1(A)/{\rm Tor}(K_1(A))$ the isomorphism induced by $\kappa.$
 Define a product on $KKFT_e^{-1}(A,A)^{++}$ as follows:
\beq\label{Dkkcu-10}
(\kappa_1, \gamma_1, \zeta_1)\times (\kappa_2, \gamma_2, \zeta_2)
=(\kappa_2\times \kappa_1, \gamma_2\circ \gamma_1,\,\, \zeta_1\circ {\bar \kappa_2}+{\widetilde{\gamma_1^*}}\circ \zeta_2).
\eneq
$KKFT_e^{-1}(A,A)^{++}$ becomes a group with the identity
$([{\rm id}_A], {\rm id}_{T(A)}, 0)$ and, if
$$(\kappa, \gamma, \zeta)\in KKFT_e^{-1}(A,A)^{++},$$ then
\beq\label{Dkkcu-11}
(\kappa, \gamma, \zeta)^{-1}=(\kappa^{-1}, \gamma^{-1}, -{\widetilde{\gamma^*}}\circ \zeta\circ {\bar \kappa^{-1}} ).
\eneq
If $\af\in \Aut(A),$ define
${\tilde {\mathfrak{K}}}(\af)=([\af], \af_T, {\widetilde{ \af^{\dag}}}).$
Then ${\tilde {\mathfrak{K}}}$ is a \hm\, from $\Aut(A)$ into
$KKFT_e^{-1}(A,A)^{++}.$
To check ${\tilde {\mathfrak{K}}}$ is a \hm, let $\af,\bt\in Aut(A).$
It suffices to show that
\beq\label{Dkkcu-12}
\widetilde{(\af\circ \bt)^{\dag}}={\widetilde{\af^{\dag}}}\circ {\overline{\bt_{*1}}}+{\tilde \af_T}\circ {\tilde \bt^{\dag}},
\eneq
where ${\overline{\bt_{*1}}}: K_1(A)/{\rm Tor}(K_1(A))\to K_1(A)/{\rm Tor}(K_1(A))$ is an isomorphism induced by $\bt_{*1}.$
Since $U_0(A)/TDU(A)$ is torsion free, ${\widetilde \af^{\dag}}\circ {\overline{\bt_{*1}}}$ is induced by
\beq\label{Dkkcu-12n}
({\rm id}-s_{\bf k}\circ \pi_{\bf k})\circ\af^{\dag}\circ s_{\bf k}\circ \bt_{*1}.
\eneq
Note that
\beq\label{Dkkcu-14}
{\rm im}(({\rm id}-s_{\bf k}\circ \pi_{\bf k})\circ \bt^{\dag}\circ s_{\bf k})\subset U_0(A)/TDU(A).
\eneq
Therefore
\beq\label{Dkkcu-15}
({\rm id}-s_{\bf k}\circ \pi_{\bf k})\af^{\dag}({\rm id}-s_{\bf k}\circ \pi_{\bf k})\circ \bt^{\dag}\circ s_{\bf k})=
{\tilde \af_T}\circ ({\rm id}-s_{\bf k}\circ \pi_{\bf k})\circ \bt^{\dag}\circ s_{\bf k}.
\eneq
Then, by (\ref{Dkkcu-5+}) and (\ref{Dkkcu-15}), ${\widetilde{(\af\circ \bt)^{\dag}}}$ is induced by
\beq\label{Dkkvu-16}
&&\hspace{-0.6in}({\rm id}-s_{\bf k}\circ \pi_{\bf k})(\af\circ \bt)^{\dag}\circ s_{\bf k}\\
&=& ({\rm id}-s_{\bf k}\circ \pi_{\bf k})\af^{\dag}\circ s_{\bf k}\circ \pi_{\bf k}\circ \bt^{\dag}\circ s_{\bf k}\\
&&+({\rm id}-s_{\bf k}\circ \pi_{\bf k})\af^{\dag}\circ({\rm id}-s_{\bf k}\circ \pi_{\bf k})\circ \bt^{\dag}\circ s_{\bf k}\\
&=& ({\rm id}-s_{\bf k}\circ \pi_{\bf k})\af^{\dag}\circ s_{\bf k}\circ \bt_{*1}\\
&&+{\tilde \af_T}\circ ({\rm id}-s_{\bf k}\circ \pi_{\bf k})\circ \bt^{\dag}\circ s_{\bf k}.
\eneq
But
$$
{\tilde \af_T}\circ ({\rm id}-s_{\bf k}\circ \pi_{\bf k})\circ \bt^{\dag}\circ s_{\bf k}
$$
induces ${\tilde \af_T}\circ {\widetilde{\bt^{\dag}}}.$ It follows that
$({\rm id}-s_{\bf k}\circ \pi_{\bf k})(\af\circ \bt)^{\dag}\circ s_{\bf k}$ induces
$${\widetilde{\af^{\dag}}}\circ {\overline{\bt_{*1}}}+{\tilde \af_T}\circ {\widetilde{\bf^{\dag}}}.$$
Therefore (\ref{Dkkcu-12}) holds.
Thus, indeed, ${\tilde {\mathfrak{K}}}$ is a \hm.
}
\end{df}

\begin{thm}\label{DUext}
Let $A$  be a unital separable simple amenable \CA\, with $TR(A)\le 1.$ Suppose that $A$ satisfies the UCT,
Suppose that $(\kappa, \gamma, \zeta)\in KKFT_e^{-1}(A,A)^{++}.$
Then there exists an automorphism  $\phi: A\to A$ such that
$$
([\phi],\phi_T, {\widetilde{\phi^{\dag}}})=(\kappa, \gamma, \zeta).
$$

\end{thm}

\begin{proof}
Note, with the assumption on $A,$ \ref{Dkkcu} applies.
Let ${\bar \zeta}: K_1(A)/{\rm Tor}(K_1(A))\to U_0(A)/CU(A)$ be the \hm\, induced by $\zeta.$
In what follows in this proof, we will identify $U_0(A)/TDU(A)$
with $\Aff(T(A))/T\rho_A(K_0(A)),$ and
$CU(A)/TDU(A)$ with ${\overline{\rho_A(K_0(A))}}/T\rho_A(K_0(A))$  whenever it is convenient.
Let ${\overline{\bar \gamma^*}}: \Aff(T(A))/{\overline{\rho_A(K_0(A))}}
\to \Aff(T(A))/{\overline{\rho_A(K_0(A))}}$ be the isomorphism
induced by $\gamma.$
Define $\lambda: U(A)/CU(A)\to U(A)/CU(A)$ by
\beq\label{DUext-1-1}
\lambda(x)=s_1\circ \kappa\circ \pi_1(x)+({\bar \zeta}\circ \pi_f\circ \pi_1(x)+{\overline{\bar \gamma^*}}\circ {\bar \Delta}(x-s_1\circ\pi_1(x))
\eneq
for all $x\in U(A)/CU(A).$
Then $\lambda$ defines a continuous \hm\, from $U(A)/CU(A)$ to
$U(A)/CU(A).$ Since both $\kappa$ and ${\overline{\bar \gamma^*}}$ are isomorphism, it is easy to check that $\lambda$ is an isomorphism.
It follows that
\beq\label{DUext-1}
(\kappa, \gamma, \lambda)\in KKUT_e^{-1}(A,B)^{++}.
\eneq
It follows from Theorem 8.6 of \cite{Lninv} that there exits a unital
\hm\, $\psi: A\to A$ such that
\beq\label{DUext-2}
([\psi], \psi_T, \psi^{\ddag})=(\kappa, \gamma, \lambda).
\eneq
As in 9.10 of \cite{Lninv}, we may assume that $\psi$ is an automorphism.
Define  $\eta_1: K_1(A)/{\rm Tor}(K_1(A))\to CU(A)/TDU(A)$ by
\beq\label{DUext-4}
\eta_1(x)=\zeta(x){\widetilde{\psi^{\dag}}}(-x))\tforal x\in K_1(A)/{\rm Tor}(K_1(A)).
\eneq
This gives a \hm.  By \ref{cudiv}, 
\beq\label{DUext-4+}
0 \to T\rho_A(K_0(A))\to \overline{\rho_A(K_0(A))}\to \overline{\rho_A(K_0(A))}/T\rho_A(K_0(A))\to 0
\eneq
splits. 
One obtains a \hm\, $\eta_2: K_1(A)/{\rm Tor}(K_1(A))\to \overline{\rho_A(K_0(A))}$ such that
$\pi_a\circ \eta_2=\eta_1,$ where
$$
\pi_a: \overline{\rho_A(K_0(A))}\to \overline{\rho_A(K_0(A))}/T\rho_A(K_0(A))\cong CU(A)/TDU(A)
$$
is the quotient map. 
Define $\eta: K_1(A)\to \overline{\rho_A(K_0(A))}$ by
$$
\eta(x)=\eta_2\circ \pi_f(x)\tforal x\in K_1(A),
$$
where $\pi_f: K_1(A)\to K_1(A)/{\rm Tor}(K_1(A))$ is the quotient map. 
It follows from 9.10 of \cite{Lninv} and 9.8 of \cite{Lninv} that there is an automorphism $\af$ such that
\beq\label{DUext-5}
([\af], \af_T, \af^{\ddag})=([{\rm id}_A], {\rm id}_T, {\rm id}_A^{\ddag})
\eneq
and
\beq\label{DUext-6}
\overline{R_{\id, \af}}={\overline{\eta}}\circ (\psi_{*1}^{-1}),
\eneq
where ${\bar \eta}\in {\rm Hom}(K_1(A), {\overline{\rho_A(K_0(A))}})/{\cal R}_0$
is the image of $\eta$ in the quotient.
Define $\phi=\af\circ \psi.$
Then
\beq\label{DUext-7}
([\phi], \phi_T)=(\kappa, \gamma).
\eneq
We will show that
\beq\label{DUext-8}
{\widetilde{\phi^{\dag}}}=\zeta.
\eneq
Then ${\widetilde{\phi^{\dag}}}$ is induced by 
\beq\label{DUext-9}
&&\hspace{-0.4in}({\rm id}-s_{\bf k}\circ \pi_{\bf k})(\af\circ \psi)^{\dag}\circ s_{\bf k}\\
&=& \af^{\dag}(\psi)^{\dag}\circ s_{\bf k}
-s_{\bf k}\circ \pi_{\bf k}\circ (\af^{\dag}\circ \psi^{\dag})\circ s_{\bf k}\\
&=& \af^{\dag}\circ (\psi)^{\dag}\circ s_{\bf k}
-s_{\bf k}\circ \af_{*1}\circ \psi_{*1}\\
&=& \af^{\dag}\circ ({\rm id}-s_{\bf k}\circ \pi_{\bf k})\circ\psi^{\dag}\circ s_{\bf k}+\af^{\dag}\circ s_{\bf k}\circ \pi_{\bf k}\circ\psi^{\dag}\circ s_{\bf k}-s_{\bf k}\circ \psi_{*1}\\
&=& {\tilde \af_T}\circ ({\rm id}-s_{\bf k}\circ \pi_{\bf k})\circ\psi^{\dag}\circ s_{\bf k}+\af^{\dag}\circ s_{\bf k}\circ \psi_{*1}-s_{\bf k}\circ \psi_{*1}\\
&=&({\rm id}-s_{\bf k}\circ \pi_{\bf k})\circ\psi^{\dag}\circ s_{\bf k}+(\af^{\dag}\circ s_{\bf k}\circ \psi_{*1}-s_{\bf k}\circ \psi_{*1})
\eneq
Note that
$$
({\rm id}-s_{\bf k}\circ \pi_{\bf k})\circ\psi^{\dag}\circ s_{\bf k}
$$
induces ${\widetilde{\psi^{\dag}}}.$ 
In other words, ${\widetilde{\phi^{\dag}}}-{\widetilde{\psi^{\dag}}}$ is induced by
$$
\af^{\dag}\circ s_{\bf k}\circ \psi_{*1}-s_{\bf k}\circ \psi_{*1}.
$$

For any $u\in U(A),$ let $x$ be the image of $ s_1\circ \psi_{*1}([u])$ in  $U(A)/DU(A).$
Then, since $\af^{\ddag}={\rm id}_A^{\ddag},$
\beq\label{DUext-10}
\af^{d}(x)x^*\in CU(A)/DU(A),
\eneq
where $\af^d: U(A)/DU(A)\to U(A)/DU(A)$ is the isomorphism induced by $\af.$ 
Let $v\in U(A)$ such that $\tilde{v}=x$ in $U(A)/DU(A).$
If we identify  $CU(A)/DU(A)=\overline{\rho_A(K_0(A))}/\rho_A(K_0(A)),$
then
\beq\label{DUext-10n-1}
\af^{d}(x)x^*=\overline{\Delta(\af(v)v^*)}.
\eneq
However
\beq\label{DUext-11}
\Delta(\af(v)v^*)=R_{{\rm id}_A, \af}(\psi_{*1}([u]))=\eta([u]) \,\,\,\,({\rm mod} \,{\cal R}_0 \,).
\eneq
It follows that
\beq\label{DUext-12}
{\widetilde{\phi^{\dag}}}={\widetilde{\psi^{\dag}}}
+\eta_1=\zeta.
\eneq	
This ends the proof.
\end{proof}

\begin{thm}\label{LT}
Let $A$ be a unital separable simple amenable \CA\, with $TR(A)\le 1$ which satisfies the UCT. Suppose that $K_1(A)/{\rm Tor}(K_1(A))$ is free, or
$$
0\to {\rm ker}\rho_A\to K_0(A)\to \rho_A(K_0(A))\to 0
$$
splits.

 Then two automorphisms $\af, \bt\in \Aut_R(A)$ are
 outer conjugate, i.e.,
there exists a unitary $u\in U(A)$ and $\sigma\in \Aut(A)$ such that
\beq\label{LT-0-1}
\af={\rm Ad}\, u\circ \sigma\circ \bt \circ \sigma^{-1}
\eneq
 if and only if
${\tilde {\mathfrak{K}}}(\af)$ and ${\tilde {\mathfrak{K}}}(\bt)$ are conjugate in $KKFT_e^{-1}(A,A)^{++},$
i.e., there exists \\
$\zeta\in KKFT_e^{-1}(A,A)^{++}$ such that
\beq\label{LT-0-2}
{\tilde{\mathfrak{K}}}(\af)=\zeta\times {\tilde{\mathfrak{K}}}(\bt)\times \zeta^{-1}.
\eneq
Moreover, when (\ref{LT-0-2}) holds, one can require that ${\tilde{\mathfrak{K}}}(\sigma)=\zeta.$

If, in addition, $H_1(K_0(A), K_1(A))=K_1(A),$
then there exists a sequence of unitaries and a sequence $\{\sigma_n\}\subset \Aut(A)$  such that
\beq\label{LT-0}
\af={\rm Ad}\, u_n\circ \sigma_n\circ \bt\circ \sigma_n^{-1}\andeqn
\lim_{n\to\infty}\|u_n-1\|=0.
\eneq
if and only if
there exists $\zeta\in KKFT_e^{-1}(A,A)^{++}$ such that
\beq\label{LT-00}
{\tilde {\mathfrak{K}}}(\af)=\zeta\times {\tilde {\mathfrak{K}}}(\bt)\times \zeta^{-1}.
\eneq

Moreover, when (\ref{LT-0}) holds, one can require that ${\tilde{\mathfrak{K}}}(\sigma_n)=\zeta$
for all $n.$
\end{thm}

\begin{proof}
Let $\af$ and $\bt$ be two automorphisms with the Rokhlin property such that
$$
\af={\rm Ad}\, u\circ \sigma\circ\bt \circ  \sigma^{-1}
$$
for some unitary $u\in U(A)$ and some $\sigma\in \Aut(A).$
Note that
$$
({\rm Ad}\, u, ({\rm Ad}\, u)_T, ({\rm Ad}\, u)^d)=
([{\rm id}_A], {\rm id}_T, {\rm id}_{U(A)/DU(A)}),
$$
where $({\rm Ad}\, u)^d: U(A)/DU(A)\to U(A)/DU(A)$ is the isomorphism
induced by ${\rm Ad}\, u.$
It follows that
\beq\label{LT-3}
{\tilde {\mathfrak{K}}}(\af)={\tilde {\mathfrak{K}}}(\sigma)\times
{\tilde {\mathfrak{K}}}(\bt)\times {\tilde {\mathfrak{K}}}(\sigma)^{-1}.
\eneq

Now suppose that there exists $\zeta\in KKFT_e^{-1}(A,A)^{++}$
such that
\beq\label{LT-4}
{\tilde {\mathfrak{K}}}(\af)=
\zeta\times {\tilde {\mathfrak{K}}}(\bt)\times \zeta^{-1}.
\eneq
Write $\zeta=(\zeta_K, \zeta_T, \lambda),$
where $\zeta_K\in KK_e^{-1}(A,A)^{++},$
$\zeta_T: T(A)\to T(A)$ is an affine homeomorphism and
$\lambda: K_1(A)/{\rm Tor}(K_1(A))\to \overline{\rho_A(K_0(A))}/\rho_A(K_0(A))$ is a \hm.
It follows from \ref{DUext} that there exists $\sigma'\in \Aut(A)$
such that
\beq\label{LT-5-1}
([\sigma'], \sigma'_T, {\tilde (\sigma')^{\dag}})=
(\zeta_K, \zeta_T, \lambda).
\eneq
Let $\bt_1=\sigma'\circ \bt \circ(\sigma')^{-1}.$
Then
\beq\label{LT-6}
([\af], \af_T, {\widetilde{ \af^{\dag}}})=([\bt_1], (\bt_1)_T, {\tilde\bt_1^{\dag}}).
\eneq
In particular,
\beq\label{LT-7}
\af^{\ddag}=\bt_1^{\ddag}.
\eneq
We will now show that
\beq\label{LT-8}
\overline{R}_{\af, \bt_1}=0.
\eneq
First we note, since $[\af]=[\bt_1],$ there is \hm\, $\Theta: \underline{K}(A)\to \underline{K}(M_{\af,\bt_1})$ such that
\beq\label{LT-9}
[\pi_0]\circ \Theta=[{\rm id}_A],
\eneq
where $\pi_0: M_{\af, \bt_1}\to A$ is the point-evaluation of the mapping torus at $t=0.$
Put $\theta=\Theta|_{K_1(A)}.$ 
Let $u\in U(A)$ be a unitary. Let $z\in U(M_{\af, \bt_1})$ be a unitary such that
$z(0)=\af(u)$ and $z(1)=\bt_1(u).$ Moreover, we may assume that $z$ is piece-wise smooth.
Define $z_1(t)=\bt_1(u)^*z(t)$ for $t\in [0,1].$ Then $z_1$ is a continuous and piece-wise smooth with
$z_1(0)=\af(u)\bt_1(u)^*$ and $z_1(1)=1_A.$ By (\ref{LT-6}),
\beq\label{LT-10}
{1\over{2\pi i}}\int_0^1\tau( {dz_1(t)\over{dt}}z_1(t)^* )dt\in T\rho_A(K_0(A)),
\eneq
 where $\tau\in T(A).$  
  It follows that
\beq\label{LT-10+}
R_{\af, \bt_1}([z])\in T\rho_A(K_0(A)).
\eneq
On the other hand there exist two projections $p,\, q\in M_l(A)$ (for some integer $l\ge 1$) such that
\beq\label{LT-11}
\theta([u])=[(z+1_{M_{(l-1)}})v],
\eneq
where $v(t)=(e^{i2\pi t}p +(1_{M_l(A)}-p))(e^{-2\pi t}q+(1_{M_l(A)}-q))$ for $t\in [0,1].$
Then
\beq\label{LT-12}
R_{\af, \bt_1}([(z+1_{M_{(l-1)}})v])=R_{\af, \bt_1}([z])+R_{\af, \bt_1}([v])\in T\rho_A(K_0(A)).
\eneq
It follows that
\beq\label{LT-13}
R_{\af, \bt_1}\circ \theta\in {\rm Hom}(K_1(A), T\rho_A(K_0(A))).
\eneq
Since $\Aff(T(A))$ is torsion free so is its subgroup $T\rho_A(K_0(A)),$ therefore $R_{\af, \bt_1}\circ \theta$ factors through \linebreak  $K_1(A)/{\rm Tor}(K_1(A)),$
i.e., there is $R: K_1(A)/{\rm Tor}(K_1(A))\to T\rho_A(K_0(A))$ such that
\beq\label{LT-14}
R\circ \pi_f=R_{\af, \bt_1}\circ \theta.
\eneq
(Recall that $\pi_f: K_1(A)\to K_1(A)/{\rm Tor}(K_1(A))$ is the quotient map.)
 However, since \linebreak $T\rho_A(K_0(A))/\rho_A(K_0(A))$
 is torsion and $K_1(A)/{\rm Tor}(K_1(A))$ is torsion free,   one actually  has 
 that
 \beq\label{LT-10n}
 R_{\af, \bt_1}\circ \theta\in {\rm Hom}(K_1(A), \rho_A(K_0(A))).
\eneq
If $K_1(A)/{\rm Tor}(K_1(A))$ is free, or 
 $$
 0\to {\rm ker}\rho_A\to K_0(A)\to \rho_A(K_0(A))\to 0
 $$
 splits, 
 there is $\dt: K_1(A)/{\rm Tor}(K_1(A))\to K_0(A)$
such that
\beq\label{LT-15}
\rho_A\circ \dt=R.
\eneq
Put $\theta_1=\theta|_{K_1(A)}-\iota_{*1}\circ\dt\circ \pi_f,$ 
where $\iota: SA\to M_{\af, \bt_1}$ is the embedding (see \ref{Drotation}). 
By the universal coefficient theorem (as well as Dadarlat and Loring's  result in \cite{DL}), 
there is $\Theta_0'\in {\rm Hom}_{\Lambda}(\underline{K}(A),\underline{K}(SA))$ such that
$\Theta_0'|_{K_0(A)}=0$ and $\Theta_0'|_{K_1(A)}=\dt\circ \pi_f.$
Define 
$$
\Theta_1=\Theta-[\iota]\circ \Theta_0.
$$
Then 
\beq\label{LT-16}
[\pi_0]\circ \Theta_1=[{\rm id}]\andeqn \Theta_1|_{K_1(A)}=\theta_1.
\eneq
Note that, since $[\af]=[\bt],$
\beq\label{LT-16+1}
R_{\af, \bt_1}|_{K_0(A)}=\rho_A.
\eneq
It follows (from (\ref{LT-14}) and (\ref{LT-15})) that
\beq\label{LT-16+}
R_{\af,\bt_1}\circ \theta_1=0.
\eneq
In other words,
\beq\label{LT-20}
\overline{R_{\af, \bt_1}}=0.
\eneq
It follows from (\ref{LT-20}), (\ref{LT-6}), (\ref{LT-7}) and Theorem 7.2 of \cite{Lninv} that
$\af$ and $\bt_1$ are asymptotically unitarily equivalent.

It follows from Theorem \ref{MT1} that there exits a unitary
$u\in U(A)$ and a strongly asymptotically inner  automorphism
\beq\label{LT-21}
\af&=&{\rm Ad}\, u\circ \sigma_1\circ \bt_1\circ \sigma_1^{-1}\\
&=& {\rm Ad}\, u\circ \sigma\circ \bt\circ \sigma^{-1},
\eneq
where $\sigma=\sigma_1\circ \sigma'.$
Note that
\beq\label{LT-22}
{\tilde {\mathfrak{K}}}(\sigma)={\tilde {\mathfrak{K}}}(\sigma_1)\times {\tilde {\mathfrak{K}}}(\sigma')=\zeta.
\eneq
In the case that $K_1(A)=H_1(K_o(A), K_1(A)),$ by \ref{MT1},
there exists a sequence of unitaries $\{u_n\}\subset U_0(A)$ and
a sequence of strongly asymptotically inner automorphisms
$\sigma_n''\in \Aut(A)$ such that
\beq\label{LT-23}
\af={\rm Ad}\, u_n\circ \sigma_n''\circ \bt_1\circ (\sigma_n'')^{-1}
\andeqn \lim_{n\to\infty}\|u_n-1\|=0.
\eneq
Put $\sigma_n=\sigma_n''\circ \sigma'.$ Theorem follows.

\end{proof}

\bibliographystyle{plain}

\end{document}